\documentclass{report}
\usepackage{dsfont}
\usepackage[final]{graphicx}
\usepackage{pstricks} 
\usepackage{pst-all}
\usepackage{colortbl}
\usepackage{algorithmicx}
\usepackage{stackengine}
\usepackage[T1]{fontenc}
\usepackage{amsmath}
\usepackage{mathtools}
\usepackage{mathptmx}
\usepackage{amsthm}
\usepackage{amssymb}
\usepackage{mathrsfs}
\usepackage{algorithm}
\usepackage{algpseudocode}
\newtheorem{theorem}{Theorem}[section]

\newtheorem{lemma}[theorem]{Lemma}
\newtheorem{definition}{Definition}
\usepackage{listings}
\usepackage{graphics}
\usepackage{float} 
\usepackage{xcolor}
\definecolor{RoyalBlue}{RGB}{65,105,225} 
\definecolor{Blue}{RGB}{0,0,255} 
\definecolor{ForestGreen}{RGB}{34,139,34} 
\definecolor{YellowGreen}{RGB}{154,205,50} 
\DeclareGraphicsExtensions{.bmp,.png,.pdf,.jpg,eps.}
\usepackage{enumerate}
\usepackage{amssymb,amsmath,amsthm,bbm,keyval,color,psfrag,multirow,lscape}
\usepackage[bookmarks=true,bookmarksopen=true,breaklinks=true,letterpaper=true,pdftitle={DCFFT},plainpages=false,pdfauthor={StudentRUM},colorlinks=true,hypertexnames=false,citecolor=blue,linkcolor=blue,file
color=blue]{hyperref}
\usepackage[square,sort,comma,numbers]{natbib}
\usepackage{rotating}   
\usepackage{subfigure}  
\usepackage{multicol}
\usepackage[font=footnotesize,labelfont=bf]{caption}
\usepackage{subfigure}

\theoremstyle{plain}

\hypersetup{urlcolor=blue}			 				

\newenvironment{myabstract}{%
	\begin{center}%
		\bfseries Abstract\end{center}}%
{\par\medskip}
\title{A UNIVERSAL ROBUST BOUND FOR THE INTRINSIC BAYES FACTOR.}
\author{Richard Clare}
\date{\today}
\begin{document}
\maketitle 
\begin{myabstract}In this work, we undertake a comprehensive reformulation, modification, and extension of Smith \& Spiegelhalter's \cite{smith1980bayesfactors} and \cite{spiegelhalter1982bayesfactors} Bayes Factor work within the evolving subject of Objective Bayes Factors. Our primary focus centers on defining and computing empirical and theoretical bounds for the Intrinsic Bayes Factor (IBF) across various models, including normal, exponential, Poisson, geometric, linear, and ANOVA. We show that our new bounds are useful, feasible, and change with the amount of information. We also propose a methodology to construct the least favorable (for the null model) intrinsic priors that result in the lower and upper bounds of the Intrinsic Bayes Factors under certain conditions.  Notably, our lower bounds exhibit superior performance compared to the well-known $-ep\log(p)$ bound proposed by Sellke et al. (2001) \cite{sellke2001calibration} based on p-values.
\end{myabstract}
\tableofcontents

\chapter{Introduction}  
\hspace{10pt}The misuse of classical hypothesis testing methods, specifically p-values, has been the subject of substantial critique in the statistical literature (Wasserstein and Lazar, 2016 \cite{wasserstein2016asa}; Wasserstein et al., 2019 \cite{wasserstein2019moving}). Despite these warnings, the practice of p-value null hypothesis significance testing (NHST) continues to be a predominant technique for model selection. However, alternatives do exist. Among them, the Bayesian testing methodology developed by Sir Harold Jeffreys (1961) \cite{jeffreys1961theory} stands as one of the most widely accepted and used. Jeffreys's methodology, premised on the concept that the null hypothesis carries a positive point mass, quantifies the support for or against a model based on the data. Additionally, it offers practical applications through derived approximations for complex or analytically infeasible solutions. A significant contribution of Jeffreys's framework is the Bayes factor, a measure of support for one model over another, irrespective of the absolute correctness of these models. Its effectiveness spans various inference desiderata, including interpretability, adherence to Occam's razor, and consistency in large samples. Kass and Raftery (1995) \cite{kass1995bayes} offer an extensive introduction to Bayes factors.
Our research primarily concentrates on the pivotal role of training samples within statistical methodologies. These methodologies include classification and discrimination, robustness, model selection, and cross-validation. Recent advancements in Bayesian model selection, namely the intrinsic Bayes factor by Berger and Pericchi (1996a) \cite{berger1996intrinsic} and the expected posterior prior by Perez and Berger (2002) \cite{perez2002expectedposterior}, have employed training samples to transition improper objective priors into more suitable distributions for model selection. 

A challenge often faced is the indeterminacy of classical Bayes Factors, due to the dependence on an undefined ratio of constants derived from the improper priors. To circumnavigate this, Jeffrey suggested conventional proper priors for the extra-parameters under the larger hypothesis, albeit improper for common parameters. Various techniques have been proposed to address this issue, including those by Smith and Spiegelhalter (1980) \cite{smith1980bayesfactors} and Spiegelhalter and Smith (1982) \cite{spiegelhalter1982bayesfactors}. Though approximate, these techniques facilitate sensible scaling for Bayes Factors and bear a direct relationship with more recent approaches like the Intrinsic Bayes Factor, the Intrinsic Priors, and EP-priors. To elucidate these connections, we will re-conceptualize Smith and Spiegelhalter's approach and the associated Bayes Factor, hereby referred to as SSBF's ($B^{SS}$), drawing parallels to other methodologies.

Chapter three specifically focuses on the quasi-Bayes factor, which emerges as a universal lower bound for the Intrinsic Bayes factor. In Chapter four we introduce the concept of SP Priors (Superior Posterior priors) that either generate the Intrinsic Bayes Factors bound under specific assumptions or serve as an asymptotic approximation of the bounds. These priors are the least favorable priors for the null hypothesis and through chapter three to chapter eight we performed numerical simulations and verified the consistency of the SP, the empirical and the theoretical Intrinsic Bayes Factor bounds across various essential distributions such as the normal, exponential, Poisson, geometric, and negative binomial. Furthermore, we calculated the bounds for two nested normal-linear models and ANOVA. We discussed model selection problems across different distributions and hypothesis tests such as nested and separated hypotheses, and for those cases we provided closed forms for the Intrinsic Bayes factors Bounds, the SP Priors, and the SP Bayes Factors. 

\chapter{LITERATURE REVIEW}  
\section{Bayes factors}
\hspace{10pt}Traditional point estimation methods, such as the maximum likelihood estimator (MLE) or the method of moments (MoM), typically involve the utilization of specific estimates for each model parameter based on the available data. In contrast, Bayesian inference takes a different approach by assigning a probability distribution to each model parameter.

Let $\textbf{y} = (y_1,...,y_n)$ represent a set of $n$ random samples obtained from a random variable $Y$ with a density $f_i(y|\theta_i)$, where $\theta_i$ denotes a vector of parameters associated with model $M_i$. Employing the Bayes' Theorem, the posterior distribution for model $M_i$ can be expressed as:
\[
\pi(\boldsymbol\theta_i|\textbf{y}) = \displaystyle\frac{f_i(y|\theta))\pi_i(\boldsymbol\theta_i)}{m_i(\textbf{y})} = \displaystyle\frac{f_i(\textbf{y}|\theta_i)\pi(\boldsymbol\theta_i)}{\displaystyle\int f_i(\textbf{y}|\theta_i)\pi(\boldsymbol\theta_i)d\boldsymbol\theta_i} 
\]
where $m_i(\textbf{y})$ is called the marginal, evidence, or global likelihood, and $\pi_i(\boldsymbol\theta)$ is the prior distribution of the parameters for $M_i$. 
Suppose that we are comparing two models for the data \textbf{y}
\[
M_k: \textbf{Y} \mbox{ has a density } f_k(\textbf{y}|\boldsymbol\theta_k) \hspace{5pt} k = i,j
\]
where $\boldsymbol\theta_k$ are unknown model parameters associated with model $M_k$. Suppose we have available prior distributions $\pi_k(\boldsymbol\theta_k)$ for the unknown parameters. Define the marginal or predictive densities of \textbf{y} as
\[
m_i(\textbf{y}) = \int f_i(\textbf{y}|\boldsymbol\theta_i)\pi_i(\boldsymbol\theta_i)d\boldsymbol\theta_i.
\]
The \emph{Bayes factor} of $M_i$ to $M_j$ is given by 
\[
B_{ij} = \displaystyle\frac{m_i(\textbf{y})}{m_j(\textbf{y})} = \displaystyle\frac{\displaystyle\int f_i(\textbf{y}|\boldsymbol\theta_i)\pi_i(\boldsymbol\theta_i)d\boldsymbol\theta_i}{\displaystyle\int f_j(\textbf{y}|\boldsymbol\theta_j)\pi_j(\boldsymbol\theta_j)d\boldsymbol\theta_j}.
\]
The \emph{Bayes factor} is often interpreted as the evidence provided by the data for the model $M_i$ vs. the alternative model $M_j$, but the \emph{Bayes factor} depends also on the \emph{priors}. Alternatively, $B_{ij}$ can be interpreted as the weighted likelihood ratio of $M_i$ to $M_j$.
with the priors being the "weighting functions."\\
\section{Prior and Posterior Model Probabilities }
\hspace{10pt}Suppose that prior probabilities $\pi(M_i), i = 1,...,n$, of the $n$ models, are available, then we can obtain for each model $M_i$ the posterior probability from the Bayes factor given the data $\textbf{y}$ by calculating
$$ P(M_i|\textbf{y}) = \frac{\pi(M_i)m_i(\textbf{y})}{\displaystyle\sum_{k=1}^{n}\pi(M_k)m_k(\textbf{y})}$$ 
This equation can also be written as
$$ P(M_i|\textbf{y}) =\Big[\sum_{k=1}^{n}\frac{\pi(M_k)}{\pi(M_i)}B_{ki}(\textbf{y})\Big]^{-1}$$
where $O_{ki} =\pi(M_k)/\pi(M_i)$ is known as the prior odds of the model $M_k$ over the model $M_i$, and the posterior probabilities odds are given by
$$ \frac{P(M_i|\textbf{y})}{P(M_j|\textbf{y})} = \frac{\pi(M_i)m_i(\textbf{y})}{\pi(M_j)m_j(\textbf{y})} = O_{ij}B_{ij}$$ 
Hence, the posterior odds are given the product of the prior odds and the Bayes factor. A common unbiased choice of the prior model probabilities is given by $\pi(M_i) = 1 / n$ where each model has the same prior probability. Note that using those uniform priors we can obtain the re-normalized marginal distributions from the posterior model probabilities as
$$\overline{m}_i(\textbf{y}) = \frac{m_i(\textbf{y})}{\sum_{k=1}^{n}m_k(\textbf{y})} $$
and also the Bayes Factor can be in terms of the re-normalized marginals as  
$$ B_{ij}(\textbf{y}) = \frac{\overline{m}_{i}(\textbf{y})}{\overline{m}_j(\textbf{y})}$$ 
\section{Jeffreys's objective priors}
\hspace{10pt}Consider the Fisher information matrix when there are n parameters, so that $\boldsymbol\theta$ is a $n\times1$ vector $\boldsymbol\theta = \begin{bmatrix} \theta_{1}, \theta_{2}, \dots , \theta_{n}\end{bmatrix}^{\mathrm T}$, then the Fisher information takes the form of an $n\times n$ matrix with typical element
\[
{I_{ij} \left(\theta \right) = \operatorname{E} \left[\left. \left(\frac{\partial}{\partial\theta_i} \log f(\textbf{y}|\boldsymbol\theta)\right) \left(\frac{\partial}{\partial\theta_j} \log f(\textbf{y}|\boldsymbol\theta)\right) \right|\boldsymbol\theta\right]}. 
\]

Jeffreys prior is a non-informative (objective) prior distribution for a parameter space; it is proportional to the square root of the determinant of the Fisher information matrix
\[
\pi^{N}\left(\boldsymbol\theta\right) \propto \sqrt{\det I\left(\boldsymbol\theta\right)}\,
\]
and this prior is invariant under reparameterization of the parameter vector $\boldsymbol\theta$ and this property makes it of special interest for use with location and scale parameters.

\section{Objective Bayes Factors}
\hspace{10pt}The unscaled Bayes factors are obtained using the non-informative objective priors $\pi_i^N(\theta_i)$ and $\pi_j^{N}(\theta_j)$ for the models $M_i$ and $M_j$ respectively and can be calculated as 
$$ B^{N}_{ij}(\textbf{y}) = \frac{\displaystyle\int f_i(\textbf{y}|\theta_i)\pi_i^{N}(\theta_i)d\theta_i}{\displaystyle\int f_j(\textbf{y}|\theta_j)\pi_j^{N}(\theta_j)d\theta_j}$$

Non-informative priors can be used for one-sided hypothesis testing but unfortunately, when we are using improper priors for point null hypotheses the Bayes Factor depends on an undetermined ratio of constants $C_i/C_j$, which comes from the improper priors. Some techniques are widely used to overcome this issue which can approximate the desired ratio of constants and those will be studied and extended in the next sections. \\
\section{Minimal training samples}
\hspace{10pt}There are different approaches to calculate Bayes Factors based on the concept of a proper 'minimal training sample' which is a subset of the entire data $\textbf{y}$.
\\
$\textbf{Definition}$: A training sample $\textbf{y}(\ell) \subset \textbf{y}$, is called proper if $0 < m(\textbf{y}(\ell)) < \infty$, and minimal if it is proper and no subset is proper.\\
\subsection{A formal definition for the set of all minimal training samples}
The formal definition and formula for the set of all possible minimal training samples can be obtained using the definition above and extending the concept. Let $y(\ell_i)=y(i_1,i_2,...,i_{k}) = \{y_{i_1},y_{i_2},...,y_{i_{k}}\}$ be any minimal training sample with $k \leq n$. The set $\{y(\ell_1),...,y(\ell_L)\}$ of all possible minimal training samples (non-repeated subsets of size $k$) can be defined as the set
$$D_{n,k} = \Bigg\{ \Big\{\big\{\cdot \cdot \cdot\{y(i_1,i_2,...,i_{k})\}_{i_{k}= i_{(k-1)}}^{i_n}\cdot\cdot\cdot \big\}^{i_n}_{i_2 =i_1}\Big\}_{i_1 = 1}^{i_n} : 1 \leq i_1 < i_2 < i_3< \dots < i_{k} \leq n\Bigg\}$$
Hence,
$$|D_{n,k}| = {n\choose k} = L \mbox{ \hspace{10pt} and \hspace{10pt} } D_{\infty,k} = \lim_{n\to \infty}D_{n,k}$$
Therefore,
$$ |D_{\infty,k}| = |\lim_{n\to\infty}D_{n,k}| = \lim_{n\to\infty} |D_{n,k}| =  \lim_{n\to\infty} \frac{(n)(n-1)(n-(k+1))}{k!} = \infty$$
For ease of notation, we will denote the set $D_{\infty,k}$ as $D$ throughout this thesis.

The determination of the set $D_{n,k}$, encompassing all minimal training samples of size $k$ from a sample of size $n$, will be facilitated by an original, exhaustive, albeit sub-optimal algorithm. This algorithm yields the computation of Bayes factors involving these terms?such as the Arithmetic Intrinsic Bayes Factor, Median Intrinsic Bayes Factor, EP-Bayes factors, and SS-Bayes factors?which will be expounded upon in subsequent chapters.

Similarly, considering $L = $ ${n}\choose{k}$,, we can derive the following computationally intuitive formula applicable to the Arithmetic Intrinsic Bayes Factor (AIBF)
$$\boxed{\frac{1}{L}\displaystyle\sum_{i_1=1}^{i_n}\sum_{i_2 = i_1}^{i_n}\cdot\cdot\cdot\sum_{i_{k}=i_{(k - 1)}}^{i_n}\mathds{1}(1 \leq i_1 < i_2 < \dots < i_{k}\leq n)B_{01}(y(i_1,...,i_{k})) = \frac{1}{L}\sum_{d \in D_{n,k}}B_{01}(d). }
$$
In the next section, we will present the general definition algorithm that can receive a vector of observations, and will generate the set of all possible minimal training samples.\\
\subsection{An algorithm to generate the set of all minimal training samples}
\hspace{10pt} We present here an original algorithm devised for the purpose of generating the set of all minimal training samples. It is important to note that the algorithm proposed in this study may not be suitable for processing very large sample sizes due to its computational complexity, which amounts to approximately $O(n^k)$ operations. For future investigations, it is recommended to explore more efficient algorithms for generating the matrix $\textbf{M}$ based on the algorithm discussed herein.

Alternatively, a generalized version of this algorithm can be constructed using a recurrence relation instead of nested for loops. Such an approach would enable the development of a versatile algorithm applicable to any given values of $n$ and $k$. In cases where the sample sizes are substantial, it is advisable to employ randomization techniques to generate a representative subset of minimal training samples, which can yield satisfactory approximations of the Intrinsic Bayes Factors (IBFs). 
\begin{algorithm}
	\caption{An algorithm to produce the set of all possible minimal training samples}\label{alg:cap}
	\begin{algorithmic}
		\Require{$n,k \in \mathbb{N}, n\geq k$}
		\Ensure{A matrix with k columns and ${n}\choose{k}$ rows with the indexes of all of the possible minimal training samples of size k from a sample of size n.}
		\Function{MTS}{$(n,k)$}
		\\ Initialize an empty matrix M
		\For{$i_1 \mbox { in } 1:n$} {
			\For{$i_2 \mbox{ in } i_1:n$}{
				\\
				\hspace{50pt}\vdots
				\For{$i_k \mbox { in } i_{k-1}:n$}{
					\If{$1 \leq i_1 < i_2 < ... < i_k \leq n$}
					\State Add the row with values $(i_1,...,i_k)$ to the bottom of the matrix M. 
					\EndIf
				}
				\EndFor
				\\
				\hspace{50pt}\vdots
			}
			\EndFor
		}
		\EndFor
		\\ Return M
		\EndFunction
	\end{algorithmic}
\end{algorithm}
\newpage
\section{The Intrinsic Bayes Factor}
\hspace{10pt}Suppose that non-informative (usually improper) priors $\pi_i^{N}(\boldsymbol\theta_i), i = 1,...,n$ are available for the models $M_1,...M_n$ respectively. The general recommendation is that we choose those priors as \textquotedbl reference priors" (Berger and Bernardo 1992 \cite{berger1992ordered}), but there are many other choices and approximations with good results and advantages over the reference priors in terms of computational and mathematical complexity.  \\
\indent The \emph{Intrinsic Bayes Factor} (IBF) is calculated using a minimal training sample $\textbf{y}(\ell)$ which converts the improper prior $\pi_i^N(\boldsymbol{\theta}_i)$ to proper posteriors $\pi_i^N(\boldsymbol{\theta}_i|\textbf{y}(\ell))$. The Bayes factor is calculated using this proper posterior as the prior distribution of the remaining of the data $\textbf{y}(-\ell) = \textbf{y}\setminus \textbf{y}(\ell)$ as follows

$$
\boxed{
	B_{ij}^{I} = B_{ij}(\ell) = \frac{\displaystyle\int f_i(\textbf{y}(-\ell)|\boldsymbol{\theta}_i,\textbf{y}(\ell))\pi_i(\boldsymbol{\theta}_i|\textbf{y}(\ell))d\boldsymbol{\theta}_i}{\displaystyle\int f_j(\textbf{y}(-\ell)|\boldsymbol{\theta}_j,\textbf{y}(\ell))\pi_j(\boldsymbol{\theta}_i|\textbf{y}(\ell))d\boldsymbol{\theta}_j}
}
$$

where $B_{ij}(\ell) = B_{ij}(\textbf{y}(-\ell)|\textbf{y}(\ell))$. There are $L = {n \choose k}$ different minimal training samples for $n$ i.i.d samples, where $k$ is the number of parameters of the density under consideration. Denote by $B^N_{ij}(\mathbf{y})$, the Bayes Factor with the whole sample $\mathbf{y}$ and un-scaled objective (Non-Informative) priors $\pi^N_k(\boldsymbol{\theta}_k), k=i,j$, and,
$B_{ij}(\mathbf{y(-\ell)}|\mathbf{y(\ell)})$, the "trained" Bayes Factor after the training sample has been employed to get a proper scaling of the Bayes Factor. A fundamental Lemma with Bayes Factors is the following (Berger and Pericchi(1996) \cite{berger1996intrinsic}\\
\begin{lemma} Let $\textbf{y} = \{y_1,...,y_n\}$ be independent and identically distributed random samples, and assume that $\textbf{y}(\ell) \in D$ with $\textbf{y}(-\ell)\cup \textbf{y}(\ell) = \textbf{y}$, then 
	\begin{equation}
		\mbox{      }
		B_{ij}(\ell)=B^N_{ij}(\mathbf{y}) B^N_{ji}(\mathbf{y(\ell)}), \forall \ell.
		\label{LemmaTrainedBF}
	\end{equation}
\end{lemma}
\begin{proof} The Bayes Factor of the remaining samples given the training samples can be written as 
	$$ B_{ij}(\mathbf{y(-\ell)}|\mathbf{y(\ell)}) = \frac{m_i(\mathbf{y(-\ell)}|\mathbf{y(\ell)})}{m_j(\mathbf{y(-\ell)}|\mathbf{y(\ell)})} =  
	\frac{\displaystyle\int f_i(\mathbf{y(-\ell)}|\theta_i)\pi(\theta_i|\mathbf{y(\ell)})d\theta_i }
	{\displaystyle\int f_i(\mathbf{y(-\ell)}|\theta_j)\pi(\theta_j|\mathbf{y(\ell)})d\theta_j }
	$$ 
	The proper posterior distribution of the parameters given the MTS can be written as
	$$\pi(\theta_k|\mathbf{y(\ell)}) = \frac{f_k(\theta_k|\mathbf{y(\ell)})\pi_k^{N}(\theta_k)}{m_k^{N}(\mathbf{y(\ell)})}, \mbox{ for } k = i,j$$
	Hence, 
	$$\frac{\displaystyle\int f_i(\mathbf{y(-\ell)}|\theta_i)\pi(\theta_i|\mathbf{y(\ell)})d\theta_i }
	{\displaystyle\int f_j(\mathbf{y(-\ell)}|\theta_j)\pi(\theta_j|\mathbf{y(\ell)})d\theta_j }
	= \frac{\displaystyle\int f_i(\mathbf{y(-\ell)}|\theta_i)\frac{f_i(\theta_i|\mathbf{y(\ell)})\pi_i^{N}(\theta_i)}{m_i^{N}(\mathbf{y(\ell)})}d\theta_i }
	{\displaystyle\int f_j(\mathbf{y(-\ell)}|\theta_j)\frac{f_j(\theta_j|\mathbf{y(\ell)})\pi_j^{N}(\theta_j)}{m_j^{N}(\mathbf{y(\ell)})}d\theta_j }$$
	\begin{equation}
		= \frac{m_j^N(\mathbf{y(\ell)})}{m_i^N(\mathbf{y(\ell)})}\frac{\displaystyle\int f_i(\mathbf{y(-\ell)}|\theta_i)f_i(\theta_i|\mathbf{y(\ell)})\pi_i^{N}(\theta_i)d\theta_i }
		{\displaystyle\int f_j(\mathbf{y(-\ell)}|\theta_j)f_j(\theta_j|\mathbf{y(\ell)})\pi_j^{N}(\theta_j)d\theta_j }
	\end{equation}
	Since $f_k(\mathbf{y}(\ell)|\boldsymbol{\theta}_k)f_k(\mathbf{y}(-\ell)|\boldsymbol{\theta}_k) = f_k(\mathbf{y}|\boldsymbol{\theta}_k) \mbox{ for } k = i,j$, the ratio of the marginals is
	$$ \frac{m_j^N(\mathbf{y(\ell)})}{m_i^N(\mathbf{y(\ell)})} = B_{ji}^{N}(\mathbf{y(\ell)})$$ 
	Hence, equation (2) becomes
	$$B_{ji}^{N}(\mathbf{y(\ell)})\frac{\displaystyle\int f_i(\mathbf{y}|\theta_i)\pi_i^{N}(\theta_i)d\theta_i }
	{\displaystyle\int f_j(\mathbf{y}|\theta_j)\pi_j^{N}(\theta_j)d\theta_j } = B_{ji}^{N}(\mathbf{y(\ell)})B_{ij}^{N}(\mathbf{y}).
	$$
\end{proof}
By employing the methodology elucidated in the aforementioned lemmas, we have successfully mitigated the reliance of $B_{ij}$ on the scales of the priors, thereby resolving the issue of undetermined constants associated with the Bayes factor. However, a new dependence arises, contingent upon the arbitrary selection of the minimal training sample denoted as $\textbf{y}(\ell)$. To address this dependency and enhance stability, we propose to perform an "average" of $B_{ij}(l)$ over all feasible training samples $\textbf{y}(\ell)$, where $\ell = 1,...,L$.\\
\section{The Expected Posterior Bayes Factors}
\hspace{10pt}Suppose that non-informative (typically improper) priors $\pi_i^{N}(\boldsymbol\theta_i)$, where $i = 1, \ldots, n$, associated with the models $M_1, \ldots, M_n$, respectively are available. The concept of Expected Posterior (EP) priors, introduced by Perez and Berger in 2002 \cite{perez2002expectedposterior}, allows for the calculation of proper posteriors $\pi_i^N(\boldsymbol{\theta}_i|\textbf{y}(\ell))$ using a minimal training sample $\textbf{y}(\ell)$. The EP prior for model $M_i$ is then defined as:

$$ \boxed { \pi_{i}^{EP}(\theta_i) = \frac{1}{L}\displaystyle\sum_{\ell=1}^{L}\pi_i^N(\boldsymbol{\theta}_i|\textbf{y}(\ell)) } $$

The \emph{EP Bayes factor}, which enables the comparison between models $M_i$ and $M_j$, is computed as follows:
$$
\boxed{ 
B_{ij}^{EP}(\textbf{y}) = \frac{\displaystyle\int f_i(\textbf{y}|\boldsymbol{\theta}_i)\pi_i^{EP}(\boldsymbol{\theta}_i)d\boldsymbol{\theta}_i}{\displaystyle\int f_j(\textbf{y}|\boldsymbol{\theta}_j)\pi^{EP}_j(\boldsymbol{\theta}_j)d\boldsymbol{\theta}_j}
}
$$
In the above equations, $f_i(\textbf{y}|\boldsymbol{\theta}_i)$ and $f_j(\textbf{y}|\boldsymbol{\theta}_j)$ represent the likelihood functions associated with models $M_i$ and $M_j$, respectively. The EP prior, derived from improper priors using the minimal training sample $\textbf{y}(\ell)$, allows for the calculation of the EP Bayes factor, facilitating model comparison and selection.
\chapter{A UNIVERSAL ROUBUST BOUND}
\section{The Smith and Spiegelhalter's method}
\hspace{10pt}This chapter presents an innovative methodology for computing Bayes Factors, building upon the seminal work of Spiegelhalter and Smith (1982) \cite{spiegelhalter1982bayesfactors}. Their approach, which encompassed examining an alternative prior specification, could be interpreted as a genuine subjective Bayesian analysis based on specific prior beliefs, or alternatively, as a formal analysis serving as a theoretical tool for comparing a model $M_0$ with a localized subset of models within another model $M_1$. Spiegelhalter and Smith employed the concept of "imaginary training samples" to maximize support for the simpler model $M_0$. While their paper primarily focused on assigning the constant $c_0/c_1$ in equation (5), Lempers (1971) \cite{lempers1971posterior} and Atkinson (1978) \cite{atkinson1978posterior} emphasized that assigning vague prior information directly to $c_0$ and $c_1$ is inherently arbitrary without incorporating a training sample to determine appropriate variable scaling. Consequently, they proposed partitioning a portion of the data for a pre-comparison inference phase, enabling the formation of suitable priors for the parameters within each model.

The method suggested by Spiegelhalter and Smith involved determining $\textbf{y}_0$ as the argument that maximizes ${m_0^N(\textbf{y}(\ell))}/{m_1^N(\textbf{y}(\ell))}$ and subsequently solving the equation 
$$c B_{01}^N(\textbf{y}_0) = 1.$$
This approach yielded the Bayes factor expressed as
$$\boxed{B^{SS}_{10}(\textbf{y}) = cB^N_{10}(\textbf{y}) = B^N_{01}(\textbf{y}_0)B_{10}^N(\textbf{y}).}$$

While Spiegelhalter and Smith's technique focused exclusively on nested point hypotheses, the upcoming sections will expand upon this method to encompass more general cases and examples. Furthermore, a novel lower bound will be introduced, leveraging the concept of "imaginary training samples" to maximize support for $M_0$. However, there is a distinction between this proposed method and the aforementioned approach by Smith and Spiegelhalter. Our method does not concern itself with the constant $c$; instead, we aim to establish a lower bound for various forms of Intrinsic Bayes Factors, including the arithmetic, geometric, expected and harmonic. This bound encompasses both theoretical and empirical versions, with the empirical version providing the tightest bound. The subsequent section will outline the procedures for obtaining both theoretical and empirical bounds.\\
\section{The universal robust bound for the Intrinsic Bayes Factors} 
\hspace{15pt} In this study, we propose a novel method that generates a lower bound for the Bayes factor and corresponding priors, employing an approach that is closely related to the concept of Intrinsic Bayes Factors and the technique employed by Spiegelhalter and Smith. Let $\textbf{y}(\ell)$ denote a minimal training sample, and $\pi(\boldsymbol{\theta_i}|\textbf y(\ell))$ represent the posterior distribution of parameter $\boldsymbol{\theta}_i$ given the minimal training sample for two models, $M_i$ and $M_j$. Our aim is to compare models $M_i$ and $M_j$. To achieve this comparison, we define the set $D$ as the aggregate of all possible minimal training samples with size $k$, where $k$ is determined as the maximum of $\{\dim(\boldsymbol{\theta_i}),\dim(\boldsymbol{\theta_j})\}$.
\begin{definition}
	The Intrinsic Bayes Factor Lower Bound  ($\underline{B}^{I}_{01}$), is the theoretical lower bound of all possible Intrinsic Bayes Factors obtained by computing the infimum over the set of all possible imaginary training samples D.
\end{definition}
\begin{definition}
	The Empirical Intrinsic Bayes Factor Lower Bound ($\underline{B}^{I^{*}}_{01}$) is the empirical lower bound obtained by computing the minimum over the set of all possible empirical training samples $D_{n,k}$.
\end{definition}
\section{A bridge between the universal bound and the IBF}
In this section, we will explore various approaches to construct lower bounds for the Intrinsic Bayes factor, ultimately contributing to the field of model comparison and selection. There is a bridge between IBF and the Least Favorable Intrinsic Bayes Factor approach that we want to explore by considering the AIBF, defined as,
\begin{equation}\boxed{
		B^{AI}_{10}(\textbf{y})=B^N_{10}(\mathbf{y}) \times \frac{1}{L}\sum_{\ell = 1}^{L} B^N_{01}(\mathbf{y(\ell)})}
\end{equation}
where $L={n \choose m}$, where m is the minimal training sample size.  The AIBF is asymmetric the more complex model has to be placed above to guarantee the convergence.
\begin{equation}\label{AIBF10}
	\boxed{ 
		\overline{B}^{I}_{10} := B^N_{10}\sup_{\textbf{y}(\ell) \in D} B_{01}(\textbf{y}(\ell)) \geq B^N_{10}\frac{1}{L}\sum B_{01}(\textbf{y}(\ell)) = B^{A}_{10}.
	}
\end{equation}
By forming the reciprocals,  we have the important consequence for a bound on the probability of a null hypothesis:
\begin{equation}\label{AIBF01}
	\boxed{ 
		\underline{B}^{I}_{01} := B^N_{01} \inf_{\ell} B_{10}(\textbf{y}(\ell)) = B^N_{01} \frac{1}{\sup_{\textbf{y}(\ell) \in D} B_{01}(\textbf{y}(\ell))} \leq B^N_{01} \frac{1}{\frac{1}{L}\sum B_{01}(\textbf{y}(\ell))}=B^{AI}_{01}.
	}
\end{equation}
where $\textbf{y}(\ell)$ is a theoretical training sample from the set of all possible training samples. By construction, all Bayes Factors obey: $B_{01}=1/B_{10}$ and $\overline{B}^{I}_{10} \geq B^{AI}_{10}$, since the supremum is necessarily bigger or equal than the arithmetic average. On the other hand, for generalizing the bound we may consider another bound based on empirical training samples,
\begin{equation}\label{ESS}
	B^{AI}_{01}\geq B^N_{01} \inf_{\textbf{y}(\ell) \in D_{n,k}} B_{10}(\textbf{y}(\ell)) = \underline{B}^{I^{*}}_{01},
\end{equation}
where $\textbf{y}(\ell)$ are empirical real training samples, and $\underline{B}_{01}^{I^{*}}$ is formed over all the real training samples, of a minimal size $m$. 
Clearly, $\overline{B}^{I*}_{10} \le \overline{B}^{I}_{10}$ and $ B^{AIBF}_{01} \ge \underline{B}^{I*}_{01} \ge \underline{B}^{I}_{01}.$
Thus, $\underline{B}^{I^{*}}_{01}$ and $\underline{B}^{I}_{01}$ are the least favorable lower bounds on the Bayes Factor. In addition, the results obtained with the AIBF are valid for other related types of Intrinsic Bayes Factors.
\subsection{A comparison with the p-value based lower bound}
In this section, we aim to illustrate the superiority of our bound in comparison to the well-known p-value-based robust lower bound introduced by Sellke et al. (2001) \cite{sellke2001calibration}. To commence our discussion, we will first establish the definition of a p-value.
\begin{definition}
	A p-value $p(\textbf{Y})$ is a statistic satisfying $0 \leq p(\textbf{y}) \leq 1$ for every sample point $\textbf{y}$. Small values of p(\textbf{Y}) give evidence that $H_1: \theta \in \Theta_0^{c}$ is true, where $\Theta_{0}$ is some subset of the parameter space and $\Theta_{0}^{c}$ is its complement. A p-value is valid if, for every $\theta \in \Theta_0$ and every $0 \leq \alpha \leq 1$,
	$$P_\theta(p(\textbf{Y}) \leq \alpha) \leq \alpha.$$
\end{definition}
The asymptotic behavior of the p-value as the sample size tends toward infinity is contingent upon both the null hypothesis and the fundamental distribution of the test statistic. Generally speaking, in the context of a two-sided hypothesis test, if the null hypothesis holds true, the p-value follows a uniform distribution within the range [0,1]. This implies that, irrespective of the sample size, the p-value will be uniformly distributed provided that the null hypothesis is valid.

The \emph{Robust Lower Bound} offers a framework for calibrating p-values and is defined as follows:

\begin{equation}
	B_{01}^{L}(p)=
	\begin{cases}
		-ep\log(p) & \text{if } p < \frac{1}{e}\\
		1 & \text{if } p \geq \frac{1}{e}
	\end{cases}
\end{equation}
\\
We postulate that given any sample size $n$ and for a fixed $p$, our IBF Lower Bound as a function of the p-value, $\underline{B}^{I}_{01} (p)$, is superior to the robust lower bound $B_{01}^{L}(p)$. It can be mathematically expressed as
$$\underline{B}^{I}_{01} (p) \ge \underline{B}^{I}_{01} (p) \ge B_{01}^{L}(p).$$
The last part of this inequality was proven by Sellke et. al (2001) \cite{sellke2001calibration} and it should be noted that $\underline{B}^{I}_{01} (p)$ is dynamic, changing with the sample size, akin to a real Bayes Factor, whereas $B_{01}^{L}(p)$ remains static with $n$. In the next theorem, we use the Laplace approximation to show that the robust lower bound stays constant under $H_0$ while our lower bound converges to the correct decision as the sample size increases. The Laplace approximation is a method used to approximate a probability distribution with a normal distribution and it's valid under the following regularity conditions:
\begin{enumerate}
	\item \textbf{Smoothness of the Logarithm of the Density}: For a density function $f(y)$, if $\log(f(y))$ is twice continuously differentiable in a neighborhood around the mode of the distribution, then the Laplace approximation tends to work well.
	
	\item \textbf{Existence of a Unique Mode}: The distribution should have a single, well-defined mode within the region of interest.
	\item \textbf{Curvature of the Log Density}: The curvature (second derivative) of $\log(f(y))$ at the mode should not be zero. A non-zero curvature ensures that the distribution can be well-approximated by a quadratic form, which is the characteristic shape of a normal distribution.
	
	\item \textbf{Convergence to Normality}: As the sample size tends to infinity, the distribution should converge to a normal distribution in a suitable sense (like the Central Limit Theorem).
\end{enumerate}
These conditions, particularly the smoothness and curvature of the log density, are crucial for the Laplace approximation to accurately represent the shape of the distribution around its mode.

However, it's important to note that while the Laplace approximation is often effective for many distributions that meet these conditions, it might not work well for highly skewed or heavy-tailed distributions where the normality assumption might not hold. In such cases, other approximation methods or numerical techniques might be more appropriate. 
\newpage
\begin{theorem} Let $f_0(\textbf{y}|\boldsymbol{\theta_0})$ and $f_1(\textbf{y}|\boldsymbol{\theta_1})$ be the densities that satisfies regularity conditions (1)-(4). Let $H_0 \subset H_1$, $k_0 = dim(\mathbf{\theta}_0)$ and $k_1 = dim(\mathbf{\theta}_1)$ where $k_0 < k_1$, then under $H_0$
	$$\frac{B_{01}^{L}(p)}{\underline{B}_{01}^{I}}  \to 0 \mbox{ as } n \to \infty$$\end{theorem}

\begin{proof}
	For a proper minimal training sample $y^{*}(\ell), \exists b,c$ such that $ 0 < b \leq c < \infty$ and
	$$0 < b \leq B_{10}^N(\textbf{y}^{*}(\ell)) \leq c $$
	The Laplace approximation for $B_{01}^N(\textbf{y})$ (Berger \& Pericchi (2001) \cite{berger2001objective}) yields
	$$\Big(\frac{f_0(\textbf{y}|\hat{\theta}_0)|\hat{I}_{0}|^{-1/2}}{f_1(\textbf{y}|\hat{\theta}_1)|\hat{I}_{1}|^{-1/2}}\cdot \frac{\pi^{N}_0(\hat{\theta}_0)}{\pi^{N}_1(\hat{\theta}_1)}\Big) \cdot b \leq B_{01}^{N}(\textbf{y}) B_{10}^{N}(\textbf{y}^{*}(\ell)) \leq \Big(\frac{f_0(\textbf{y}|\hat{\theta}_0)|\hat{I}_{0}|^{-1/2}}{f_1(\textbf{y}|\hat{\theta}_1)|\hat{I}_{1}|^{-1/2}}\cdot \frac{\pi^{N}_0(\hat{\theta}_0)}{\pi^{N}_1(\hat{\theta}_1)}\Big) \cdot c$$
	
	where $\hat{I}_i$ and $\hat{\theta}_i$ are the observed information matrix and m.l.e, respectively under model $M_i$. The BIC approximation for $\underline{B}_{01}^{I}$ gives us
	
	$$ b\cdot\frac{f_0(\textbf{y}|\hat{\theta}_{0})}{f_1(\textbf{y}|\hat{\theta}_{1})} \cdot \big(n\big)^{(k_1-k_0)/2} \leq \underline{B}_{01}^{I} \leq c\cdot\frac{f_0(\textbf{y}|\hat{\theta}_{0})}{f_1(\textbf{y}|\hat{\theta}_{1})} \cdot \big(n\big)^{(k_1-k_0)/2} $$
	
	In the nested hypothesis case under the null hypothesis, the likelihood ratio in the equation above is bounded because $-2\cdot \log(LR_{01})$ is bounded  (it converges to a central Chi-square distribution by Wilks theorem). Hence, $\underline{B}_{01}^{I}(\textbf{y}) \to \infty$ and $\overline{B}_{10}^{I}(\textbf{y}) \to 0 \mbox{ as } n \to \infty$ and by definition, $ 0\leq B_{01}^{L}(p) \leq 1$. Hence,
	$$ 0 \leq\frac{B_{01}^{L}(p)}{\underline{B}_{01}^{I}} \leq \frac{1}{\underline{B}_{01}^{I}} = \overline{B}_{10}^{I} \longrightarrow \lim_{n\to\infty}0 = 0 \leq\lim_{n\to\infty}\frac{\underline{B}_{01}^{L}(p)}{\underline{B}_{01}^{I}} \leq \lim_{n\to\infty} \overline{B}_{10}^{I} = 0$$
	By the Squeezing Theorem,
	$$\lim_{n\to\infty}\frac{B_{01}^{L}(p)}{\underline{B}_{01}^{I}} = 0.\hspace{30pt}$$
\end{proof}
Although this proof is for nested hypotheses, we believe that this is true in general and we support this claim with examples for separated hypotheses (see Geometric vs. Poisson, Negative Binomial vs Poisson examples).\\ \\
\textbf{A study of the p-value and the RBL vs. the SS Bayes factor under $H_1$}
\\We want to study the following limit under $H_1$
$$\lim_{n\to\infty} \frac{B_{01}^{L}(p)}{\underline{B}_{01}^{I}} $$
Using an approach similar to the one used in the previous theorem, we can see that 
$$ \frac{f_1(\textbf{y}|\hat{\theta_1})}{c\cdot n^{k_1-k_0}f_0(\textbf{y}|\hat{\theta_0})} \leq \frac{B_{01}^{L}(p)}{\underline{B}_{01}^{I}} \leq \frac{f_1(\textbf{y}|\hat{\theta_1})}{b\cdot n^{k_1-k_0}f_0(\textbf{y}|\hat{\theta_0})}$$

To explore this limit we need the asymptotic distribution of the likelihood ratio under the alternative hypothesis and related results by Self, S. G., \& Liang, K. Y. (1987) \cite{self1987asymptotic} known as the generalized likelihood ratio test (GLRT) can be used to handle the case where the alternative hypothesis is true. The GLRT is a modification of the likelihood ratio test that allows for testing of both null and alternative hypotheses, and it is based on the same asymptotic theory as Wilks' theorem. Specifically, as the sample size goes to infinity, the distribution of the GLRT under the alternative hypothesis approaches a non-central chi-squared distribution with a non-centrality parameter proportional to the squared distance between the true parameter value and the null parameter value, and degrees of freedom equal to the difference in the number of parameters between the null and alternative hypotheses. Mathematically, we can write this as:

$$LR \sim \chi^2(p,\lambda) \mbox { as } n \to \infty  \mbox{ under } H_1 $$

where $\chi^2(p,\lambda)$ denotes the non-central chi-squared distribution with p degrees of freedom, the non-centrality parameter $\lambda$ and $LR = 2*\log(L(\hat{\theta}_1)/L(\hat{\theta}_0))$. The GLRT can be used to test hypotheses in a wide range of statistical models, including linear regression, logistic regression, and ANOVA, among others. So, while Wilks' theorem applies only to the null hypothesis, the related GLRT provides a way to handle testing of the alternative hypothesis using the same underlying asymptotic theory. 

In some cases, it may be possible to write the asymptotic distribution of the likelihood ratio test statistic as a function of n. 





\hspace{10pt}However, note that this formula applies only to specific cases, and the asymptotic distribution of the likelihood ratio test statistic may be different for other models and hypotheses. Determining the non-centrality parameter $\lambda$ for the asymptotic distribution of the likelihood ratio test statistic can be challenging and requires knowledge of the true parameter values under the alternative hypothesis. In general, there is no simple formula that can be used to calculate $\lambda$ for any hypotheses test. In some cases, simulation methods may be used to estimate the non-centrality parameter $\lambda$. This involves generating simulated data sets from the assumed distribution under the alternative hypothesis, calculating the likelihood ratio test statistic for each data set, and then estimating the distribution of the test statistic based on the simulated values. The estimated distribution can then be compared to the theoretical distribution to estimate the non-centrality parameter.
\hspace{10pt} In general, determining the non-centrality parameter and its rate of convergence for the asymptotic distribution of the likelihood ratio test statistic is an active area of research, and the methods used to estimate $\lambda$ depend on the specific model and hypothesis being tested. We conjecture that if there exist $r > 0$ such that $\lim_{n\to\infty}\frac{LR}{n^{r}}$ is constant under $H_1$, then
\begin{equation}
	\lim_{n\to\infty} \frac{B_{01}^{L}(p)}{\underline{B}_{01}^{I}}= 0 \mbox { if } k_1 - k_0 > r
\end{equation}
because if this limit is zero, then similarly to the case when $H_0$ is true, we can show that the limit of the ratio is also zero. Further research is needed to determine if \textbf{Theorem 3.1} is true under the alternative hypothesis in general, or if this theorem only holds under certain conditions. 
\chapter{The SP Bayes Factors}
In this chapter, we introduce a novel prior known as the SP Prior and leverage it to establish what we term the SP Bayes Factor. The designation 'SP' denotes the 'superior posterior' approach, which underpins the creation of these priors. The underlying concept involves transforming an improper prior into a proper posterior distribution by identifying either the theoretical or empirical minimal training sample that maximizes evidence for the null hypothesis (or yields the least favorable prior for the alternative).

Both the prior and the Bayes factor are presented in two versions-empirical and theoretical-similar to the universal lower bound discussed in the preceding chapter. These priors exhibit a close relationship with IBF bounds, which we elucidate within this chapter.

The empirical prior proves valuable in cases where theoretical maximization poses challenges. Remarkably, the empirical SP Bayes Factor closely approximates the actual Bayes factor value compared to the theoretical counterpart. Nevertheless, in scenarios where theoretical maximization remains viable, the theoretical prior offers computational efficiency over the empirical one. Employing both priors, the SP Bayes factors converge towards accurate decisions with increasing sample sizes-a point we substantiate both theoretically and through numerical examples in subsequent sections.\newpage
\section{The SP Prior}
\begin{definition}[The Theoretical SP Prior]
	\vspace{10pt}
	The Theoretical SP prior for $B_{10}$ is defined as
	$$\pi^{SP}_k(\theta_k) = \pi_k(\theta_1|\textbf{y}^{*}_T(\ell)) = \frac{f_k(\textbf{y}^{*}_T(\ell)|\theta_k)\pi_k^{N}(\theta_k)}{m_k(\textbf{y}^{*}_T(\ell))}, k = 0,1$$
\end{definition}
$$\mbox{where } \textbf{y}^{*}_{T}(\ell) = \arg[\sup_{\textbf{y}(\ell) \in D}B_{01}(\textbf{y}(\ell))].$$
\begin{definition}[The Empirical SP Prior]
	The Empirical SP prior for $B_{10}$ is defined as
	$$\pi_k^{SP*}(\theta_k) = \pi_k(\theta_k|\textbf{y}_E(\ell)) = \frac{f_k(\textbf{y}^{*}_E(\ell)|\theta_k)\pi_k^{N}(\theta_k)}{m_k(\textbf{y}^{*}_E(\ell))}, k = 0,1$$
\end{definition}
$$\mbox{where } \textbf{y}^{*}_{E}(\ell) = 
\arg[\sup_{\textbf{y}(\ell) \in D_{n,k}}\frac{m_0(\textbf{y}(\ell))}{m_1(\textbf{y}(\ell))}].$$
\section{The SP Bayes Factor}
\begin{definition}[The Theoretical SP Bayes Factor]\end{definition}
If $\displaystyle\arg\sup_{\textbf{y}(\ell) \in D}B_{01}(\textbf{y}(\ell)) < \infty$, then
$$B^{SP}_{10} = \frac{m_1^{SP}(\textbf{y})}{m_0^{SP}(\textbf{y})} = \frac{\displaystyle\int f_1(\textbf{y}|\theta_1)\pi_1^{SP}(\theta_1)d\theta_1} {\displaystyle\int 
	f_0(\textbf{y}|\theta_0)\pi_0^{SP}(\theta_0)d\theta_0}.$$
\begin{definition}[The Empirical SP Bayes Factor]\end{definition}
Let $\textbf{y}^{*}_T(\ell) = \displaystyle\arg\sup_{\textbf{y}(\ell) \in D_{n,k}}B_{01}(\textbf{y}(\ell))$, and $\textbf{y}(-\ell) = \textbf{y}\setminus \textbf{y}^{*}_T(\ell)$, then 
$$B^{SP*}_{10} = \frac{m_1^{SP*}(\textbf{y})}{m_0^{SP*}(\textbf{y})} = \frac{\displaystyle\int f_1(\textbf{y}(-\ell)|\theta_1)\pi_1^{SP*}(\theta_1)d\theta_1} {\displaystyle\int f_0(\textbf{y}(-\ell)|\theta_0)\pi_0^{SP*}(\theta_0)d\theta_0}.$$
\section{Results}
\begin{theorem}
	\label{EmpiricalSPToEmpiricalSSBF}
	The Empirical IBF upper and lower bounds can be obtained using the Empirical SP Prior, that is
	$$ B^{SP*}_{10} =B_{10}^{N}[\textbf y] \sup_{\ell = 1,...,L}B_{01}{[\textbf{y}(\ell)]} = \overline{B}_{10}^{I*}(\textbf{y}).$$ 
\end{theorem}
\begin{proof}
	Let $\textbf{y}(-\ell) = \textbf{y}\setminus\{ \textbf{y}^{*}_E(\ell)\}$, without loss of generality, we will prove it for the upper-bound with the supremum since the proof for the infimum lower bound is identical. The Empirical SP Bayes factor is
	$$B^{SP*}_{10} = \frac{m_1^{SP*}(\textbf{y})}{m_0^{SP*}(\textbf{y})} = \frac{\displaystyle\int f_1(\textbf{y}(-\ell)|\boldsymbol{\theta}_1)\frac{f( \textbf{y}^{*}_E(\ell)|\boldsymbol{\theta}_1)\pi_i^{N}(\boldsymbol{\theta}_1)}{m_1( \textbf{y}^{*}_E(\ell))}d\boldsymbol{\boldsymbol{\theta}}_1} {\displaystyle\int f_0(\textbf{y}(-\ell)|\boldsymbol{\theta}_0)\frac{f( \textbf{y}^{*}_E(\ell)|\boldsymbol{\theta}_0)\pi_j^{N}(\boldsymbol{\theta}_0)}{m_0( \textbf{y}^{*}_E(\ell))}d\boldsymbol{\theta}_0}$$
	The marginal distributions $m_0$ and $m_1$ do not depend on $\boldsymbol{\theta}_0$ and $\boldsymbol{\theta}_1$, hence we can take those terms out of the integrals to obtain
	$$ \frac{m_0( \textbf{y}^{*}_E(\ell))}{m_1( \textbf{y}^{*}_E(\ell))} \frac{\displaystyle\int f_1(\textbf{y}(-\ell)|\theta_1)f( \textbf{y}^{*}_E(\ell)|\theta_1)\pi_i^{N}(\theta_1)d\theta_1} {\displaystyle\int f_0(\textbf{y}(-\ell)|\theta_0)f( \textbf{y}^{*}_E(\ell)|\theta_0)\pi_0^{N}(\theta_0)d\theta_0} = \frac{m_0( \textbf{y}^{*}_E(\ell))}{m_1( \textbf{y}^{*}_E(\ell))}\frac{m_1(\textbf{y}(-\ell), \textbf{y}^{*}_E(\ell))}{m_0(\textbf{y}(-\ell), \textbf{y}^{*}_E(\ell))}$$
	$$= B_{01}[ \textbf{y}^{*}_E(\ell)] B^N_{10}[\textbf{y}(-\ell), \textbf{y}^{*}_E(\ell)]$$
	By definition, we have that
	$$  B_{01}[ \textbf{y}^{*}_E(\ell)]  =  B_{01}[\arg(\sup_{\ell = 1,...,L}B_{01}(\textbf{y}(\ell)))] = \sup_{\ell = 1,...,L} B_{01}[\textbf{y}(\ell)] $$
	and $\{ \textbf{y}^{*}_E(\ell)\}\cup\{\textbf{y}(-\ell)\} = \textbf{y}$ and $m_k(\textbf{y}(-\ell), \textbf{y}^{*}_E(\ell)) = m_k(\textbf y),$ for  $k = 0,1$
	\\ \\
	Therefore,
	$$B^{SP*}_{10} =  B^N_{10}[\textbf{y}]\sup_{\ell = 1,...,L} B^{N}_{01}[\textbf{y}(\ell)] =\overline{B}_{10}^{I*}(\textbf{y})$$ 
\end{proof}
\begin{theorem} Let $f(\textbf{y}|\boldsymbol{\theta})$ be a p.d.f or p.m.f a random variable $Y$ with support $S = \{ y: f(y|\theta) > 0 \}$. Suppose we want to calculate the Bayes factor for the following nested hypothesis test $H_0: \theta = \theta_0 \mbox{ vs } H_1: \theta \neq \theta_0$, then the empirical SP-Prior and SP-BF converges to the Theoretical SP-Prior and SP-BF.\end{theorem}

\begin{proof} Without loss of generality we can assume that $dim(\theta) = 1$, so that the minimal training sample size is one. In this case, the set $D_{n,k} = D_{n,1} = \textbf y_n = (y_1,...,y_n) \subseteq S $, $\displaystyle\lim_{n\to \infty} D_{n,1} = D$ and $L = {n\choose k} = {n \choose 1} = n$. The proof will be divided in the discrete and continuous support cases. \\
	
	\textbf{Case 1:} The support $S$ is discrete
	$$\lim_{n\to\infty} D_{n,1} = \lim_{L\to\infty}\{\textbf{y}(\ell)\}_{l=1}^{L} =  \{\textbf{y}(\ell)\}_{l=1}^{\infty} = D = S$$
	Therefore,
	$$\displaystyle\lim_{n\to\infty}B_{01}^{SP*}(\textbf{y}) = B^N_{01}[\textbf{y}]\displaystyle\lim_{L\to\infty}\sup_{\{\textbf{y}(\ell)\}_{\ell = 1}^{L}} B_{10}[\textbf{y}(\ell)] = B^N_{01}[\textbf{y}]\sup_{D}B_{10}[\textbf{y}(\ell)]$$
	Hence,
	$$ \lim_{n \to \infty}B_{01}^{SP*}(\textbf{y}) = B_{01}^{SP}(\textbf{y}) $$
	Clearly, in the discrete support case, we can eventually (for some n) obtain the minimal training sample $\textbf{y}_n^{*}(\ell)$ that generates the theoretical supremum, and the empirical SS-BF will be exactly equal to the theoretical SSBF. \\ \\
	\indent \textbf{Case 2:} The support $S$ is continuous\\
	In the case of a continuous support $S$, we can only find a sequence of minimal training samples $\{\textbf{y}_{n}^{*}(\ell)\}$ that gets arbitrarily close to a minimal imaginary sample $\textbf{y}^{*}(\ell)$ that produces the theoretical supremum. We will construct such a sequence and prove its existence. Assume that $\textbf{y}^{*}(\ell) = \arg\sup_{D}B_{10}(\textbf{y}(\ell))$ and for each $n$,
	$$ 
	\textbf{y}_n^{*}(\ell) = \textbf{y}(\ell_{i_n}), \mbox{ where } i_n = \arg \min_{i \in \{1,...,{n \choose 1}\}}|\textbf{y}^{*}(\ell) - \textbf{y}(\ell_i)|
	$$
	We claim that $\textbf{y}_n^{*}(\ell) \to \textbf{y}^{*}(\ell)$ for any $\textbf{y}^{*}(\ell) \in D$\\
	\textbf{Proof}: \\ 
	Assume that $\textbf{y}^{*}(\ell) \in D$ and $\exists \epsilon > 0$ such that 
	
	$$\forall n \in \mathbb{N}, |\textbf{y}^{*}(\ell) - \textbf{y}_{n}^{*}(\ell)| > \epsilon$$
	Then the interval $(\textbf{y}^{*}(\ell) - \frac{\epsilon}{2},\textbf{y}^{*}(\ell) + \frac{\epsilon}{2}) \cap D = \emptyset$ which is a contradiction since $\textbf{y}^{*}(\ell) \in D$, thus the intersection can not be empty. Similarly, if such a point does exist in $D$, then 
	$$ \int^{\textbf{y}^{*}(\ell)+\frac{\epsilon}{2}}_{\textbf{y}^{*}(\ell)-\frac{\epsilon}{2}}f(\textbf{y}|\theta)d\textbf{y} = 0 \Longrightarrow \mbox{ f is not continuous in } \textbf{y}^{*}(\ell) \Longrightarrow \textbf{y}^{*}(\ell) \notin D$$
\end{proof}
\chapter{Hypothesis Testing with the Normal Distribution}
This chapter explores hypothesis testing for mean and precision (the inverse of scale) parameters within the Normal Distribution. It covers essential concepts such as the universal lower bound, the SP/Intrinsic Priors, and the SP/Intrinsic Bayes factor. Through numerical experiments and visual charts, we compare these approaches and visualize the results. By blending theoretical foundations with practical applications, this chapter aims to provide readers with a robust understanding of conducting and interpreting these crucial statistical tests, leveraging our novel theoretical framework.. \\
\section{The Normal Scale Hypothesis Test}
\hspace{10pt}Assume $\mathbf{y} = \{y_1,y_2,...,y_n\}$ are i.i.d samples where $y_i \sim N(y_i|\mu, h)$ for every $i$, and that the mean $\mu$ is unknown. We wish to decide if the precision parameter $h=1/\sigma^2$ is a specified value. If we want to make a decision, we should perform the following hypothesis test 
\[
H_0: h=h_0 \mbox{ vs } H_1: h\neq h_0,\
\]
The likelihood is $f(\textbf{y}|\mu,h)=( \frac{h}{2 \pi})^{n/2} \exp[-\frac{h}{2} (S^2+ n(\bar{\textbf{y}}-\mu)^2$, where $S^2=\sum_{i=1}^n (y_i-\bar{\textbf{y}})^2$. As a first step, use the so-called "Jeffrey's rule" equating the prior to the square root of the determinant of the Fisher Information. In the motivating example this leads to: $\pi^J(\mu,h)=C_0 \cdot \frac{d\mu dh}{\sqrt{h}}$.
For the so-called \textquotedbl independence" Jeffrey's we have instead
$$\pi_I^J(\mu,h)=\pi^J(\mu) \pi^J(h)= C_1 \cdot \frac{d\mu dh}{h}$$
Putting together both priors
$\pi^J(\mu,h)=C/h^r$ with $r=1/2, 1$ for the dependent or independent Jeffreys prior respectively. It turns out that the Bayes Factor is,
\begin{equation}
	B^N_{01}(\mathbf{y})=\frac{h_0^{(n-1)/2} \exp(-h_0 S^2/2)}{(2/S^2)^{(n-r)/2} \Gamma((n-r)/2)} \times \frac{C_0}{C_1},
\end{equation}
where $r$ is equal to $1/2$ for the dependent prior, and for the independent prior equal to one. In this thesis, we assume the independent prior, as put forward by Jeffreys.It is apparent from this expression that the Bayes Factor is undetermined since it depends on the undefined ratio $C_0/C_1$, the undefined constants that come from the improper priors. It can be argued, from several points of view, that the constants regarding the location $\mu$ cancel out, leaving only the indeterminacy related to the hypothesis parameter $h$. Some of these points of view are among others:
\begin{enumerate}
	\item Mean and Precision parameters are orthogonal in the Fisher Information Matrix, and thus "cancel-out" in the ratio of marginal likelihoods
	\item It turns out that two location models are predictively matched (see Pericchi (2005) \cite{pericchi2005model}) and when the scale is integrated into the denominator, models under both hypotheses become location models, and the corrections cancel out.\\
\end{enumerate}
\subsection{The universal lower bound}
\hspace{17pt} It is evident from the previous example that, unfortunately, the Bayes factor remains defined only up to the constant that pertains to the precision parameter under test. This is why Jeffrey's suggested conventional 
proper priors for the extra-parameters under the larger hypothesis (but improper for common parameters like $\mu$ here). There have been suggested techniques around this problem, as in Smith and Spiegelhalter (1980) \cite{smith1980bayesfactors} and (1982) \cite{spiegelhalter1982bayesfactors}. These techniques, although approximate, are useful in devising sensible scaling for Bayes Factors. In fact, we proved there for the first time that this approach has a very direct relationship with more recent approaches that have been studied in detail, particularly the Intrinsic Bayes Factor, the Intrinsic Priors, and EP-priors.

In the example of the section, the minimal training sample consists of two observations $(y(\ell_1),y(\ell_2))$, and the expression (2) turns out to be
$$
B^N_{01}(\mathbf{y}) \times B^N_{10}(\mathbf{\textbf{y}(\ell)}) =\frac{h_0^{(n-1)/2}\exp(-h_0 S^2/2)}{(2/S^2)^{p} \frac{C_0}{C_1}\sqrt{h_0/ \pi} \cdot \exp(-h_0 \textbf{D}(\ell)^2/4) |\textbf{D}(\ell)| \cdot \frac{C_1}{C_0}}, 
$$

where $\textbf{D}(\ell)={y(\ell_1)-y(\ell_2)}$ and $p = (n-1)/2 \Gamma((n-1)/2)$. It is apparent from (3) that the undefined constants cancel out. As a specific example, take $H_0:h_0=1$. The crucial step is: what to do about the (theoretical or imaginary) training samples summary statistics $\textbf{D}(\ell)$? SSBF's practical and simplifying approach is to take: 
\begin{equation}
	\sup_{\textbf{y}(\ell) \in D} B^N_{10}(\textbf{D}(\ell)),
\end{equation}
which is attained at $\hat{D}=\sqrt{2}$ and the "correction" (4) is equal to 0.484. The $IBF$ bound is then $\overline{B}_{01}^{I}=B^N_{01} \times \sup_{\textbf{y}(\ell) \in D} B^N_{10}(\textbf{y}(\ell)).$
To see the relationship with more recent approaches like Intrinsic Bayes Factors the $\sup$ above is replaced by the arithmetic mean or by the theoretical expectation, on $\textbf{y}(\ell)$, under $H_1 \supset H_0$. The \textbf{theoretical} expressions in (4) and (5), which take the formal maximum possible value for all possible training samples, can be refined by the \textbf{empirical} observed maximum for the observed $\textbf{y}(\ell)$'s:
\begin{equation}
	\sup_{\ell = 1,...,L} B^N_{10} (\mathbf{y(\ell)}) \le  \sup_{\textbf{y}(\ell) \in D} B^N_{10} (\textbf{y}(\ell)),
\end{equation}
where the distinction arises as $\mathbf{y(\ell)}$ is the set of empirically observed training samples and D which is the set of all possible minimal training samples from the whole support. The empirical bound is then sharper than the theoretical bound, and we can refer to the theoretical bound and the empirical bound of the Intrinsic Bayes Factors. In the Expected IBF case, for any pair of i.i.d. training samples $\{y(\ell_1),y(\ell_2)\}$, assumed (when taking expectations) to be drawn from $H_1\supset H_0$, results in $\textbf{D}(\ell) \sim N(0,2/h)$. Then under $H_1$:
\[
\frac{B^{EIBF}_{10}(\textbf{y})}{B^N_{10}(\mathbf{y})}= \int \frac{|\textbf{D}|}{\sqrt{\pi}} exp(-\frac{\textbf{D}^2}{4})\cdot \frac{\sqrt{h}\cdot exp(-h\cdot \textbf{D}^2/4)}{2 \sqrt{\pi}} d\textbf{D}=\frac{2 \sqrt{h}}{\pi (h+1)}.
\]
In an applied problem with real data, the parameter $h$ will be estimated by its MLE. In this theoretical exercise, we study it as a function of $h$, and computation yields that its maximum is attained at the null hypothesis $h=h_0=1$ and its value at its maximum is $1/\pi=0.318$. \\
\subsection{Remarks}
\begin{enumerate}
	\item Our ratio of constants is 0.484 and the Maximum of the Expected IBF of the Training Sample is 0.318. This may seem like a big difference (the ratio is about .66). These values are attained close to the null hypothesis, the difference far from the null, may be much bigger but there, the dominant factor $B^N_{01}(\mathbf{y})$, will overwhelm that difference for a moderate sample, and even small but larger than the training sample sizes. 
	\item The lower bound is robust, it does not depend on the fact that all the training samples belong to the alternative hypothesis $H_1$ and that all belong to the same population. This bound achieves a very important robust effect. The bound has Universal validity: It is still valid under violations of distributional assumptions, for example, Normality or homogeneity of the population of training samples. It can also be used to check the computations of IBFs and with Intrinsic Priors.
	\item The SS's bound on Intrinsic Bayes Factors is typically easy to compute and generalize to realistic scenarios.
	\item In terms of p-values, if it can be found a "tight" bound $g(p)$ so that $g(p) \le B^N_{01}$ then 
	\begin{equation}
		g(p) \times \frac{1}{\sup B^N_{01}(\textbf{y}(\ell))} \le \underline{B}^{I}_{01} \le B_{01}
	\end{equation}
	which is a Universal Bound that adapts to the sample size n.\\
\end{enumerate}
\newpage
\subsection{The SP and Intrinsic Priors}
\hspace{15pt}In this section, we turn our attention to the calculation of Superior Posterior (SP) and Intrinsic priors within the context of the Normal Scale Hypothesis testing scenario. An in-depth exploration will be conducted to illuminate the nuanced attributes and similarities shared between these two types of priors. To enhance the comprehension and tangibility of the underlying principles, we shall include illustrative representations in the form of charts. These graphical elucidations will serve as effective tools for visualizing the properties and similarities of both SP and Intrinsic priors.

Moreover, to ensure results and facilitate the application of the presented methods, we will accompany our discussion with executable \textbf{R} code snippets. These will illustrate the computational procedures required to carry out the prior calculations and generate the corresponding charts. By making this code available, we aim to provide a practical and accessible means for readers to engage with our methodology.
Consider the Bayes factor for the minimal training sample of the normal scale example with an unknown mean 
$$ B_{10}^N(\textbf{y}(\ell)) \propto \frac{1}{\exp\{-h_0 \textbf{D}(\ell)^2/4\}|\textbf{D}(\ell)|} = \frac{\exp\{h_0\textbf{D}(\ell)^2/4\}}{|\textbf{D}(\ell)|}$$

To find the value of $\textbf{D}(\ell)$ that generates the $\sup$ of $B_{10}^N( \textbf{y}(\ell))$, 
we must take logarithms on both sides, and differentiate the equation $B_{10}^N( \textbf{y}(\ell)) = 0$ with respect to $\textbf{D}(\ell)$,
$$\frac{d}{d\textbf{D}(\ell)}\log(\frac{\exp\{h_0 \textbf{D}(\ell)^2/4\}}{|\textbf{D}(\ell)|}) = 2h_0\frac{\hat{D}}{4} - \frac{1}{\hat{D}} = 0 \Longrightarrow h_0\frac{\hat{D}^2}{2} = 1 \Longrightarrow \boxed { \hat{D} = \sqrt{\frac{2}{h_0}  }}$$
After calculating this value, we must consider the prior of $\mu$ and $h$ given $\textbf{y}(\ell)$ which is
\begin{equation}
	\pi(\mu,h|\textbf{y}(\ell)) = \frac{f(\textbf{y}(\ell)|\mu,h)\pi(\mu)\pi(h)}{\int f(\textbf{y}(\ell)|\mu,h)\pi(\mu)\pi(h)d\mu dh} =  \frac{\frac{1}{h}\frac{h}{2\pi}\exp\{-\frac{h}{2}[(y(1) - \mu)^2 + (y(2) - \mu)^2]\}}{\frac{1}{2\pi}\displaystyle\int  \exp\{-\frac{h}{2}(\frac{\textbf{D}(\ell)^2}{2} + 2(\mu - \bar{\textbf{y}}(\ell))^2 \}d\mu d h} 
\end{equation}

Simplifying this equation, we obtain
$$
\frac{\exp\{-\frac{h}{2}[\frac{\textbf{D}(\ell)^2}{2} + 2(\mu - \bar{\textbf{y}}(\ell))^2] \}}
{\displaystyle\int^{\infty}_{0}\exp\{-\frac{h}{4}\textbf{D}(\ell)^2\}\frac{\sqrt{2\pi}}{\sqrt{2h}}dh}
$$
To solve the last integral, we should consider the following change of variable 
$$u = \frac{h\textbf{D}(\ell)^2}{4}, \hspace{10pt} h = \frac{4u}{\textbf{D}(\ell)^2}, \hspace{10pt}dh = \frac{4}{\textbf{D}(\ell)^2}du$$
$$
\frac{4\sqrt{\pi}}{\textbf{D}(\ell)^2}\displaystyle\int_{0}^{\infty}(\frac{4u}{\textbf{D}(\ell)^2})^{\frac{1}{2} - 1}\exp\{-u\}du = 
\Gamma(\frac{1}{2})\frac{4\sqrt{\pi}}{\textbf{D}^2}\frac{D}{2} = \frac{2\pi}{\textbf{D}}
$$
Hence,
\begin{equation}
	\pi(\mu,h|D=\hat{D}) = \frac{\exp\{-\frac{h}{2}[\frac{\hat{D}^2}{2} + 2(\mu - \bar{y}(\ell))^2] \}\hat{D}}
	{2\pi}
\end{equation}
Integrating with respect to $\mu$, we obtain the prior for $h$
\begin{align*}
	\pi(h|\textbf{D}) &= \int \pi(\mu,h|\textbf{D}) \, d\mu = \int \frac{\textbf{D}}{2\pi} \exp\left\{-\frac{h}{2}\left(\frac{\textbf{D}^2}{2} + 2(\mu - \bar{\textbf{y}}(\ell))^2\right)\right\} \, d\mu \\
	&= \frac{\textbf{D}}{2\pi} \exp\left\{-\frac{h}{2}\frac{\textbf{D}^2}{2}\right\}\frac{\sqrt{2\pi}}{\sqrt{2h}} = \frac{\textbf{D}\exp\left\{-\frac{h\textbf{D}^2}{4}\right\}}{\sqrt{2\pi}\sqrt{2h}}
\end{align*}
and by substituting $\textbf{D}$ by $\hat{D}$, we can obtain the SP prior for $h$
$$ \boxed{ \pi^{SP}(h|\hat{D} = \sqrt{\frac{2}{h_0}}) = \frac{1}{\sqrt{2\pi hh_0}}\exp\{-\frac{h}{2h_0}\}}$$
Evidently, the prior $\pi^{SP}(h)$ is $Gamma(\alpha = 1/2,\beta = 2h_0)$.
To find the prior for $\mu$, we must integrate with respect to $h$ as follows
$$\pi^{SP}(\mu|\textbf{D}) = \displaystyle\int_{0}^{\infty}\exp\{-\frac{h}{2}[\frac{\textbf{D}^2}{2} + 2(\mu - \bar{\textbf{y}}(\ell))^{2}]\}\frac{\textbf{D}}{2\pi}dh = \frac{8\pi}{\textbf{D}^2 + 4(\mu - \bar{\textbf{y}}(\ell))}$$ 

$$ \pi^{SP}(\mu|\hat{D} = \sqrt{\frac{2}{h_0}}) = \frac{8\pi}{\frac{2}{h_0} + 4(\mu - \bar{\textbf{y}}(\ell))^2} = \frac{2\pi}{\frac{1}{2h_0} + (\mu - \bar{y}(\ell))^2} = \frac{4h_0\pi}{1 + (\frac{\mu - \bar{\textbf{y}}(\ell)}{\sqrt{2h_0}})^2} $$

The distribution of $\mu$ for $B_{01}$ is proportional to a $Cauchy(\bar{\textbf{y}},\gamma)$ where $\gamma = \sqrt{2h_0}$. 
$$ \boxed{ \pi^{SP}(\mu) \propto \frac{1}{\pi \gamma\Big(1 + (\frac{\mu - \bar{\textbf{y}}(\ell)}{\gamma})^2\Big)}  }$$
Next, we will compute the intrinsic priors for the same hypothesis test above to compare those with the SP priors. The improper intrinsic prior for $\mu$ is proportional to a constant $\pi^{I}(\mu) \propto c$. In the other hand, the intrinsic prior for the scale $h$ is
$$ \pi^{I}(\mu,h) = \pi^{N}(\mu,h)\displaystyle\int \frac{m_0^{N}(\textbf{y}(\ell))}{m_1^{N}(\textbf{y}(\ell))}f(\textbf{y}(\ell)|\mu,h)d\textbf{y}(\ell) = \displaystyle\int \pi(\mu,h|\textbf{y}(\ell))m_0(\textbf{y}(\ell))d\textbf{y}(\ell)$$
$$ \boxed{ \pi^{I}(h) \propto \frac{1}{\pi}\frac{\sqrt{\frac{h}{h_0}}}{\frac{h+1}{h_0}}\frac{1}{h} = \frac{1}{\pi h_0}\frac{(\frac{h}{h_0})^{1/2 -1}}{(\frac{h}{h_0} + 1)} }$$

Evidently, the prior $\pi^{I}(h)$, is a $SBeta2(h|p=1/2,q=1/2,b=h_0)$ (Scale beta2 distribution). The Scale beta two density is given by
$$ f(y|p,q,b) = \frac{\Gamma(p + q)}{\Gamma(p)\Gamma(q)}\frac{(y/b)^{p -1}}{(\frac{y}{b} + 1)^{(p+q)}}, \hspace{5pt} y,p,q,b > 0$$ 
The subsequent script is written in the \textbf{R} programming language that facilitates the generation of a comparative graphical representation of both priors. These priors are superimposed on the same set of axes to foster a direct comparison. This visual comparison is executed across distinct values of the null hypothesis ($H_0 = 1$ and $H_0 = 10$).\\
$\textbf{R program}$
\begin{lstlisting}[language=R]
	sp_prior = function(h,h0) { (1/sqrt(2*pi*h*h0))*exp(-h/(2*h0)) }
	i_prior = function(h,h0) { (1/(pi*h0))*( (h/h0)^(-0.5) )/(h/h0+1) }
	par(mfrow=c(1,2)); h = seq(0.1,5,0.01);h0=1; iprior = i_prior(h,h0); spprior = sp_prior(h,h0)
	ymin = min(spprior,iprior); ymax = max(spprior,iprior)
	plot(h,spprior,type="l",main="Comparison of the SP and Intrinsic Priors of h",ylab="Priors")
	mtext(text="Hypothesis test: H0 = 1 vs H0 != 1",side=3); lines(h,iprior,type="l",col="red")
	legend(3.4, 1.2, legend=c("SP", "Intrinsic"),col=c("black", "red"), lty=c(1,1), cex=0.8,title="Priors", text.font=4, bg='lightblue')
	h = seq(1,15,0.01);h0=10; iprior = i_prior(h,h0); spprior = sp_prior(h,h0)
	ymin = min(spprior,iprior); ymax = max(spprior,iprior)
	plot(h,spprior,type="l",main="Comparison of the SP and Intrinsic Priors of h",ylim=c(ymin,ymax),ylab="Priors")
	mtext(text="Hypothesis test: H0 = 10 vs H0 != 10",side=3); lines(h,iprior,type="l",col="red")
	legend(10.5, 0.12, legend=c("SP", "Intrinsic"),col=c("black", "red"), lty=c(1,1), cex=0.8,title="Priors", text.font=4, bg='lightblue')
\end{lstlisting}
\begin{figure}[ht]
	\includegraphics[scale=0.35]{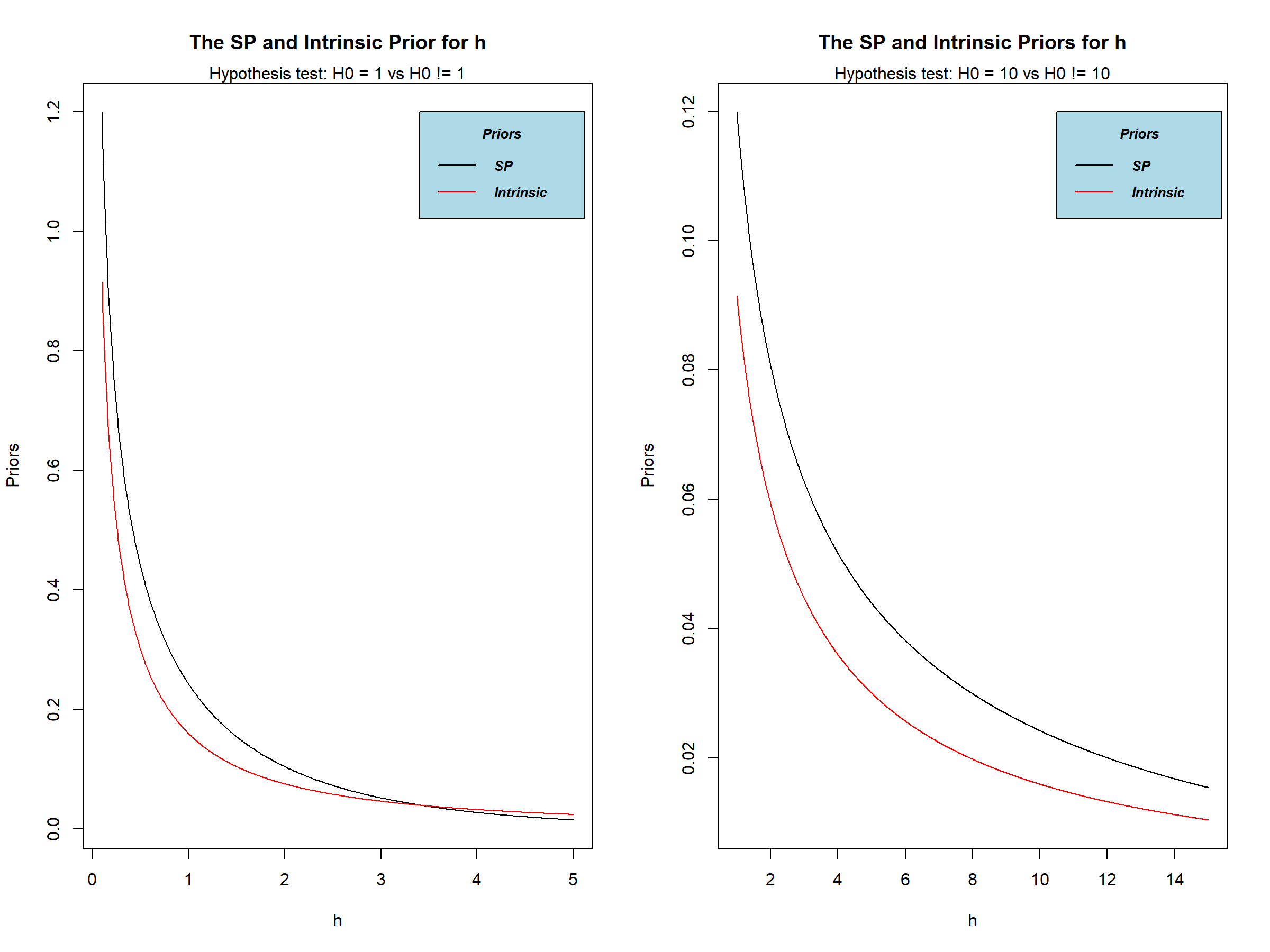}
	\caption{A comparison between the SP Prior and the Intrinsic Prior for the Normal Scale Hypothesis test for different null hypothesis values.}\label{spvsip}
\end{figure}
\newpage
As illustrated in the preceding figures, a noteworthy characteristic of the Superior Posterior (SP) Prior emerges. This prior assigns a heightened probability to the null hypothesis in comparison to the Intrinsic Prior, specifically when the values of $h$ are less than that of the null hypothesis. It is furthermore observed that the probability ascribed by the SP is invariably equal to or exceeds that of the Intrinsic Prior under these conditions. This pattern elucidates the differential weights that these priors place on the null hypothesis, offering valuable insights for our investigation. \\ 
\section{The Normal Mean Hypothesis Test}
\hspace{10pt} This section is dedicated to exploring the domain of Normal Mean Hypothesis Testing. Here, we uncover the core principles and methodologies that drive inference regarding the mean parameter within the Normal Distribution. Assume that $y_i \sim N(\mu,\sigma_0^2)$ and consider the following hypothesis test
$$
H_0: \mu=\mu_0 \mbox{ vs } H_1: \mu \neq \mu_0, \sigma_0, \mbox{ known }
$$
Hence, the non-informative prior is $\pi^{N}(\mu) = 1$ and a minimal training sample is only one observation $y(\ell)$, for any $\ell =1,...,n$. We will use that information to calculate different Bayes factors and their bounds in the next sections. 
\subsection{The Expected Intrinsic Bayes Factor for the Normal Mean Hypothesis Test}
Consider the general formula for the Expected Intrinsic Bayes Factor (EIBF)
$$\boxed{B_{10}^{EI}(\textbf{y}) =  B_{10}^{N}(\textbf{y})E^{\hat\theta}_{\textbf{y}(\ell)|H_0}( \frac{m_0^{N}(\textbf{y}(\ell))}{m_1^{N}(\textbf{y}(\ell))} )} $$ 
For this specific hypothesis test, the EIBF can be obtained by solving the following integral
$$ E^{\hat\theta}_{\textbf{y}(\ell)|H_0}( \frac{m_0^{N}(\textbf{y}(\ell))}{m_1^{N}(\textbf{y}(\ell))} ) = \int_{-\infty}^{\infty}\Big(\frac{1}{\sqrt{2\pi}\sigma_0}\Big)^2\exp\{-\frac{1}{2\sigma_0^2}(\textbf{y}(\ell) - \mu_0)^2\}\exp\{-\frac{1}{2\sigma_0^2}(\textbf{y}(\ell) - \mu)^2\}d\textbf{y}(\ell)$$
This integral can be easily solved to obtain the following equality
$$ E^{\hat\theta}_{\textbf{y}(\ell)|H_0}( \frac{m_0^{N}(\textbf{y}(\ell))}{m_1^{N}(\textbf{y}(\ell))} ) = N(\mu|\mu_0,2\sigma_0^2)  = \frac{1}{\sqrt{2\pi}\sqrt{2\sigma_0^2}}\exp\{\frac{-1}{4\sigma_0^2}(\mu - \mu_0)^2\}$$
In this case, the parameter $\mu$ is unknown and we should estimate it using the maximum likelihood estimator $\bar{\textbf{y}}$ to obtain
$$E^{\hat\theta}_{\textbf{y}(\ell)|H_0}( \frac{m_0^{N}(\textbf{y}(\ell))}{m_1^{N}(\textbf{y}(\ell))} )  = \frac{1}{\sqrt{2\pi}\sqrt{2\sigma_0^2}}\exp\{\frac{-1}{4\sigma_0^2}(\bar{\textbf{y}} - \mu_0)^2\}$$

After finding this term, we can proceed to calculate the Expected Intrinsic Bayes Factor 
$$ B_{10}^{EI} = \frac{1}{\sqrt{2\pi}\sqrt{2\sigma_0^2}}\exp\{\frac{-1}{4\sigma_0^2}(\bar{\textbf{y}} - \mu_0)^2\}  \frac{\sqrt{2\pi}\frac{\sigma_0}{\sqrt{n}}}{\exp\{\frac{-1}{2\sigma_0^2}n(\bar{\textbf{y}} - \mu_0)^2\}}$$ 
Simplifying the equation above yields
$$ \boxed{ B_{10}^{EI} = \frac{1}{\sqrt{2n}}\exp\{\frac{1}{2\sigma_0^2}(n - \frac{1}{2})(\bar{\textbf{y}} - \mu_0)^2 \} }$$ 

\subsection{The Intrinsic Bayes Factor Bounds for the Normal Mean Hypothesis Test}
Let us consider a scenario where the random vector $\mathbf{Y} = (Y_1, \ldots, Y_n)$ is composed of independent and identically distributed (i.i.d) random variables $Y_i$, following a normal distribution with mean $\mu_0$ and standard deviation $\sigma_0$ (known) under the null hypothesis $M_0$, and a normal distribution with mean $\mu$ and standard deviation $\sigma_0$ (known) under the alternative hypothesis $M_1$. We adopt the non-informative prior, specifically $\pi_1^{N}(\mu) = {1}$.

$$ B_{01}^{N}(\textbf{y}(\ell)) = \frac{\frac{1}{\sqrt{2\pi}\sigma_0}\exp\{\frac{-1}{2\sigma_0^2}(\textbf{y}(\ell) - \mu_0)^2\}}{
	\displaystyle\int_{-\infty}^{\infty}\frac{1}{\sqrt{2\pi}\sigma_0}\exp\{\frac{-1}{2\sigma_0^2}(\textbf{y}(\ell) - \mu)^2\}d\mu}$$
From the equation above, it is evident that the maximum value of the Bayes factor is obtained when $\textbf{y}(\ell) = \mu_0$, and the maximum value is 
$$ \boxed{\sup_{\textbf{y}(\ell)\in D}B_{01}^{N}(\textbf{y}(\ell)) = \frac{1}{\sqrt{2\pi}\sigma_0}}.$$ 
By using the result above, we can compute the Bayes factor bound
\begin{align*}
	\overline{B}_{10}^{I} &= B_{10}^{N}(\textbf{y})\times \sup B_{01}^{N}(\textbf{y}(\ell))\\
	&= \frac{\displaystyle\int_{0}^{\infty}(\frac{1}{\sqrt{2\pi}\sigma_0})^{n}\exp\{-\frac{1}{2\sigma_0^2}\sum[(\textbf{y}(\ell) - \bar{\textbf{y}})^2 + n(\bar{\textbf{y}} - \mu)^2]\}d\mu}{\displaystyle(\frac{1}{\sqrt{2\pi}\sigma_0})^n\exp\{-\frac{1}{2\sigma_0^2}\sum(\textbf{y}(\ell) - \mu_0)^2\}}\times \sup B_{01}^{N}(\textbf{y}(\ell)) \\
	&= \frac{\sqrt{2\pi}\frac{\sigma_0}{\sqrt{n}}}{\exp\{\frac{-1}{2\sigma_0^2}n(\bar{\textbf{y}} - \mu_0)^2\}} \sup B_{01}^{N}(y(\ell)) = \frac{\sqrt{2\pi}\frac{\sigma_0}{\sqrt{n}}}{\exp\{\frac{-1}{2\sigma_0^2}n(\bar{\textbf{y}} - \mu_0)^2\}}\frac{1}{\sqrt{2\pi}\sigma_0}
\end{align*}
Simplifying the equation above yields
$$\boxed{ \overline{B}_{10}^{I}= \frac{\exp\{\frac{1}{2\sigma_0^2}n(\bar{\textbf{y}} - \mu_0)^2\}}{\sqrt{n}} }$$
Similarly, we can obtain the IBF lower bound
$$\boxed{ \underline{B}_{01}^{I}= {\sqrt{n}}{\Big(\exp-\{\frac{1}{2\sigma_0^2}n(\bar{\textbf{y}} - \mu_0)^2\}}\Big)}.$$

\subsection{Normal mean hypothesis test with unknown variance.}

Let us consider a scenario where the random vector $\mathbf{Y} = (Y_1, \ldots, Y_n)$ is composed of independent and identically distributed (i.i.d) random variables $Y_i$, following a normal distribution with mean $\mu_0$ and standard deviation $\sigma_0$ under the null hypothesis $M_0$, and a normal distribution with mean $\mu$ and standard deviation $\sigma_1$ under the alternative hypothesis $M_1$.
$$H_0: \mu=\mu_0 \mbox{ vs } H_1: \mu \neq \mu_0, \sigma \mbox { unknown}$$
We adopt non-informative priors, specifically $\pi_0^{N}(\sigma_0) = {1}/{\sigma_0}$ and $\pi_1^{N}(\mu,\sigma_1) = {1}/{\sigma_1^2}$. It is important to note that the choice of $\pi_1^N$ differs from that of $\pi_0^N$ due to a seminal work by Berger, Pericchi, and Varshavsky in 1998 \cite{berger1998bayes}, wherein they calculated the marginal of $\textbf{y}(\ell)$ using $\pi_1^{N}(\mu,\sigma_1) = {1}/{\sigma_1}$. In this context, the expression for the density function under the alternative hypothesis takes the form:
$$ m_1(\textbf{y}(\ell)) = \frac{1}{2|y_1-y_2|}$$ 
where $\mathbf{y}(\ell)$ represents the observed data.
However, we face a challenge in evaluating the Bayes Factor due to the nature of the above density function. Specifically, the infimum and supremum of this function can potentially become either zero or infinity, leading to indeterminate bounds for the Bayes Factor. To circumvent this issue and ensure well-defined bounds, we opt for the aforementioned choice of priors. This strategic choice of priors enables us to derive meaningful marginals and subsequently evaluate the Bayes Factor.
$$ m_0^N(\textbf{y}(\ell)) = \frac{1}{2\pi(y_1^2 + y_2^2)}, \hspace{10pt} m_1^{N}(\textbf{y}(\ell)) = \frac{1}{\sqrt{\pi}(y_1 - y_2)^2}$$
and the Bayes factor is
$$B_{01}(\textbf{y}(\ell)) = \frac{m_{0}(\textbf{y}(\ell))}{m_{1}(\textbf{y}(\ell))} = \frac{(y_1 - y_2)^2}{2\sqrt{\pi}(y_1^2 + y_2^2)} $$
To find the IBF bounds, we found the supremum of the Bayes factor above
$$\sup_{\textbf{y}(\ell) \in D}B_{01}(\textbf{y}(\ell)) = \frac{m_{0}(\textbf{y}(\ell))}{m_{1}(\textbf{y}(\ell))} = \frac{(y_1 - y_2)^2}{2\sqrt{\pi}(y_1^2 + y_2^2)} = \frac{1}{\sqrt{\pi}} $$
The IBF upper bound can be obtained as
$$\overline{B}_{10} = B_{10}(\textbf{y})\sup B_{01}(\textbf{y}(\ell)) = \frac{m_1(\textbf{y})}{m_0(\textbf{y})}\frac{1}{\sqrt{\pi}}$$
To find the upper bound, we must calculate the marginals $m_0$ and $m_1$. To obtain $m_0$, we must solve the following integral
$$m_0(\textbf{y})) = \int_{0}^{\infty} (\frac{1}{\sqrt{2\pi}})^n \exp\{\frac{-1}{2\sigma_0^2}[\sum(y_i - \mu_0)^2] \} \frac{1}{\sigma_0^{(n+1)}} d\sigma_0$$
Consider the change of variable $\beta = \frac{1}{2}[\sum(y_i - \mu_0)^2], \frac{1}{\sigma^2} = \eta$, then
$$\frac{-2}{\sigma^3}d\sigma = d\eta \Longrightarrow d\sigma = -\frac{\sigma^{3}}{2}d\eta$$
$$ \frac{1}{\eta^{1/2}} = \sigma \Longrightarrow d\sigma = -\eta^{3/2}/2d\eta $$
$$\boxed{m_0(\textbf{y}) = \frac{1}{2}(\frac{1}{\sqrt{2\pi}})^n \int_{0}^{\infty}\exp\{-\beta\eta\} {\eta^{n/2 -1}}d\eta = \frac{1}{2}(\frac{1}{\sqrt{2\pi}})^n \frac{\Gamma(n/2)}{\beta^{n/2}}}$$
\\
Similarly, we can calculate $m_1$
$$
m_1(\textbf{y}) = 
(\frac{1}{\sqrt{2\pi}})^n
\int_{0}^{\infty}
\frac{1}{\sigma_1^{n+2}}
\int_{-\infty}^{\infty}
\exp\{-\frac{1}{2\sigma_1^2}(\sum(y_i - \mu)^2)) \}
d\mu d\sigma_1
$$
Adding and subtracting $\bar{\textbf{y}}$ inside the exponential function yields
$$ 
\sum(y_i - \mu)^2 = 
\sum( (y_i + \bar{\textbf{y}}) - (\mu + \bar{\textbf{y}}) )^2
=
\sum (y_i - \bar{\textbf{y}})^2 + n(\mu - \bar{\textbf{y}})^2
$$
Inserting this result inside the integral with respect to $\mu$ results in
\begin{align*}
	\int_{-\infty}^{\infty}
	\exp\{-\frac{1}{2\sigma_1^2}(\sum(y_i - \mu)^2)) \}
	d\mu
	&= 
	\exp\{\frac{-1}{2\sigma_1^2}\sum(y_i - \bar{y})^2\}\int_{-\infty}^{\infty}\exp\{\frac{-1}{2(\sigma_1/\sqrt{n})^2}(\mu - \bar{\textbf{y}})^2\}
	d\mu\\
	&= \sigma_1\sqrt{2\pi/n}\exp\{-\frac{1}{2\sigma_1^2}(\sum(y_i - \bar{\textbf{y}})^2)\}
\end{align*}
To calculate $m_1$, we insert this result to solve the integral with respect to $\sigma_1$
$$
\frac{1}{\sqrt{n}}(\frac{1}{\sqrt{2\pi}})^{n-1}
\int_{0}^{\infty}\frac{1}{\sigma_1^{n+1}}\exp\{-\frac{1}{2\sigma_1^2}(\sum(y_i - \bar{y})^2)\}d\sigma_1
$$
Let $\beta = \frac{\sum(y_i - \bar{y})^2}{2}$ and $\eta = \frac{1}{\sigma_1^2}$, then 
$$
\frac{1}{2\sqrt{n}}(\frac{1}{\sqrt{2\pi}})^{n-1}
\int_{0}^{\infty}
\eta^{n/2 - 1}
\exp\{-\beta\eta\}
d\eta
$$
This integral is the kernel of a Gamma distribution which results in
$$
\boxed{m_1(\textbf{y}) = \frac{1}{2\sqrt{n}}(\frac{1}{\sqrt{2\pi}})^{n-1}
	\frac{\Gamma(n/2)}{(\frac{\sum(y_i - \bar{y})^2}{2})^{n/2}}}
$$
After obtaining $m_0$ and $m_1$, we can calculate the Bayes Factor upper bound
$$\boxed{\overline{B}_{10} =  \sqrt{2/n}(\frac{\sum(y_i - \mu_0)^2}{\sum(y_i - \bar{y})^2})^{n/2}}.$$
Similarly, the Bayes Factor Lower bound is
$$\boxed{\underline{B}_{01} = \sqrt{n/2}(\frac{\sum(y_i - \bar{y})^2}{\sum(y_i - \mu_0)^2})^{n/2}.}$$

\chapter{The Exponential Model Hypothesis Test}

In the exploration of statistical hypothesis testing, this chapter embarks on a comprehensive investigation into exponential hypothesis testing. Specifically, we delve into the comparison of two hypotheses $H_0: \lambda = \lambda_0$ and $H_1: \lambda \neq \lambda_0$ where $\lambda$ is the parameter for an exponential density $f(y|\lambda) = \lambda e^{-\lambda y}$. We will employ several robust methodologies to asses this hypothesis test from different but consistent approaches.

Our focus centers on evaluating and contrasting several pivotal methodologies: the EP Approach, the $-ep\log(p)$ lower bound, the Universal Robust Bound for the Intrinsic Bayes Factor, and the SP Approach. These methodologies serve as our guiding compasses in the pursuit of discerning their effectiveness in hypothesis validation. We rely on simulations to shed light on the performance and consistency of these methodologies. By employing these simulations, we aim to craft illustrative charts that will serve as visual aids, elucidating the comparative efficacy and nuances of each approach. This chapter unfolds as a meticulous exploration, aiming not only to identify the most effective methodology but also to unravel the intricacies underlying their application.

\hspace{10pt}The exponential density is defined as
\[
f(y|\lambda) = {\lambda}e^{-\lambda y}, \mbox{ where } y \ge 0 \mbox{ and } \lambda > 0 
\]
The likelihood, Jeffrey's prior, and posterior distribution for an exponential density with parameter $\lambda$ are given by 
$$
L(\lambda|\textbf{y}) = {\lambda^n}\exp(\displaystyle -\lambda \sum_{i=1}^{n}y_i), \hspace{5pt} \pi^N(\lambda) \propto \frac{1}{\lambda},\hspace{5pt}
\pi(\lambda|\textbf{y}) = \frac{L(\lambda|\textbf{y})\pi_J(\lambda)}{m(\textbf{y})}
$$
After calculating the likelihood and Jeffrey's prior for the exponential density we can easily obtain the following marginal
\[
m(\textbf{y}) \propto \int_{0}^{\infty} {\lambda^{(n-1)}}\exp(-\displaystyle{\lambda\sum_{i=1}^{n}y_i})d\lambda = \displaystyle\frac{(\sum_{i = 1}^{n}y_i)^{n}}{\Gamma(n)}
\]
Finally, we can obtain the posterior distribution for $\lambda$
\[
\boxed{\pi(\lambda|\textbf{y}) \propto
	\frac{\displaystyle{\Gamma(n)}{\lambda^{n-1}}\exp \Big(-{\lambda}{\sum_{i=1}^{n} y_i}\Big)}{{(\sum_{i = 1}^{n}y_i)^{n} }}
}
\]
In the following sections, we'll tackle a hypothesis test:
$$H_0: \lambda = \lambda_0 \mbox{ vs. } H_1: \lambda \neq \lambda_0$$We'll use different methods to calculate Bayes Factors, bounds, and priors. This helps us compare traditional approaches with our new methods, giving us a clearer understanding of their outcomes.

\section{The Intrinsic Bayes Factor Dynamic and Universal Lower Bounds}
\hspace{10pt}In this section, we're computing our new bound for the Intrinsic Bayes Factor. To derive this bound, we start by examining the Bayes factor for the minimal training sample:
$$B_{01}^{N}(y(\ell)) = \frac{\mathbf{y(\ell)}\lambda_0}{\exp^{\lambda_0\mathbf{y(\ell)}}}$$
Additionally, we need the full Bayes factor for this test
$$B_{01}^{N}(\textbf{y}) = \frac{\lambda_0^n \exp(-\lambda_0 \sum y_i)}{\int_{0}^{\infty}{\lambda^n \exp(-\lambda \sum y_i)}\frac{1}{\lambda}d\lambda} = \frac{\lambda_0^n \exp(-\lambda_0 \sum y_i)(\sum y_i)^n}{\Gamma(n)}
$$
By minimizing the first Bayes factor, we obtain the Empirical IBF lower bound
$$ \boxed { \underline{B}^{I^{*}}_{01} = \frac{(\sum
		y_i)^{n}\lambda_{0}^{n-1}}{\textbf{y}^{*}(\ell)\Gamma(n)}\exp(\lambda_0[ y^{*}(\ell) - \sum y_i]),  \hspace{5pt}y^{*}(\ell) = \arg \displaystyle\min_{y_i \in \{y_1,...,y_n\}}\Big\{\frac{\mathbf{y_i}\lambda_0}{\exp^{\lambda_0\mathbf{y_i}}}\Big\}}$$

Our next goal is to obtain the theoretical IBF bound, and to achieve this we need to find $\displaystyle \sup_{y(\ell) \in D}B^{N}_{01}(y(\ell))$. This maximization can be attained by solving the following equation
\begin{align*}
	\frac{d}{dy(\ell)}[\log(B_{01}^{N}(y(\ell))] &= \frac{d}{dy(\ell)}[\log(y(\ell)) + \log(\lambda_0) - \lambda_0 y(\ell)] &= 0\\
	&= \frac{1}{y^{*}(\ell)} + \lambda_0 &= 0\\
	&y^{*}(\ell) = \frac{1}{\lambda_0}
\end{align*}
To verify that this is a maximum, we take the second derivative with respect to $y(\ell)$
$$\frac{d^2}{dy(\ell)}[\log(B_{01}^{N}(y(\ell))] = \frac{-1}{y(\ell)^2} < 0, \forall \mbox{ }y(\ell) \in \mathbb{R}$$
Therefore, $y^{*}(\ell)$ the supremum in $(0,\infty)$ and $B_{01}^{N}(y^{*}(\ell)) = 1/e$. By using these results, we can obtain the following Intrinsic Bayes Factor upper bound 
$$\boxed{\overline{B}^{I}_{10} = 1/e \frac{\Gamma(n)\exp(\lambda_0\sum y_i)}{(\sum y_i)^{n}\lambda_{0}^{n}}}.$$
Similarly, the IBF lower bound is given by
$$\boxed{ \underline{B}_{01}^{I} = \frac{e(\sum
		y_i)^{n}\lambda_{0}^{n}}{\Gamma(n)\exp(\lambda_0\sum y_i)}}$$
\section{The SP Bayes Factors}
In this section, our focus lies in computing the SP priors and the SP Bayes factor for the exponential hypothesis test. We aim to contrast the outcomes from the SP approach with the intrinsic Bayes factor bounds derived earlier. To bolster our findings, we'll conduct simulations and generate illustrative charts. These visual aids will fortify and elucidate our conclusions.

\hspace{10pt}The theoretical SP Priors for the exponential model for $B_{10}$ are given by
$$ \boxed{\pi_0^{SP}(\lambda) =  1, \hspace{5pt} \pi_1^{SP}(\lambda) =  \frac{1}{\lambda_0}\exp^{-\lambda/\lambda_0},\hspace{5pt} \lambda > 0}$$
It's noteworthy that $\pi_1^{SS}(\lambda)$ follows an exponential distribution with the parameter $1/\lambda_0$. 
Let's assume, without loss of generality, that $y_1 = 1/{\lambda_0}$. The SP marginals for the hypothesis test $H_0: \lambda = \lambda_0 \text{ versus } H_1: \lambda \neq \lambda_0$ are
\begin{align*}
	m^{SP}_0(\textbf{y}) &\propto \lambda_0^{n} e^{-\lambda_0\sum_{i=1}^{n} y_i} \\
	&= \lambda_0^{n} e^{-\lambda_0\sum_{i=1}^{n}y_i}
\end{align*}
\begin{align*}
	m^{SP}_1(\textbf{y}) &\propto \frac{1}{\lambda_0}\int_{0}^{\infty}\lambda^{n}e^{-\lambda (\sum_{i=1}^{n} y_i + \frac{1}{\lambda_0})} \\
	&= \frac{1}{\lambda_0}[\Gamma(n+1)(\sum_{i=1}^{n} y_i + \frac{1}{\lambda_0})^{(-(n + 1))}] \\
	&= \frac{1}{\lambda_0}[\Gamma(n+1)(\sum_{i=1}^{n} y_i+ \frac{1}{\lambda_0})^{-(n+1)}]
\end{align*}
Utilizing these marginals leads to the results of the SP Bayes factor as follows:
$$ \boxed{B^{SP}_{10} =   
	\frac{\Gamma(n+1)e^{\lambda_0\sum _{i=1}^{n}y_i}}{(\sum_{i=1}^{n} y_i + \frac{1}{\lambda_0})^{n+1}\lambda_0^{n+1}} }$$
In order verify the validity of our calculations, and to compare the numerical values of $B^{SP}_{10}$ with $\overline{B}^{I}_{10}$, a simulation was conducted and the corresponding chart was obtained. The simulation was implemented using the following code snippet:\\ \\
\textbf{R Program}
\begin{lstlisting}
	b10_theoretical = function(n,h0,lambda=lambda0) { 
		y = c(); b10_t = c(); b10_sp = c(); s = c(); e = exp(1)
		for (i in seq(1,n,1)) {
			y = c(y,rexp(1,rate=1/lambda))
			s = sum(y)
			val = -log(e,e) + - n*log(s,e) - n*log(h0,e) + log(gamma(n),e)+ h0*s
			b10_t = c(b10_t,val)  
			val = log(gamma(n+1),e) + h0*s - (n+1)*log(s + 1/h0,e) - (n+1)*log(h0,e)
			b10_sp = c(b10_sp,val)
		}
		plot(seq(1,n,1),b10_t,type="l",col="red",
		xlab="Number of samples",
		ylab="log(BF10)", ylim=c(min(b10_t,b10_sp),max(b10_t,b10_sp)),
		main="Theorical SS Bayes Factor (Red) vs SP Bayes Factor (Blue) (H0 TRUE)")
		lines(b10_sp,col="blue")
	}
\end{lstlisting}
\begin{figure}[H]
	\centering
	\includegraphics[width=0.9\textwidth]{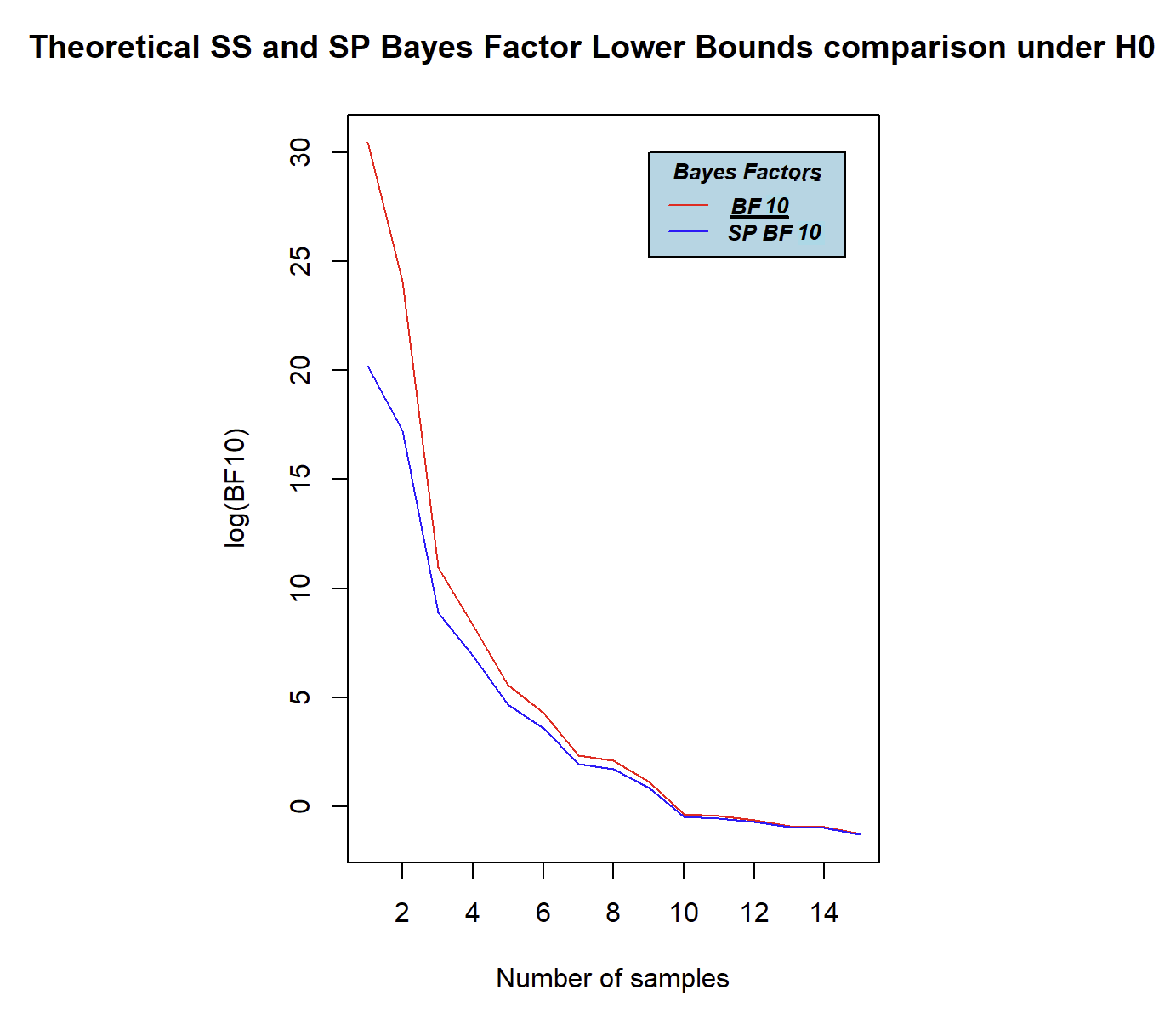}
	\caption{$\overline{B}^{SS}_{10}$ vs $B^{SP}_{10}$ under $H_0$ for increasing sample size and exponential data.}
	\label{fig:expspbf}
\end{figure}
This comparison serves to analyze the numerical relationship between ${B}^{SP}_{10}$ and $\overline{B}^{I}_{10}$, thereby contributing to the evaluation and understanding of their respective characteristics. The obtained chart provides a visual representation of the convergence of the $\overline{B}^{I}_{10}$ to ${B}_{10}^{SP}$ as the number of samples increases.
\newpage
\section{The EP Bayes Factors}
\hspace{10pt}The exponential Expected posterior prior under the alternative (EP-Prior) for the exponential hypothesis test $H_0: \lambda = \lambda_0 \mbox{ vs } H_1: \lambda \neq \lambda_0$ can be calculated as
$$ \pi_1^{EP}(\lambda) = \frac{\displaystyle\sum_{i=1}^{L}\pi(y(\ell)|\theta_i)}{n} = \frac{\displaystyle\sum_{i=1}^{n}y_ie^{-\lambda y_i}}{n}$$
It can be verified that the prior in question integrates unity with respect to $\lambda$ by performing integration overall positive values of $\lambda$. Additionally, the SP prior, which is an exponential distribution with a parameter of $1/{\lambda_0}$, also integrates one.

The objective of this study is to compare the SP prior with the EP prior under different sample sizes, specifically for the cases of $n=10$ and $n=100$. To achieve this, simulations were conducted utilizing the following R code:\\
\textbf{R Program}
\begin{lstlisting}[language=R]
	prior = function(data,lambda0,lambda) {
		data = c(); object = c(); lambda0 = 0; ep1 = c(); ss1 = rep(0,length(lambda)); ss1 = (1/lambda0)*exp(-lambda/lambda0)
		for (j in 1:length(lambda)) {
			epnew = 0
			for (i in 1:length(data)) { epnew = epnew + data_[i]*exp(-data[i]*lambda[j]) }
			ep1 = c(ep1,epnew)
		}
		object$ss = ss1; object$ep = ep1/length(data)
		return(object)
	}
	data = rexp(n=100,r=1);lambda0=1; lambda = seq(0.1,2,0.1)
	object = prior(data = data,lambda0=lambda0,lambda=lambda);
	ss = exp(-lambda); ep = object$ep
	plot(lambda,ep,col="red",ylim=c(min(ss,ep),max(ss,ep)),type="b",ylab="Prior",
	main="EP Prior under H1 (Red) / SP Prior under H1 (Blue)"; 
	lines(lambda,ss,col="blue")
\end{lstlisting}
\hspace{10pt} By employing these simulations, we aim to investigate and evaluate the performance of the SP prior relative to the EP prior in the context of various sample sizes. The obtained results will contribute to enhancing our understanding and analysis of the behavior and characteristics of these priors.
\newpage
\begin{figure}[ht]
	\centering
	\subfigure[\centering Exponential simulation with 10 samples for $\lambda_0 = 1$]{{\includegraphics[width=8cm]{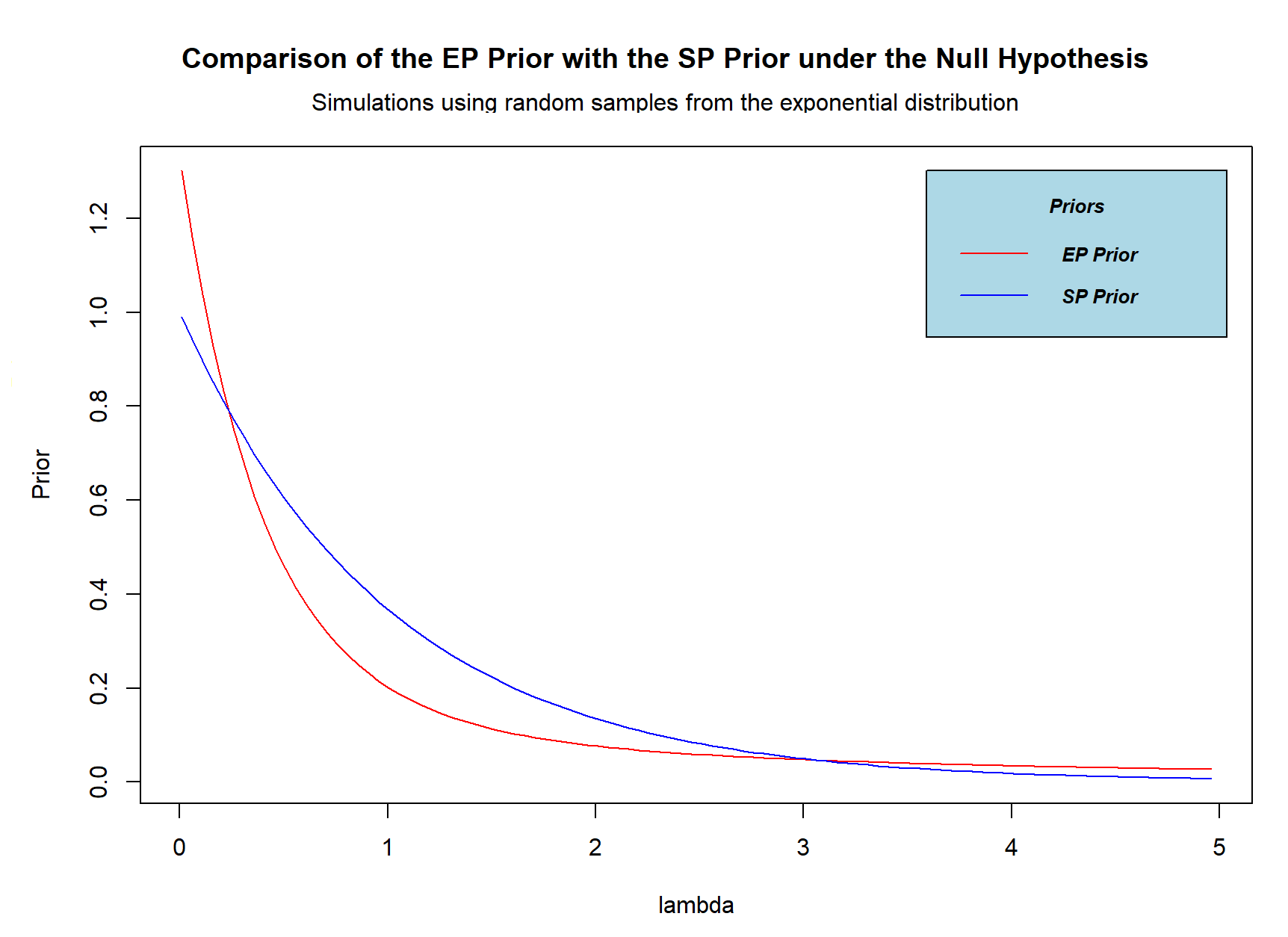} }}%
	\qquad
	\subfigure[Exponential simulation with 100 samples for $\lambda_0 = 1$]{{\includegraphics[width=8cm]{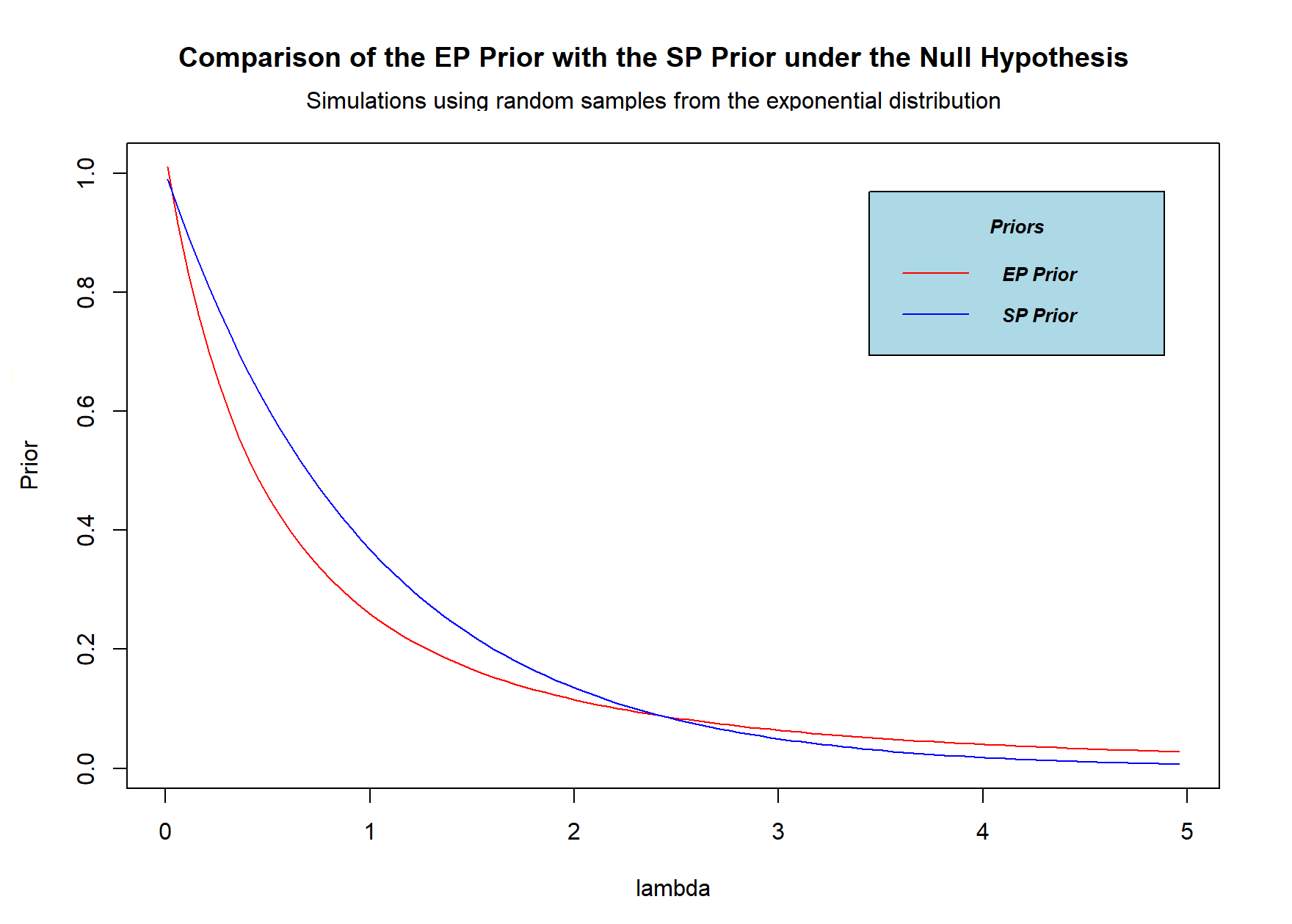} }}%
	\caption{Simulations with different number of samples for the exponential SP and EP Priors for $\lambda_0=1$}%
	\label{fig:example}%
\end{figure}
The exponential EP-Bayes factor for $\lambda = \lambda_0 \mbox { vs } \lambda \neq \lambda_0$ can be obtained by calculating the following expression
$$ B_{10}^{EP}(\textbf{y}) = \frac{m_{1}^{EP}(\textbf{y})}{m_{0}^{EP}(\textbf{y})} = \frac{\frac{1}{n}\displaystyle\int_{0}^{\infty} \lambda^{n}\exp\{-\lambda\sum_{i=1}^{n} y_i \}\sum_{i=1}^{n}y_i\exp\{-\lambda y_i\} d\lambda}{\lambda_0^{n}\exp\{-\lambda_0\displaystyle\sum_{i=1}^{n} y_i \}} $$
Our objective is to resolve the integral mentioned earlier. To facilitate this, let's define $S = \displaystyle\sum_{i=1}^{n}y_i$. Subsequently, the integral transforms into:
$$ \displaystyle\sum_{i=1}^{n}y_i\int_{0}^{\infty}\lambda^{n}\exp\{-\lambda(S + y_i)\}d\lambda$$
Consider the change of variable $\alpha = n + 1$ and $\beta_i = \frac{1}{S + y_i}$, which leads to the transformation:
$$\displaystyle\sum_{i=1}^{n}y_i\int_{0}^{\infty}\lambda^{n}\exp\{-\lambda(S + y_i)\}d\lambda = \displaystyle\sum_{i=1}^{n}y_i\int_{0}^{{\infty}}\lambda^{\alpha - 1}\exp\{-\lambda/\beta_i\}d\lambda$$
This new form represents the kernel of a Gamma density with parameters $\alpha$ and $\beta_i$, resulting in the integral:
$$\Gamma(n+1)\sum_{i=1}^{n}y_i\Big(\frac{1}{S+ y_i}\Big)^{n+1}$$
Hence, the EP Bayes Factor, in this case, becomes:
$$ \boxed{ B_{10}^{EP}(\textbf{y}) = \frac{\frac{1}{n}\Gamma(n+1)\displaystyle\sum_{i=1}^{n}y_i\Big(\frac{1}{S+ y_i}\Big)^{n+1}}{\lambda_0^{n}\exp\{-\lambda_0\displaystyle\sum_{i=1}^{n} y_i \}} } $$
\hspace{10pt}The principal objective of this numerical experiment is to compare the EP and SP Bayes factors, previously derived, by creating 100 simulations, each consisting of 100 samples. The process entails averaging the Bayes factor for each sample size, ranging from 1 to 100. Furthermore, we intend to plot the logarithm of each Bayes factor on the same axis to discern their respective behaviors under two distinct scenarios: when the null hypothesis is valid and when it is not. The forthcoming R program has been devised to generate the anticipated output.
\newpage
\textbf{R Program}
\begin{lstlisting}[language=R]
	#Initialize vectors and variables
	finallogbf10 = c(); finallogbf10sp=c(); nsim = 100; n = 100;
	data = c(); s = 0; sumdata = 0;logbf10=c();logbf10sp=c();
	#H0 TRUE
	true_lambda = 1; lambda0 = 1;
	#Use this for H0 FALSE
	#true_lambda = 1; lambda0 = 0.3;
	for (k in 1:nsim) {
		#Generate random samples
		data =rexp(n=n,rate=true_lambda)
		for (i in 1:n) {
			s = sum(data[1:i])
			for (j in 1:i) { sumdata = sumdata + data[j]*((1/(s+data[j]))^(i+1)) }
			#EP marginals
			m1_ep = (gamma(i+1)*sumdata)/i; m0_ep = (lambda0)^i*exp(-lambda0*s)
			#Log EP Bayes factor
			logbf10 = c(logbf10,log(m1_ep) - log(m0_ep))
			#SP Marginals    
			m1_sp = gamma(i+1)*exp(lambda0*s); m0_sp = ((s + 1/lambda0)^(i+1) )*(lambda0^(i+1))
			#Log SP Bayes Factor
			logbf10sp = c(logbf10sp,log(m1_sp) - log(m0_sp))
			sumdata = 0
		}
		#Storing simulations
		finallogbf10 = cbind(finallogbf10,logbf10); finallogbf10sp = cbind(finallogbf10sp,logbf10sp)
	}
	logbf10_avg = rep(0,n); logbf10sp_avg = rep(0,n);
	#Averaging Bayes factors for all the simulations
	for (i in 1:n) {
		logbf10_avg[i] = mean(finallogbf10[i,]); logbf10sp_avg[i] = mean(finallogbf10sp[i,])
	}
	x = 1:n
	y1 = logbf10_avg[x]; y2 = logbf10sp_avg[x]
	if (lambda0 == true_lambda) { h0 =TRUE } else { h0 = FALSE }
	plot(x,y1,type="l",col="red",ylim=c(min(y1,y2),max(y1,y2)), main="log(EP Bayes Factor 10) (Red) vs log(SP Bayes Factor 10) (Black)",xlab="Sample size",ylab="Log Bayes Factor 10")
	lines(x,y2,type="l",col="black")
	mtext(paste("H0 is ",h0," / Number of simulations: ",nsim," / Number of samples: ",n," / lambda = ",true_lambda, " / lambda_0 = ",lambda0, sep=""),side=3)
	
\end{lstlisting}
\newpage
\begin{figure}[ht]
	\centering
	\includegraphics[width=1\textwidth]{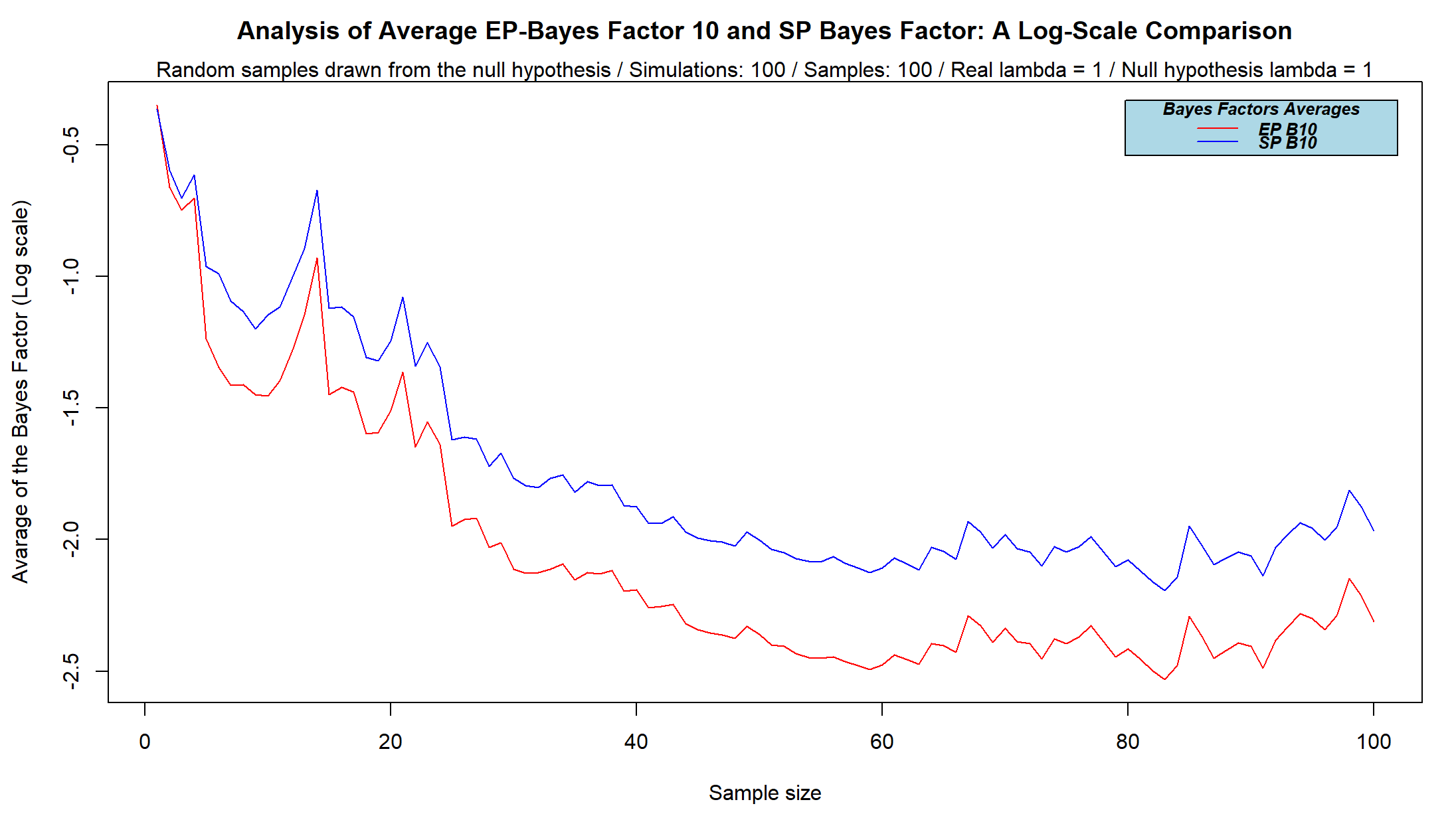}
	\caption{Logarithm of the SP and EP Bayes factor using one-hundred exponential simulations (100 samples each) for $\lambda_0 = 1, H_0$ True}
	\label{fig:SPVSEPH0}
\end{figure}
Figure \ref{fig:SPVSEPH0} demonstrates the consistency of both the SP and EP Bayes factors. The figure illustrates that both Bayes factors approach zero when calculating the Bayes factor of the alternative hypothesis over the null hypothesis. This observation is expected since the data used for the calculations are random samples generated from the null hypothesis. The decreasing trend of the Bayes factors indicates their convergence towards zero, indicating consistency in the evidence provided by both methods.
\newpage
\begin{figure}[ht]
	\centering
	\includegraphics[width=1\textwidth]{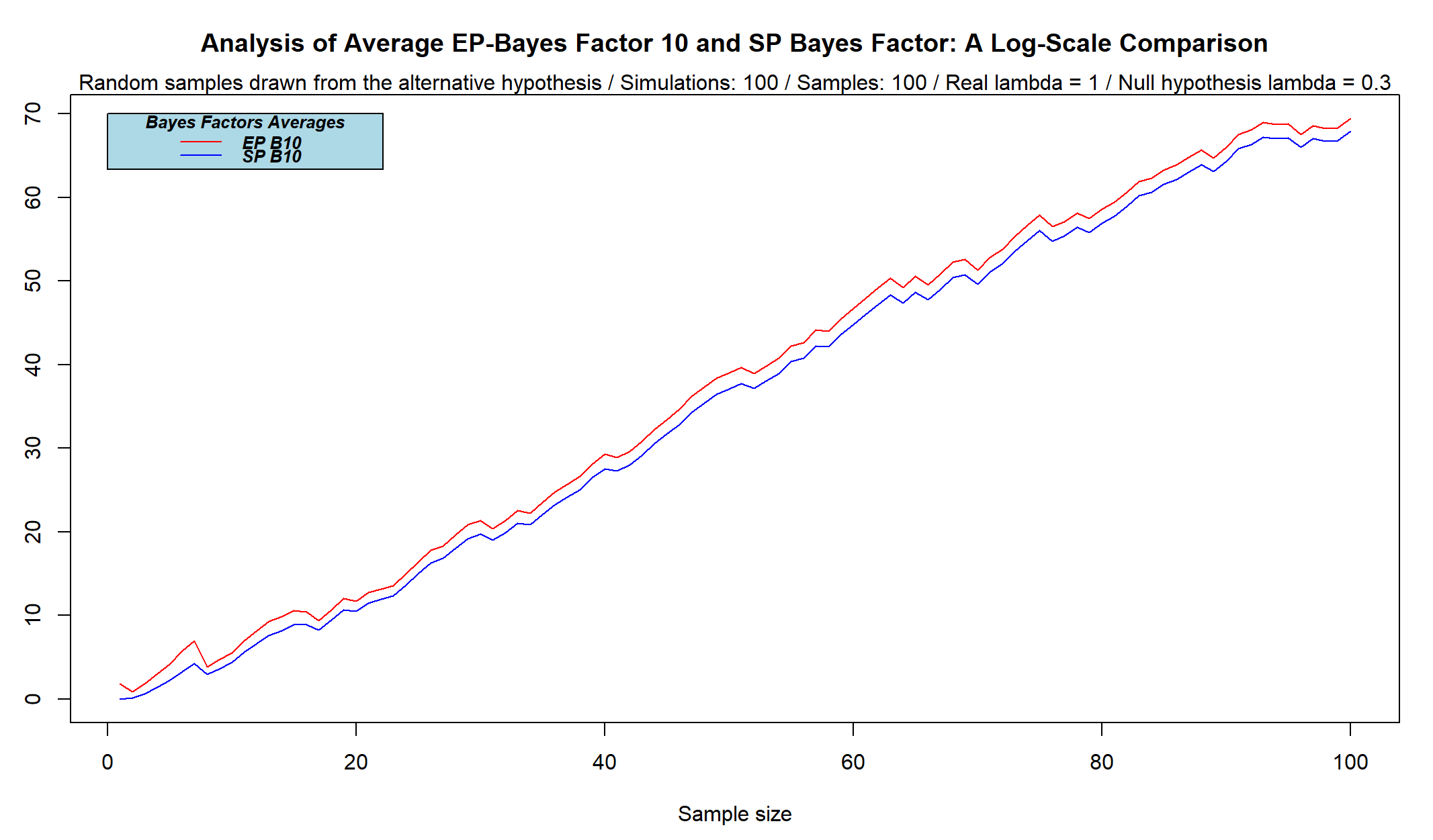}
	\caption{Logarithm of the SP and EP Bayes factor using one-hundred exponential simulations (100 samples each) for $\lambda_0 = 1, H_0$ False}
	\label{fig:SPVSEPH1}
\end{figure}
Figure \ref{fig:SPVSEPH1} demonstrates the consistency of both the SP and EP Bayes factors. The figure illustrates that both Bayes factors increase in value with the sample size when calculating the Bayes factor of the alternative hypothesis over the null hypothesis. This observation is expected since the data used for the calculations are random samples generated from the alternative hypothesis. The increasing trend of the Bayes factors indicates consistency in the evidence provided by both methods.
\newpage
\section{The -eplogp lower bound}
\hspace{15pt} In this section we revisit the Sellke, Bayarri, and Berger (2001) \cite{sellke2001calibration} lower bound for the Bayes factor that employs the classical p-value of the hypothesis test, represented as $-ep\log(p) \leq B_{01}$. We will perform a simulation to compare this this p-value approximation of the Bayes factor to contrast it with the intrinsic Bayes factor lower bound for the exponential test. In an attempt to compute the $-ep\log(p)$ lower bound of $B_{01}$, we determined the likelihood ratio for the hypotheses $H_0$ against $H_1$. The approximation of the p-value was accomplished through the utilization of Wilk's theorem and the chi-squared distribution. The graphical representation provided demonstrates the lower bound and the $-ep\log(p)$ in tandem as the sample size increases.

The forthcoming illustration elucidates how both the empirical and theoretical bounds converge toward the appropriate decision as the sample size escalates. This is demonstrated by drawing samples from an exponential distribution characterized by a known rate parameter, while simultaneously altering the null hypothesis.\newpage
\textbf{R Program}
\begin{lstlisting}[language=R]
	
	exponential_simulations=function(n_sim,n_data,h0,rate) {
		lambda = h0; sim = c(); bfbound = c(); bfssbound = c(); pval_seq= c(); epbound = c(); ssbound = c();
		for (j in 1:n_sim) {
			for (n in 1:n_data) {
				sim = c(sim,rexp(n=1,rate=rate))
				lrt_null = n*log(lambda) - lambda*sum(sim)
				lrt_alt = n*(log(n) - log(sum(sim)) - 1)
				lrt = exp(lrt_null-lrt_alt)
				ts = -2*log(lrt)
				#p-value approximation wilks theorem
				pval = 1-pchisq(ts,1)
				pval_seq = c(pval_seq,pval)
				#selke bound
				if (pval > (1/exp(1))) { bfbound = c(bfbound,1) }
				else { bfbound = c(bfbound,-exp(1)*pval*log(pval)) }
				bfssbound = c(bfssbound,(exp(1)*(sum(sim)*lambda)^n)/gamma(n)*exp(-lambda*sum(sim)))
			}
			epbound = cbind(epbound,bfbound); ssbound = cbind(ssbound,bfssbound)
		}
		epbound_avg = c(); ssbound_avg = c()
		for (i in 1:n_data) { 
			epbound_avg = c(epbound_avg,mean(epbound[i,]))
			ssbound_avg = c(ssbound_avg,mean(ssbound[i,]))
		}
		object = c(); object$epbound_avg = epbound_avg; object$ssbound_avg = ssbound_avg
		return(object)
	}  
	
	#100 Simulations and 100 Samples; H0 True
	n_sim = 100; n_data = 100; h0 = 1; rate = 1
	
	simulations = exponential_simulations(n_sim=n_sim,n_data = n_data,h0=h0,rate=rate)
	
	par(mfrow=c(1,2))
	plot(1:n_data,simulations$epbound_avg,type="l",col="red",xlab="Number of samples",ylab="Average of the Lower bounds",ylim=c(min(simulations$epbound_avg,simulations$ssbound_avg),max(simulations$ssbound_avg,simulations$epbound_avg)),
	main="Comparison between the -eplogp and IBF lower bounds")
	lines(simulations$ssbound_avg,col="blue")
	
	mtext(paste("Simulations: lambda = ",rate, " / Null hypothesis: lambda = ",h0, sep=""),side=3)
	legend(0, 10, legend=c("-eplogp", "SS"),col=c("red", "blue"), lty=c(1,1), cex=0.8,title="Lower bounds", text.font=4, bg='lightblue')
	
\end{lstlisting}
\newpage
By employing the aforementioned code, a simulation was carried out to calculate both the empirical and theoretical lower bounds of $B^{I}_{01}$ for the exponential distribution. This was executed to test the hypothesis $H_0$ against $H_1$ under two distinct circumstances: when $H_0$ holds true and when it does not. The objective was to examine the convergence of the Bayes factor lower bound toward the accurate decision as the sample size expands. The outcomes of this simulation are presented in the following figures.
\begin{figure}[ht]
	\centering
	\includegraphics[width=12.5cm]{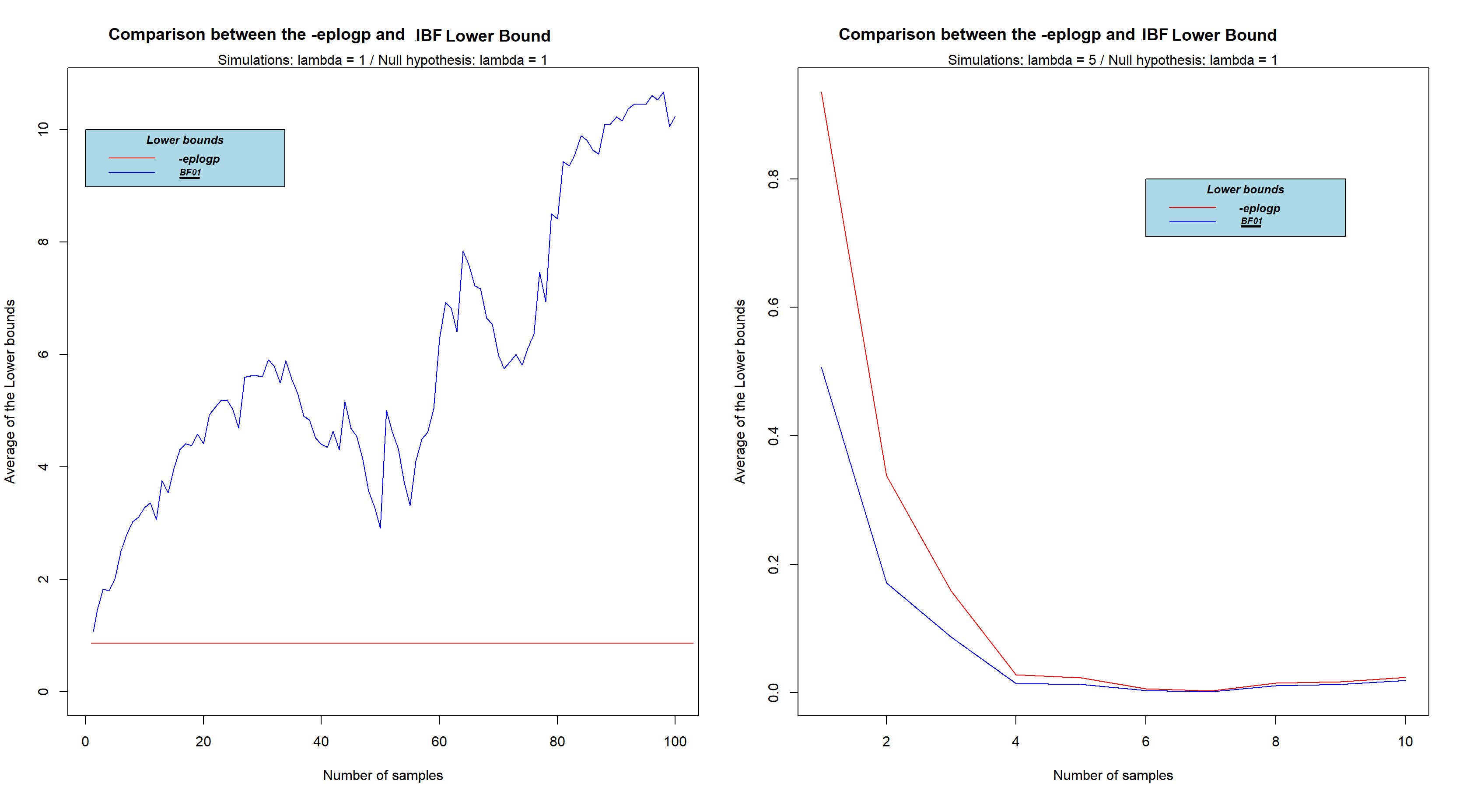}
	\caption{The $-ep\log(p)$ and $\underline{B}_{01}^{I}$  using 100 exponential simulations (100 samples each) for $\lambda_0 = 1$ (Null hypothesis true), and 100 exponential simulations with 10 samples each for $\lambda_0 = 1$ and the simulations rate parameter $r = 5$ (false null hypothesis)}
	\label{fig:ssvseph0}
\end{figure}
\\\\
In \textbf{Figure \ref{fig:ssvseph0}} we illustrated the consistency of the IBF lower bound in both scenarios. On the other hand, the second figure (right) illustrates that both Bayes factor's lower bounds approach zero when calculating the Bayes factor of the null hypothesis over the alternative hypothesis. This observation is expected since the data used for the calculations are random samples generated from the alternative hypothesis. The decreasing trend of the lower bounds indicates their convergence towards zero, indicating consistency in the evidence provided by both methods under the alternative, however, the $-ep\log(p)$ bound does not provide decisive evidence to accept the null hypothesis when we draw random samples from it, because the evidence is bounded by one.

\newpage
\section{Conclusions}

\hspace{10pt}Our research indicates that the Theoretical SP Bayes Factor, derived using the SP-Priors, is congruent with the Intrinsic Bayes factor Bound $\overline{B}^{I}_{10}$, albeit with the inclusion of an additional observation being the supremum of $B^{N}_{01}(y(\ell))$. We demonstrated that the SS-prior yields comparable results and that this methodology facilitates the generation of a lower bound for the theoretical SS Bayes factor.

As the sample size increases, the discrepancy between these two methods diminishes, given that we are more likely to encounter samples closer to the supremum (under $H_0$). In the context of comparing the EP and SP Priors for $H_1$, it becomes evident that the Theoretical SP prior remains unchanged for any sample, while the EP-prior fluctuates in accordance with the samples.

Moreover, it can be inferred from the simulations and the chart that the SP prior assigns a higher probability to the null hypothesis $\lambda_0 = 1$, owing to its wider area from 0 to 1 compared to the EP-prior for both $n=10$ and $n=100$. Both the EP-Bayes factor and the SP-Bayes factors exhibit similar patterns of convergence and values. However, the SP-Bayes factor presents higher values under the null and lower values under the alternative.

The consistencies observed in the results validate the results of the SP-Bayes factors and confirm the accuracy of the computations performed in this study.
\chapter{Linear regression and ANOVA Hypothesis Testing with Bayes Factors}
In the realm of statistical modeling, the avenues of linear regression and analysis of variance (ANOVA) stand as cornerstones, offering valuable insights into relationships between variables and differences among groups. Traditional frequentist approaches have long governed these methodologies, relying on p-values and significance testing. However, the burgeoning field of Bayesian statistics introduces a paradigm shift, emphasizing coherent inference and robust evidence assessment through Bayes factors.

This chapter embarks on an exploration of linear regression and ANOVA within the Bayesian framework. Departing from traditional methods, we delve into the application of Bayes factor approximations, building upon the discussions from our previous chapters. Moreover, we extend the foundations laid by Smith and Spiegelhalter, broadening our insights by calculating these approximations through the utilization of a generalized Jeffrey's prior.

ANOVA, a stalwart technique in comparing group means, undergoes a transformation in this chapter as we explore its foundations through Bayesian inference. The comparison of multiple groups finds new light as Bayes factors unveil nuanced insights into the evidence supporting differing hypotheses.
\section{Linear regression}
Linear regression is a statistical model that is commonly utilized in several disciplines, ranging from economics to engineering. At its core, linear regression is a conditional model where the outcome variable is predicated on a linear combination of the predictor variables, in conjunction with an unobserved error term that introduces variability into the relationship between the input and output variables.

The fundamental form of a simple linear regression model is as follows:
\begin{equation}
	Y = \theta_0 + \theta_1X + \varepsilon
\end{equation}
where $Y$ is the dependent or outcome variable, $X$ is the independent or predictor variable, $\theta_0$ is the Y-intercept, which represents the expected value of $Y$ when all $X$ are 0,
$\theta_1$ is the slope of the regression line, indicating the degree to which $Y$ changes for each unit change in $X$, and
$\varepsilon$ is the error term, embodying the difference between the observed and predicted values of $Y$.

The goal of linear regression analysis is to estimate the coefficients $\theta_0$ and $\theta_1$ that minimize the sum of the squared residuals, thus providing the "best fit" line for the observed data.

In a multiple linear regression scenario, the model is expanded to include more than one independent variable:
\begin{equation}
	Y = \theta_0 + \theta_1X_1 + \theta_2X_2 + ... + \theta_pX_p + \varepsilon
\end{equation}
where $p$ represents the number of predictor variables. Here, each $\theta_i$ (for $i = 1, 2, ..., p$) signifies the change in the expected value of $Y$ for each unit change in the corresponding predictor $X_i$, holding all other predictors constant.

Linear regression models assume linearity, independence, homoscedasticity (constant variance), and normality of residuals. Violations of these assumptions may necessitate the application of more complex models or transformations of the data. Linear regression serves as a fundamental building block in understanding more intricate statistical models.

\section{Two nested normal-linear models}
Consider the comparison of two nested normal-linear models $M_0 \subset M_1$, defined as in Smith and  Spiegelhalter (1980) \cite{smith1980bayesfactors} by
$$ M_i: \textbf{y} \sim N(\textbf{A}_i\theta_i,\sigma^2\textbf{I}_n), \hspace{10pt} i = 0,1$$
where $\textbf{A}_i$ is a full rank $p_i$ known matrix, $\textbf{y}$ is a vector with dimension $n$, $\boldsymbol{\theta_i} = (\theta_{i1},...,\theta_{ip_i})$ is a vector of $p_i$ unknown parameters and $\sigma^2$ is unknown. This can be written in matrix notation as
$$ {\displaystyle \textbf{y} = \textbf{A}_i\theta_i + \epsilon_i}, i = 0,1$$
where $\boldsymbol{\epsilon}_i \sim N(\textbf{0},\sigma^2 I_n)$ for $i = 0,1$. Let
$$ \hat{\theta}_i = (A_i^T A_i)^{-1}A_i^T\textbf{y} \mbox { and } \textbf{R}_i = |\textbf{y} - A_i\hat{\theta}_i|^2$$
denote the least squares estimator of $\theta_i$ and the residual sum of squares, respectively.
The Bayes factor in this case is given by
$$B_{01} = \frac{\int \int p(\textbf{y}|\textbf{A}_0,\boldsymbol{\theta}_0,\sigma)p(\boldsymbol{\theta}_0,\sigma|\textbf{A}_0)d\boldsymbol{\theta}_0 d\sigma}{\int \int p(\textbf{y}|\textbf{A}_1,\boldsymbol{\theta}_1,\sigma)p(\boldsymbol{\theta}_1,\sigma|\textbf{A}_1)d\boldsymbol{\theta}_1 d\sigma}$$

Smith and  Spiegelhalter showed that for the case when $i=0,1$, $p(\boldsymbol{\theta_i},\sigma | \textbf{A}_i)$ has an improper limit form representing vague prior information for all parameters inside each linear model. The improper limiting version of the normal-inverse-$\chi^2$ conjugate prior can be written in the form
$$ p(\boldsymbol{\theta}_i,\sigma|\textbf{A}_i) = p(\boldsymbol{\theta}_i|\textbf{A}_i,\sigma)p(\sigma) = c_i(2\pi\sigma^2)^{-p_i/2}\sigma^{-1}$$
and the likelihood is given by
$$p(\textbf{y}|A_i,\theta_i,\sigma) = (\frac{1}{\sqrt{2\pi}\sigma})^{n}\exp\{-\frac{1}{2\sigma^2}\Big[\textbf{R}_i(\textbf{y}) + (\theta_i - \textbf{y})^T A_i^T A_i (\theta_i - \textbf{y})\Big]\}.$$
The Bayes factor in this case is given by:
\begin{equation}\label{generallmbf}
	B_{01} = \frac{c_0}{c_1}[|\textbf{A}_1^T \textbf{A}_1|/|\textbf{A}_0^T \textbf{A}_0|]^{\frac{1}{2}}\big[1 + \frac{(p_1 - p_0)}{(n-p_1)}F\big]^{-(n/2)}
\end{equation}
Here, $F$ represents the $F$-test statistic for comparing models $M_0$ and $M_1$. Due to the ratio of undefined constants, this Bayes factor is inherently indeterminate. The formula for the $F$-test statistic is:
$$
{\displaystyle F={\frac {\left({\frac {\textbf{R}_{0}-\textbf{R}_{1}}{p_{1}-p_{0}}}\right)}{\left({\frac {{\textbf{R}}_{1}}{n-p_{1}}}\right)}},}
$$ 
Moreover, the $F$ value used in the Bayes Factor equation can be derived from the following identity:
$$ \frac{\textbf{R}_0}{\textbf{R}_1} = 1+ \Big(\frac{p_1 - p_0}{n-p_1}\Big)F$$
Spiegelhalter and Smith proposed a satisfactory solution to the problem of determining the ratio $c_0/c_1$ by introducing the concept of an imaginary training sample. Let $\textbf{A}_0(\ell)$ and $\textbf{A}_1(\ell)$ be the design matrix of $M_0$ and $M_1$ occurring in the "thought experiment" generating the imaginary training sample, they obtained the following result
\begin{equation}\label{constants_ratio}
	\frac{c_0}{c_1} = [|\textbf{A}_1(\ell)^T\textbf{A}_1(\ell)|/|\textbf{A}_0(\ell)^T \textbf{A}_0(\ell)|]^{-\frac{1}{2}}
\end{equation} 
Our goal is to calculate the SS Bayes Factor, in general, using the generalized prior $\pi(\theta_i,\sigma) \propto {\sigma}^{-(1+q_i)}$ where different values of q will result in well-known priors such as the reference, complete and modified Jeffrey's priors. Additionally, we will calculate this Bayes factor for ANOVA, to compare it with the results obtained by Spiegelhalter and Smith (1980)\cite{smith1980bayesfactors}.

Suppose that we want to generalize the result above and compare $M_0$ vs $M_1$ where $M_0$ has $p_0$ unknown parameters, and $M_1$ has $p_1$ unknown parameters. Assume that our prior has the general form 
$$\pi(\theta_i,\sigma
|A_i) \propto \sigma^{-(1+q_i)}, i = 0,1$$
where there are many possible choices for $q_i$. According to Berger and Pericchi (1996) \cite{berger1996intrinsic}, $q_i = 0$ is the reference prior (example above), and $q_i = p_1 - p_0$ is the modified Jeffrey's prior.
In general, the marginal distribution for $n$ samples and $i = 0,1$ is given by
$$ m_i(\textbf{y}) = \displaystyle\int_{-\infty}^{\infty}\int_{0}^{\infty}\frac{c_i}{\sigma^ {1+q_i}}(\frac{1}{\sqrt{2\pi}\sigma})^{n}\exp\{-\frac{1}{2\sigma^2}\Big[\textbf{R}_i(\textbf{y}) + (\theta_i - \textbf{y})^T A_i^T A_i (\theta_i - \textbf{y})\Big]\}d\theta_i d\sigma$$
Where $\textbf{R}_i(\textbf{y}) = (\textbf{y}- A_i\hat{\theta_i})^T(\textbf{y}- A_k\hat{\theta_i})$. Since the second term inside the exponential function is a multivariate normal distribution, we can integrate it with respect to $\theta_i$. Let $\dim(\theta_i) = p_i$, then the marginal can be written as:
$$ m_i(\textbf{y}) = c_i\displaystyle\Big(\frac{1}{\sqrt{2\pi}}\Big)^{n-p_i}|\textbf{A}_i^{T} \textbf{A}_i|^{-1/2}\int_{0}^{\infty}(\frac{1}{\sigma})^{n + q_i - p_i + 1 }\exp\{-\frac{\textbf{R}_i(\textbf{y})}{2\sigma^2}\} d\sigma.$$
We can solve this integral by employing the following change of variables
$$ \eta = \frac{\textbf{R}_i(\textbf{y})}{2\sigma^2}, \hspace{10pt} \sigma = \frac{\textbf{R}^{1/2}_i(\textbf{y})}{\sqrt{2}\eta^{1/2}}, \hspace{10pt} d\sigma = -\frac{\textbf{R}^{1/2}_i(\textbf{y})}{\sqrt{2}}\frac{1}{2}\frac{1}{\eta^{3/2}}d\eta $$
\begin{align*}
	m_i(\textbf{y}) &= \Big(\frac{1}{\sqrt{2\pi}}\Big)^{n-p_i}c_i|\textbf{A}_i^{T} \textbf{A}_i|^{-1/2}\int_{0}^{\infty}\exp^{-\eta}\eta^{\frac{n + q_i + 1 - p_i}{2}}\Big( \frac{2}{\textbf{R}_i(\textbf{y})} \Big)^{(n + q_i + 1- p_i)/2}\Big(\frac{\textbf{R}_i(\textbf{y})}{2}\Big)^{1/2}\frac{1}{2}\frac{1}{\eta^{3/2}}d\eta\\
	&=2^{\frac{n+q-p_i - 2}{2}}\Big(\frac{1}{\sqrt{2\pi}}\Big)^{n-p_i}\Big( \frac{1}{\textbf{R}_i(\textbf{y})} \Big)^{\frac{n + q_i -p_i + 1}{2}}\Big(\textbf{R}_i(\textbf{y})\Big)^{1/2}|\textbf{A}_i^{T} \textbf{A}_i|^{-1/2}\frac{c_i}{2}\int_{0}^{\infty}\exp^{-\eta}\eta^{\frac{n + q_i- p_i + 3}{2} -  1}d\eta
\end{align*}
Solving the integral above yields
$$ \int_{0}^{\infty}\exp^{-\eta}\eta^{\frac{n + q_i - p_i + 3}{2} - 1}d\eta = \Gamma(\frac{n + q_i - p_i + 3}{2})$$
Therefore, the marginal for $M_i$ is
\begin{align*}
	m_i(\textbf{y}) &= c_i\mbox{ }2^{\frac{n+q_i-p_i}{2}}\Big(\frac{1}{\sqrt{2\pi}}\Big)^{n-p_i}\Big( \frac{1}{\textbf{R}_i(\textbf{y})} \Big)^{\frac{n + q_i- p_i + 1}{2}}\Big(\frac{\textbf{R}_i(\textbf{y})}{2}\Big)^{1/2}|\textbf{A}_i^{T} \textbf{A}_i|^{-1/2}\Gamma(\frac{n - p_i+q_i + 3}{2})\\
	&\propto c_i\sqrt{|\textbf{A}_i^{T} \textbf{A}_i|^{-1}\hspace{3pt}[4\pi]^{p_i}
		\hspace{3pt}\textbf{R}_i(\textbf{y})^{-(n+q_i - p_i)}}\hspace{3pt}
\end{align*}
After we obtain the marginals, we can calculate the Bayes factor as
$$ B_{01}^N(\textbf{y}) = \frac{m_0(\textbf{y})}{m_1(\textbf{y})} = \frac{c_0}{c_1}\sqrt{[4\pi]^{(p_0 - p_1)}\frac{|\textbf{A}_1^{T} \textbf{A}_1|}{|\textbf{A}_0^{T} \textbf{A}_0|} \frac{[\textbf{R}_1(\textbf{y})]^{n+q_1 - p_1}}{[\textbf{R}_0(\textbf{y})]^{n+q_0 - p_0}}}$$
The following step involves computing the Bayes factor for the minimal training sample. For this situation, where at least one observation per unknown parameter is necessary, the minimal training sample size can be found as:
$$n_{01}= \max\{p_0 + 1,p_1 + 1\}.$$
In this case, the Bayes factor for the minimal training sample takes the form:
$$ B_{10}^N(\textbf{y}(\ell)) = \frac{m_1(\textbf{y}(\ell))}{m_0(\textbf{y}(\ell))} = \frac{c_1}{c_0}\sqrt{[4\pi]^{(p_1 - p_0)}\frac{|\textbf{A}_0^{T}(\ell) \textbf{A}_0(\ell)|}{|\textbf{A}_1^{T}(\ell) \textbf{A}_1(\ell)|}\frac{[\textbf{R}_0(\textbf{y}(\ell))]^{n_{01}+q_0 - p_0}}{[\textbf{R}_1(\textbf{y}(\ell))]^{n_{01}+q_1 - p_1}}}$$
The previously discussed \textbf{Lemma 1} establishes that:
$$ B_{01}(\textbf{y}(-\ell)|\textbf{y}(\ell)) =  B_{01}^{N}(\textbf{y}) B_{10}^{N}(\textbf{y}(\ell)) $$
Consequently,
$$ B_{01}(\textbf{y}(-\ell)|\textbf{y}(\ell)) =  \sqrt{\frac{|\textbf{A}_1^{T} \textbf{A}_1|}{|\textbf{A}_0^{T} \textbf{A}_0|}\frac{|\textbf{A}_0^{T}(\ell) \textbf{A}_0(\ell)|}{|\textbf{A}_1^{T}(\ell) \textbf{A}_1(\ell)|}\frac{[\textbf{R}_1(\textbf{y})]^{n+q_1 - p_1}}{[\textbf{R}_0(\textbf{y})]^{n+q_0 + p_0}}
	\frac{[\textbf{R}_0(\textbf{y}(\ell))]^{n_{01}+q_0 - p_0}}{[\textbf{R}_1(\textbf{y}(\ell))]^{n_{01} +q_1 - p_1}}}$$
By leveraging these results, we derive the Empirical Intrinsic Bayes Factor (IBF) lower bound for General Linear Models:
$$ \underline{B}_{01}^{GLI^{*}} =  \sqrt{\frac{|\textbf{A}_1^{T} \textbf{A}_1|}{|\textbf{A}_0^{T} \textbf{A}_0|}\frac{[\textbf{R}_1(\textbf{y})]^{n+q_1 - p_1}}{[\textbf{R}_0(\textbf{y})]^{n+q - p_0}}}
\displaystyle \min_{\ell = 1,...,L}\sqrt{\frac{|\textbf{A}_0^{T}(\ell) \textbf{A}_0(\ell)|}{|\textbf{A}_1^{T}(\ell) \textbf{A}_1(\ell)|} \frac{[\textbf{R}_0(\textbf{y}(\ell))]^{n_{01} +q_0 -p_0}}{[\textbf{R}_1(\textbf{y}(\ell))]^{n_{01} +q_1 - p_1}}}$$
Additionally, the Theoretical Intrinsic Bayes Factor (IBF) Lower bound for General Linear Models can be expressed as:
\begin{equation}\label{BFGSS}
	\underline{B}_{01}^{GLI} =  \sqrt{\frac{|\textbf{A}_1^{T} \textbf{A}_1|}{|\textbf{A}_0^{T} \textbf{A}_0|}\frac{[\textbf{R}_1(\textbf{y})]^{n+q_1 - p_1}}{[\textbf{R}_0(\textbf{y})]^{n+q_0 - p_0}}\inf_{y(\ell) \in D}\frac{|\textbf{A}_0^{T}(\ell) \textbf{A}_0(\ell)|}{|\textbf{A}_1^{T}(\ell) \textbf{A}_1(\ell)|}
		\frac{[\textbf{R}_0(\textbf{y}(\ell))]^{n_{01} +q_0 - p_0}}{[\textbf{R}_1(\textbf{y}(\ell))]^{n_{01} +q_1 - p_1}}}
\end{equation} 
\newpage

\section{The One-way Layout (ANOVA)}
The One-way Layout in analysis of variance (ANOVA) involves $m$ groups of observations, each with $n_i$ observations ($j = 1,...,n_i$) independently drawn from a normal distribution $y_{ij} \sim N(\mu_i,\sigma^2)$, given $\mu_1,...,\mu_m,\sigma^2$.

The models considered are:
$$ M_0: \mu_1 = ... = \mu_m \text{ versus } M_1: \mu_i \neq \mu_j \text{ for some } i \neq j. $$
In this scenario, $p_1 = m$, $p_0 = 1$, and the general form of the matrices $\textbf{A}_0$ and $\textbf{A}_1$ can be derived. Let $n = \sum_{i=1}^{m} n_i$. Notably, the design matrices $A_0$ and $A_1$ for the one-way ANOVA can generally be written as follows:
\begin{center}
	$A_0 = 
	\begin{bmatrix}
		1 \\
		\cdot \\
		\cdot \\
		\cdot \\
		1 
	\end{bmatrix}$,
	$A_0^T =
	\begin{bmatrix}
		1 & \cdot \cdot \cdot & 1\\
	\end{bmatrix},$
	$A_1 =
	\begin{bmatrix}
		1 & 0 & \cdot \cdot \cdot & 0 & 0\\
		\vdots & \vdots & \cdot \cdot \cdot & \vdots & \vdots\\
		1 & 0 & \cdot \cdot \cdot & 0 & 0\\
		0 & 1 & \cdot \cdot \cdot & 0 & 0\\
		\vdots & \vdots & \cdot \cdot \cdot & \vdots & \vdots \\
		0 & 1 & \cdot \cdot & 0 & 0\\
		\vdots & \vdots & \cdot \cdot \cdot& \vdots & \vdots \\
		0 & 0 & \cdot \cdot \cdot & 0 & 1\\
		\vdots & \vdots & \cdot \cdot \cdot & \vdots & \vdots \\
		0 & 0 & \cdot \cdot \cdot & 0 & 1
	\end{bmatrix}$
\end{center}
\vspace{5pt}The matrices $A_0$ and $A_1$ are ($n \times m$) matrices, and each block of ones in $A_1$ corresponds to each of the m groups. In this case, every block i has $n_i$ rows and $\det(\textbf{A}_0^T \textbf{A}_0) = n$.  The matrix product of $\textbf{A}_1^T$ and $\textbf{A}_1$ can be expressed as:
$$\textbf{A}_1^T \textbf{A}_1 = 
\begin{bmatrix}
	n_1 & 0 & \cdot \cdot \cdot & 0\\
	\cdot & n_2 & \cdot \cdot \cdot & \cdot\\
	\cdot & \cdot & \ddots & \cdot \\
	0 & 0 & \cdot \cdot \cdot & n_m
\end{bmatrix}$$
The matrix resulting from the product of these two matrices is itself diagonal, with its determinant expressed as:
\begin{equation}\label{determinantM1}
	\det(\textbf{A}_1^T \textbf{A}_1) = \prod_{i=1}^{m}n_i
\end{equation}
Thus, we can compute the term $|\textbf{A}_1^T \textbf{A}_1|/|\textbf{A}_0^T \textbf{A}_0|$ in equation \eqref{generallmbf} as:
\begin{equation}\label{matrixdetermianants}
	|\textbf{A}_1^T \textbf{A}_1|/|\textbf{A}_0^T \textbf{A}_0| = \frac{\prod_{i=1}^{n}n_i}{n}
\end{equation}
In this case, the minimal training sample requires at least one observation in each group plus one extra observation in any of the groups to estimate $\sigma^2$. Thus, we require that $n_i = 1, i=1,...,j-1,j+1,...,n$, and $n_j = 2$, for some $j \in [1,m]$.

For the case mentioned above, we can use equation \eqref{constants_ratio}, and the result obtained in \eqref{determinantM1} to get $c_1$ as:
\begin{equation}\label{determinantM1_2}
	\det(\textbf{A}_1(\ell)^T \textbf{A}_1(\ell)) = (1\cdot\cdot\cdot1\cdot2\cdot1\cdot\cdot\cdot1) = 2
\end{equation}
Similarly, $c_0$ can be derived as:
\begin{equation}\label{determinantM0}
	\det(\textbf{A}_0(\ell)^T \textbf{A}_0(\ell)) = n = \sum_{j = 1}^{m} n_j = (1+\cdot\cdot\cdot+ 1 + 2 + 1 + \cdot\cdot\cdot + 1) = m + 1
\end{equation}
Thus, utilizing the outcomes from \eqref{determinantM1_2} and \eqref{determinantM0}, we can evaluate the expression in equation \eqref{constants_ratio} as follows:
\begin{equation}\label{constants_ratio_mts}
	\frac{c_0}{c_1} = \big(\frac{\det(\textbf{A}_0(\ell)^T \textbf{A}_0(\ell))}{\det(\textbf{A}_1(\ell)^T \textbf{A}_1(\ell))}\big)^{\frac{1}{2}} =  \big( \frac{m+1}{2} \big)^{\frac{1}{2}}
\end{equation}
Now, by substituting \eqref{constants_ratio_mts} and \eqref{matrixdetermianants} into \eqref{generallmbf}, derive the following expression for the Bayes factor:
\begin{equation}\label{bfanovamts}
	B^{SS}_{01} = \big( \frac{m+1}{2} \big)^{\frac{1}{2}}\big[\displaystyle\prod_{i=1}^{m}n_i\big/n\big]^{\frac{1}{2}}\big[1 + \frac{(m-1)}{n-m}F\big]^{-n/2}
\end{equation}
The SS Bayes Factor can alternatively be expressed compactly as:
\begin{equation}\label{bfanovamts}
	\boxed{B^{SS}_{01} = \Bigg(\frac{\big( \frac{m+1}{2} \big)\big[\prod_{i=1}^{m}n_i\big/n\big]}{\big[1 + \frac{(m-1)}{n-m}F\big]^{n}}\Bigg)^{\frac{1}{2}}}
\end{equation}
Furthermore, relating the $F$ statistics to the p-value can be achieved through:
$$ F_{\nu_1,\nu_2} = qf(1-p,\nu_1,\nu_2)$$
Here, $qf$ denotes the quantile function for the $F$ distribution, while $\nu_1$ and $\nu_2$ represent the degrees of freedom. For this case, $\nu_1 = p_1 - p_0$ and $\nu_2 = n - p_1$. Consequently, expressing the SS Bayes factor in terms of the p-value yields:
$$ \boxed{B^{SS}_{01}(p) = \Bigg(\frac{\big( \frac{m+1}{2} \big)\big[\prod_{i=1}^{m}n_i\big/n\big]}{\big[1 + \frac{(m-1)}{n-m}qf(1-p,p_1-p_0,n-p_1)\big]^{n}}\Bigg)^{\frac{1}{2}}}$$
This formula offers a straightforward method to convert p-values into Bayes factors. It's akin to the Sellke et al. (2001) \cite{sellke2001calibration} bound, but distinguishes itself by being dynamic, enhancing accuracy with increasing information. This technique is applicable to any Bayes factor that relies on F-statistics.

Our subsequent aim involves computing $B^{GSS}_{01}$ as in equation $\eqref{BFGSS}$ for ANOVA models and comparing it with their findings. Considering that $A_0(\ell)$ and $A_1(\ell)$ remain independent of observations in ANOVA models, and the square root is a monotonic function, we can rewrite the expression as:
$$ \displaystyle \sup_{y(\ell) \in D_n}\sqrt{\frac{|\textbf{A}_0^{T}(\ell) \textbf{A}_0(\ell)|}{|\textbf{A}_1^{T}(\ell) \textbf{A}_1(\ell)|}\frac{[\textbf{R}_0(\textbf{y}(\ell))]^{n_{01} +q_i - p_0}}{[\textbf{R}_1(\textbf{y}(\ell))]^{n_{01} +q - p_1}}}  = \sqrt{\frac{|\textbf{A}_0^{T}(\ell) \textbf{A}_0(\ell)|}{|\textbf{A}_1^{T}(\ell) \textbf{A}_1(\ell)|} \displaystyle \sup_{y(\ell) \in D_n}\frac{[\textbf{R}_0(\textbf{y}(\ell))]^{n_{01} +q_i - p_0}}{[\textbf{R}_1(\textbf{y}(\ell))]^{n_{01} +q_i - p_1}}}
$$
The General Theoretical and Empirical SS Bayes factors for ANOVA can be expressed respectively as:
$$ \boxed{B_{10}^{GSSA} =  \sqrt{\frac{|\textbf{A}_1^{T}(\ell) \textbf{A}_1(\ell)|}{|\textbf{A}_0^{T}(\ell) \textbf{A}_0(\ell)|}\frac{|\textbf{A}_0^{T} \textbf{A}_0|}{|\textbf{A}_1^{T} \textbf{A}_1|}\frac{[\textbf{R}_0(\textbf{y})]^{n+q_0 - p_0}}{[\textbf{R}_1(\textbf{y})]^{n+q_1 - p_1}}\sup_{y(\ell) \in D}
		\frac{[\textbf{R}_1(\textbf{y}(\ell))]^{n_{01} +q_1 - p_1}}{[\textbf{R}_0(\textbf{y}(\ell))]^{n_{01} +q_0 - p_0}}}}
$$  
$$ \boxed{B_{10}^{GESSA} =  \sqrt{\frac{|\textbf{A}_1^{T}(\ell) \textbf{A}_1(\ell)|}{|\textbf{A}_0^{T}(\ell) \textbf{A}_0(\ell)|}\frac{|\textbf{A}_0^{T} \textbf{A}_0|}{|\textbf{A}_1^{T} \textbf{A}_1|}\frac{[\textbf{R}_0(\textbf{y})]^{n+q_0 - p_0}}{[\textbf{R}_1(\textbf{y})]^{n+q_1 - p_1}}\displaystyle{\max_{\ell = 1,...,L}}
		\frac{[\textbf{R}_1(\textbf{y}(\ell))]^{n_{01} +q_1 - p_1}}{[\textbf{R}_0(\textbf{y}(\ell))]^{n_{01} +q_0 - p_0}}}}
$$  

\subsection{The Full Jeffrey's (no independence)  $q_k = p_k,$, for $k = i,j$ for ANOVA}
In accordance with our earlier calculations pertaining to the General Empirical SS Bayes factor for ANOVA, we apply the methodology proposed by Spiegelhalter and Smith (1982) \cite{spiegelhalter1982bayesfactors}. This involves the selection of the full Jeffrey's prior where $q_i = p_i$. Through the utilization of an imaginary training sample, which renders maximum support for $M_0$, we ascertain that $F{_y(\ell)} = 0$. Consequently, we obtain the General Theoretical SS Bayes Factor for ANOVA under the Full Jeffrey's prior
$$B_{01}^{GESSA_{FJP}} = \sqrt{\frac{m+1}{2}\big[\displaystyle\prod_{i=1}^{m}n_i\big/n\big]\Big(\frac{\textbf{R}_1(\textbf{y})}{\textbf{R}_0(\textbf{y})}\Big)^{n}\min_{\ell = 1,...,L}
	\Big(\frac{\textbf{R}_0(\textbf{y}(\ell))}{\textbf{R}_1(\textbf{y}(\ell))}\Big)^{m+1}}.
$$
Now, expressing the ratio of the residual sum of squares as an F-statistics results in:
$$B_{01}^{GESSA_{FJP}}=\sqrt{\frac{m+1}{2}\big[\displaystyle\prod_{i=1}^{m}n_i\big/n\big]\Big(1 + \frac{(m - 1)}{(n - m)}F_{y}\Big)^{-n}\max_{\ell = 1,...,L}
	\Big(1 + \frac{(p_1 - p_0)}{(n_{01} - p_1)}F_{y(\ell)}\Big)^{m+1}}$$

The General Theoretical SS Bayes factor for ANOVA results in:
$$B_{10}^{GSSA_{FJP}}=\sqrt{\frac{2}{m+1}\big[\displaystyle n\big/\prod_{i=1}^{m}n_i\big]\Big(1 + \frac{(m - 1)}{(n - m)}F_{\textbf{y}}\Big)^{n}\sup_{y(\ell) \in D}
	\Big[\frac{\textbf{R}_1(\textbf{y}(\ell))}{\textbf{R}_0(\textbf{y}(\ell))}\Big]^{m+1}}$$

Observing that $\textbf{R}_0(y(\ell)) \geq \textbf{R}_1(y(\ell))$, it follows that:
$$ \sup_{y(\ell) \in D}
\frac{[\textbf{R}_1(\textbf{y}(\ell))]}{[\textbf{R}_0(\textbf{y}(\ell))]} = 1 \Longrightarrow\sup_{y(\ell) \in D}
\Big[\frac{\textbf{R}_1(\textbf{y}(\ell))}{\textbf{R}_0(\textbf{y}(\ell))}\Big]^{m+1} = 1, \hspace{5pt}, \forall m > 0 $$ 
Consequently, the General Theoretical SS Bayes Factor, employing the full Jeffrey's prior for ANOVA models, simplifies to:
$$\boxed { B_{10}^{GSSA_{FJP}}=\sqrt{\frac{2}{m+1}\big[\displaystyle n\big/\prod_{i=1}^{m}n_i\big]\Big(1 + \frac{(m - 1)}{(n - m)}F_{\textbf{y}}\Big)^{n}} = B_{10}^{SS}.}$$
Alternatively, this Bayes factor can be expressed as:
$$ \boxed{B_{10}^{GSSA_{FJP}} =  \sqrt{\frac{2}{m+1}\big[\displaystyle n\big/\prod_{i=1}^{m}n_i\big] \Big[\frac{\textbf{R}_0(\textbf{y})}{\textbf{R}_1(\textbf{y})}\Big]^n}.}
$$ 
\subsection{The Reference prior ($q_k = 0, k = i,j$)  for ANOVA}
In this section, we delve into the theoretical SS Bayes factor for the ANOVA scenario, employing the reference prior obtained by setting $q_k = 0$ for $k=i,j$. The theoretical SS Bayes Factor for ANOVA, specifically under the \textit{reference prior} context, is formulated as:
$$\boxed{B_{10}^{GSSA_{RP}} =  \sqrt{\frac{|\textbf{A}_1^{T}(\ell) \textbf{A}_1(\ell)|}{|\textbf{A}_0^{T}(\ell) \textbf{A}_0(\ell)|}\frac{|\textbf{A}_0^{T} \textbf{A}_0|}{|\textbf{A}_1^{T} \textbf{A}_1|}\frac{[\textbf{R}_0(\textbf{y})]^{n+ 1}}{[\textbf{R}_1(\textbf{y})]^{n- (m+1)}}\sup_{y(\ell) \in D}
		\Big(\frac{1}{\textbf{R}_0(\textbf{y}(\ell))}\Big)^{m}}}
$$

Upon inspection of the aforementioned equation, it becomes clear that under this particular prior, the theoretical SS Bayes factor tends to either zero or infinity. Consequently, the computation of the theoretical SS BF with this prior is rendered unfeasible. This implies that for reference priors, the SS bounds do not exist. This limitation is not previously highlighted in the works of Smith and Spiegelhalter (1980)\cite{smith1980bayesfactors} and (1982) \cite{spiegelhalter1982bayesfactors}, thereby underscoring an important consideration for future research in this domain.
\subsection{The Modified Jeffrey's Prior ($q_j = p_j - p_i, q_i = 0$) for ANOVA}
\hspace{10pt}The General Empirical SS Bayes Factor for ANOVA, under the Modified Jeffrey's prior results in: 
$$ \boxed{B_{10}^{GESSA_{MJP}} = \sqrt{\frac{2}{m+1}\big[n\big/\displaystyle\prod_{i=1}^{m}n_i\big]\Big(\frac{\textbf{R}_0(\textbf{y})}{\textbf{R}_1(\textbf{y})}\Big)^{n-1}\max_{\ell = 1,...,L}
		\Big(\frac{\textbf{R}_1(\textbf{y}(\ell))}{\textbf{R}_0(\textbf{y}(\ell))}\Big)^{m}}.}
$$ 
Similarly, the general Theoretical SS Bayes Factor for ANOVA, under the Modified Jeffrey's prior results in: 
$$ \boxed{B_{10}^{GSSA_{MJP}} = \sqrt{\frac{2}{m+1}\big[n\big/\displaystyle\prod_{i=1}^{m}n_i\big]\Big(\frac{\textbf{R}_0(\textbf{y})}{\textbf{R}_1(\textbf{y})}\Big)^{n-1}\sup_{y(\ell) \in D}
		\Big(\frac{\textbf{R}_1(\textbf{y}(\ell))}{\textbf{R}_0(\textbf{y}(\ell))}\Big)^{m}}.}
$$ 
Alternatively, this Bayes Factor can be expressed in terms of the $F$ statistics as:
$$ \boxed{B_{10}^{GSSA_{MJP}} = \sqrt{\frac{2}{m+1}\big[n\big/\displaystyle\prod_{i=1}^{m}n_i\big]\Big(1 + \frac{(m - 1)}{(n - m)}F_{y}\Big)^{n-1}}.} $$ 

\section{Conclusions}
\hspace{10pt} In conclusion, the ANOVA case analysis has provided valuable insights regarding the use of different priors in the context of Full Jeffreys, Modified Jeffreys, and Reference Prior. The SS bounds obtained for Full Jeffreys and Modified Jeffreys are informative, useful, and exhibit close proximity to each other. Importantly, it has been demonstrated that the Bayes Factor can be expressed as a function of the F-Statistics under both priors. However, in the case of the Reference Prior, the SS bound fails to provide informative results, highlighting its lack of usefulness. Notably, it is intriguing to observe the sensitivity of the SS bound to the initial objective prior, which further supports the conclusions drawn by Berger and Pericchi (1996) \cite{berger1996intrinsic} and others that the Modified Jeffreys prior is a superior choice, particularly within the Linear Gaussian Model and potentially in a broader context. Additionally, it may be of interest to explore the extent to which the Modified Jeffreys prior possesses additional properties, such as matching, besides its inherent simplicity and the one-on-one relationship it establishes with the F-Statistics.
\newpage
\chapter{Separated Hypothesis Testing}
\hspace{15pt} Traditional hypothesis testing typically involves comparing a single hypothesis to a pre-determined null hypothesis, assuming a single underlying distribution. However, in certain scenarios, there arises a need to examine and compare two distinct distributions simultaneously. This is the essence of separated hypothesis testing, a concept that expands upon conventional approaches by accommodating situations where the hypothesis being tested corresponds to two separate populations or data sources. By addressing this unique requirement, separated hypothesis testing enables researchers to gain valuable insights into the potential disparities or similarities between two distinct distributions, thereby enhancing the accuracy and comprehensiveness of statistical inference. In this chapter, we delve into the theoretical underpinnings, methodology, and applications of separated hypothesis. 

We explored the concept of separated hypothesis testing through practical examples involving the comparison of different probability distributions. Specifically, we consider the scenarios of Poisson vs. Geometric and Poisson vs. Negative Binomial distributions. We computed IBF bounds for both comparisons, ensuring the reliability and consistency of our results. To validate our findings, we conducted simulations where we progressively augmented the sample size, observing a numerical convergence toward accurate decisions in our experiments. Each example is accompanied by a comprehensive discussion of the methodologies and code used in our analyses. The figures presented serve as compelling visual evidence substantiating the validity of our conclusions.
\section{The Poisson vs. Geometric Separated Hypothesis Test}
Let's consider comparing two distinct models: the Poisson and the Geometric distributions, where $\textbf{y} = (y_1,...,y_n)$ represents independently and identically distributed (i.i.d.) observations. This frames our model selection problem as follows:$$M_0: Y \sim Poisson(\lambda) \mbox{ vs } M_1: Y \sim Geometric(\theta)$$ 
The densities for each model are:
$$f_0(y|\lambda) = \frac{e^{-\lambda}\lambda^y}{y!}, \lambda > 0, \hspace{3pt} \mbox{ and } \hspace{3pt}  f_1(y|\theta) = \theta(1-\theta)^{y}, \theta \in (0,1), y = 0,1,2,...$$
The Jeffreys's priors and likelihoods associated with $M_0$ and $M_1$ yield the following outcomes:
$$ \pi_0^{J}(\lambda) = \frac{c_0}{\lambda^{1/2}}, \hspace{10pt} \pi_1^{J}(\theta) = \frac{c_1}{\theta(1-\theta)^{1/2}}$$
$$f_0(\textbf{y}|\lambda) = \frac{e^{-n\lambda}\lambda^{\sum_{i=1}^{n} y_i}}{\displaystyle\prod_{i=1}^{n}y_i!}, \hspace{3pt} \mbox{ and } \hspace{3pt}  f_1(\textbf{y}|\theta) = \theta^n(1-\theta)^{\sum_{i=1}^{n}y_i}$$
Given that both models entail a single parameter, the minimal training sample size required is one. Utilizing the priors and likelihoods, we proceed to compute the marginals and subsequently derive the Bayes factors as follows:
$$ m_0(\textbf{y}) = \frac{c_0}{\displaystyle\prod_{i=1}^{n}y_i!}\displaystyle\int_{0}^{\infty}e^{-n\lambda}\lambda^{\sum_{i=1}^{n} y_i}\frac{1}{\lambda^{1/2}}d\lambda$$
$$ m_1(\textbf{y}) =c_1\displaystyle\int_{0}^{1} \theta^n(1-\theta)^{\sum_{i=1}^{n}y_i}\frac{1}{\theta(1-\theta)^{1/2}}d\theta$$
After straightforward algebraic manipulation, we find that the integral in $m_0(\textbf{y})$ follows a Gamma distribution, while the integral in $m_1(\textbf{x})$ corresponds to a Beta distribution, wherein:
$$ m_0(\textbf{y}) = \frac{c_0\Gamma(\sum y_i + \frac{1}{2})}{\displaystyle\prod_{i=1}^{n}y_i!}(1/n)^{(\sum y_i + \frac{1}{2})}$$
$$ m_1(\textbf{y}) =c_1\displaystyle\int_{0}^{1} \theta^{n-1}(1-\theta)^{\sum_{i=1}^{n}y_i - 1/2}d\theta =c_1\frac{\Gamma(n)\Gamma(\sum y_i + 1/2) }{\Gamma(n + \sum y_i + 1/2)} $$

Thus, the Bayes Factor results in: 
$$ B_{01}^{N}(\textbf{y}) = \frac{c_0}{c_1} \frac{\Gamma(n + \sum y_i + 1/2)}{\displaystyle\prod_{i=1}^{n}y_i!\Gamma(n)n^{(\sum y_i + \frac{1}{2})}}.$$
The Empirical IBF upper bound can be obtained by computing: 
$$\overline{B}^{I^*}_{10}(\textbf{y}) = B_{10}^{N}(\textbf{y}) \displaystyle\sup_{y(\ell) \in D_{n,k}}\frac{m_0(y(\ell))}{m_1(y(\ell))}$$
Evidently, the following marginals are required:
$$m_0(y(\ell)) = c_0\frac{\Gamma(y(\ell) + 1/2)}{y(\ell)!}, \hspace{15pt} m_1(y(\ell)) = c_1\frac{\Gamma(1)\Gamma(y(\ell) + 1/2)}{\Gamma(1 + y(\ell) + 1/2)}$$
By using those marginals, we can calculate the Bayes factor for $y(\ell)$ as:
$$ B_{01}^{N}(y(\ell)) = \frac{c_0}{c_1}\frac{1}{\Gamma(y(\ell) + 1 + 1/2)}$$ 
The above Bayes factor, as a function of $y(\ell)$ is decreasing. Therefore, calculating the theoretical supremum, which occurs at $y(\ell) = 0$, results in:
$$\sup_{y(\ell) \in D}B_{01}^{N}(y(\ell)) = \frac{c_0}{c_1}\frac{1}{\Gamma(1 + 1/2)}$$
Utilizing the identity $\Gamma(n + 1/2) = (2n!\sqrt{\pi})/(4^n n!)$, we derived:
$$\sup_{y(\ell) \in D}B_{01}^{N}(y(\ell)) = \frac{c_0}{c_1}\frac{2}{\sqrt{\pi}}$$
This results in the Theoretical IBF upper bound:
$$\boxed{\overline{B}_{10}^{I}(\textbf{y}) = \frac{\displaystyle\prod_{i=1}^{n}y_i!(n-1)!n^{(n\bar{y} + \frac{1}{2})}}{\Gamma(n(1 + \bar{y}) + 1/2)}\frac{2}{\sqrt{\pi}}}$$
And the Theoretical IBF lower bound:
$$\boxed{\underline{B}_{01}^{I}(\textbf{y}) = \frac{\Gamma(n(1 + \bar{y}) + 1/2)}{\displaystyle\prod_{i=1}^{n}y_i!(n-1)!n^{(n\bar{y} + \frac{1}{2})}}\frac{\sqrt{\pi}}{2}}$$
Similarly, the empirical IBF upper bound is:
$$\boxed{\overline{B}_{10}^{I^{*}}(\textbf{y}) = \frac{\displaystyle\prod_{i=1}^{n}y_i!(n-1)!n^{(n\bar{y} + \frac{1}{2})}}{\Gamma(n(1 + \bar{y}) + 1/2)}\frac{1}{\Gamma(y_{(1)} + 1 + \frac{1}{2}) }, \hspace{15pt} y_{(1)} = \displaystyle\min_{i=1,...,n}\{y_i\}}$$
And the empirical IBF lower bound is:
$$\boxed{\underline{B}_{01}^{I^{*}}(\textbf{y}) = \frac{\Gamma(n(1 + \bar{y}) + 1/2)\Gamma(y_{(n)} + 1 + \frac{1}{2})}{\displaystyle\prod_{i=1}^{n}y_i!(n-1)!n^{(n\bar{y} + \frac{1}{2})}}, \hspace{15pt} y_{(n)} = \displaystyle\max_{i=1,...,n}\{y_i\}}$$
\hspace{15pt} We aim to ascertain the reliability and consistency of our obtained results by implementing an \textbf{R} script specifically designed to calculate the Bayes factor in two distinct scenarios: when the null hypothesis ($H_0$) holds true and when the alternative hypothesis ($H_1$) is true. By conducting this analysis, we will demonstrate that when the samples are drawn from the Poisson distribution, the computed Bayes factor approaches zero, indicating stronger support for $H_0$. Conversely, when the samples originate from the Geometric distribution, the Bayes factor exhibits values significantly greater than one, favoring $H_1$. Furthermore, we will calculate the empirical IBF bounds, providing numerical bounds for both $B_{01}^{N}(\textbf{y})$ and $B_{10}^{N}(\textbf{y})$. The $\textbf{R}$ script below generates one hundred simulations of thirty independent samples from the Poisson distribution and from the Geometric distribution. By employing this approach, we calculated the strength of evidence in favor of the null hypothesis in both scenarios. As anticipated, the obtained values are consistent with our prior knowledge about the data that we simulated, indicating the expected outcome of growing indefinitely when the null is false and going to zero under the null. We smoothed the results by calculating the average value for each sample size over all of the simulations.
\\ \\
\textbf{R Program}
\begin{lstlisting}[language=R]
	
	#100 simulations of 30 samples each
	
	#Null hypothesis true, lambda =1
	lambda = 1
	n_sim = 100; n_data = 30; p=0.5
	bf10ss_pois = c(); data_pois = c(); bfss10_pois = c();
	bf10ss_geom = c(); data_geom = c(); bfss10_geom = c();
	for (i in 1:n_sim) { 
		for (j in 1:n_data) {
			
			data_geom = c(data_geom,rgeom(n=1,prob=p));
			data_pois = c(data_pois,rpois(n=j,lambda=lambda));
			s_geom = sum(data_geom)
			s_pois = sum(data_pois)
			y_fact_geom = factorial(data_geom)
			y_fact_pois = factorial(data_pois)
			
			#Marginals
			m1_geom = (j^(s_geom + 0.5)*factorial(j-1)*prod(y_fact_geom)*2)
			m0_geom = (gamma(j + s_geom + 0.5)*sqrt(pi))
			bfss10_geom = c(bfss10_geom,m1_geom/m0_geom)
			m1_pois = (j^(s_pois + 0.5)*factorial(j-1)*prod(y_fact_pois)*2)
			m0_pois = (gamma(j + s_pois + 0.5)*sqrt(pi))
			bfss10_pois = c(bfss10_pois,m1_pois/m0_pois)
		}
		bf10ss_pois = cbind(bf10ss_pois,bfss10_pois)
		bf10ss_geom = cbind(bf10ss_geom,bfss10_geom)
	}
	bf10ss_avg_geom = rep(0,n_data)
	bf10ss_avg_pois = rep(0,n_data)
	for (i in 1:n_data) {
		bf10ss_avg_geom[i] = mean(bf10ss_geom[i,])
		bf10ss_avg_pois[i] = mean(bf10ss_pois[i,])
	}
	par(mfrow=c(1,2))
	plot(1:n_data,bf10ss_avg_pois,type="l",col="red",main="Geometric vs Poisson Distribution Bayes Factor Bound",ylab="Average SS B10 Bound",xlab="Sample size")
	mtext(text="Data generated from the Poisson distribution",side=3)
	plot(1:n_data,bf10ss_avg_geom,type="l",col="red",main="Geometric vs Poisson Distribution Bayes Factor Bound",ylab="Average SS B10 Bound",xlab="Sample size")
	mtext(text="Data generated from the Geometric distribution",side=3)
\end{lstlisting}
\newpage 
\begin{figure}
	\centering
	\subfigure{\includegraphics[width=0.90\textwidth]{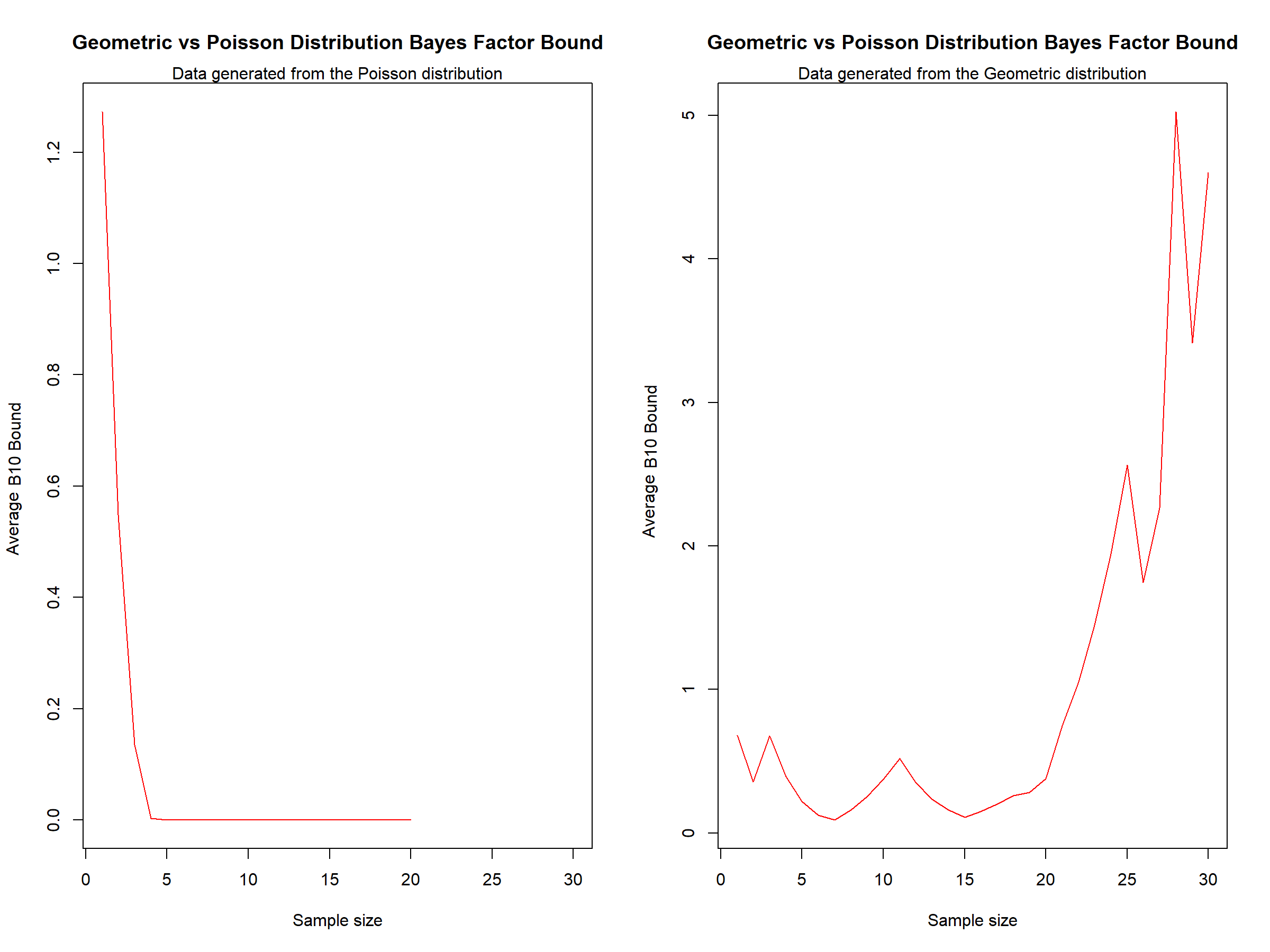}} 
	\caption{Poisson vs. Geometric IBF 10 bound as the sample size increases when $H_0$ is true (left) and when $H_0$ is false (right).}
	\label{fig:bfsspoissongeometric}
\end{figure}
\newpage
Once we obtain the vector of samples $\textbf{y}$, we can proceed to compute the empirical bound for $B_{10}^{AI}(\textbf{y})$. The code below calculates the value of the Bayes Factor for each sample in $\textbf{y}$ and then identifies the maximum value among them, representing the empirical maximum. By employing this code, we obtain a numerical measure that reflects the strength of evidence in favor of the alternative hypothesis.
\\ 
\textbf{R Program}
\begin{lstlisting}
	bfss01_mts = function(y,print=TRUE) { 
		n = length(y); bfss_01=c()
		for (i in 1:n) { m1 = gamma(y[i] + 1.5); m0 = 1; bfss_01 = c(bfss_01,m0/m1) }
		y_max = y[which.max(bfss_01)]; bfss_max = max(bfss_01)
		if (print) {
			print("Samples"); print(y); print("BFSS01 for each MTS"); print(bfss_01)
			print("The value that maximize the BFSS 01 is"); print(y_max)
			print("The maximal value of the BFSS 01 is"); print(bfss_max)
		}
		else { bfss = c(); bfss$y_max= y_max; bfss$bfss_max = bfss_max; return(bfss) }
	}
\end{lstlisting}
\newpage
Upon implementing the aforementioned algorithm, leveraging a vector of samples drawn from the Poisson distribution, we acquired the empirical supremum of the Bayes Factor. Notably, this empirical outcome corresponded precisely with the theoretical supremum. This concordance is attributable to the inclusion of the value '0' within the sample set, as is evident in the code output below
\\ 
\textbf{R Program}
\begin{lstlisting}
	> bfss01_mts(y)
	[1] "Samples"
	[1] 3 1 2 0 5 0 2 1 0 2
	[1] "BFSS01 for each MTS"
	[1] 0.085971746 0.752252778 0.300901111 1.128379167 0.003473606 
	[1] 1.128379167 0.300901111 0.752252778 1.128379167 0.300901111
	[1] "The value that maximizes the BFSS 01 is"
	[1] 0
	[1] "The maximal value of BFSS 01 is"
	[1] 1.128379
	[1] "The theoretical sup of BFSS 01 is"
	[1] 2/sqrt(pi)
	[1] 1.128379
	
\end{lstlisting}

\hspace{5pt} A valuable approach involves comparing the theoretical supremum with the empirical supremum as the sample size (n) expands. This can be accomplished utilizing specific computational methodologies, particularly when the underlying data are derived from Poisson or Geometric distributions. The implementation of this comparison can be conducted via the following code
\\ 
\textbf{R Program}
\begin{lstlisting}
	y_pois = rpois(n=20,lambda=0.5); y_geom = rgeom(n=20,prob=0.8)
	bfss10_comparison = function(y,c_theoretical = 2/sqrt(pi)) { 
		bfss10_empirical = c(); bfss10_theoretical = c(); s = 0
		for (i in 1:n) { 
			s = s + y[i]; y_fact = factorial(y[1:i]); 
			m1 = (i^(s + 0.5)*factorial(i-1)*prod(y_fact)); m0 = (gamma(i + s + 0.5))
			bfss01= bfss01_mts(y[1:i],print=FALSE); c_empirical = bfss01$bfss_max
			bfss10_empirical = c(bfss10_empirical,c_empirical*m1/m0); bfss10_theoretical= c(bfss10_theoretical,c_theoretical*m1/m0)
		}
		plot(1:n,bfss10_empirical,type="l",xlab="Number of samples",ylab="BFSS 10 Empirical")
		plot(1:n,bfss10_theoretical,type="l",xlab="Number of samples",ylab="BFSS 10 Theoretical")
	}
\end{lstlisting}
\newpage
The following charts were generated using the code above
\begin{figure}[H]
	\centering
	\subfigure(a){\includegraphics[width=1\textwidth]{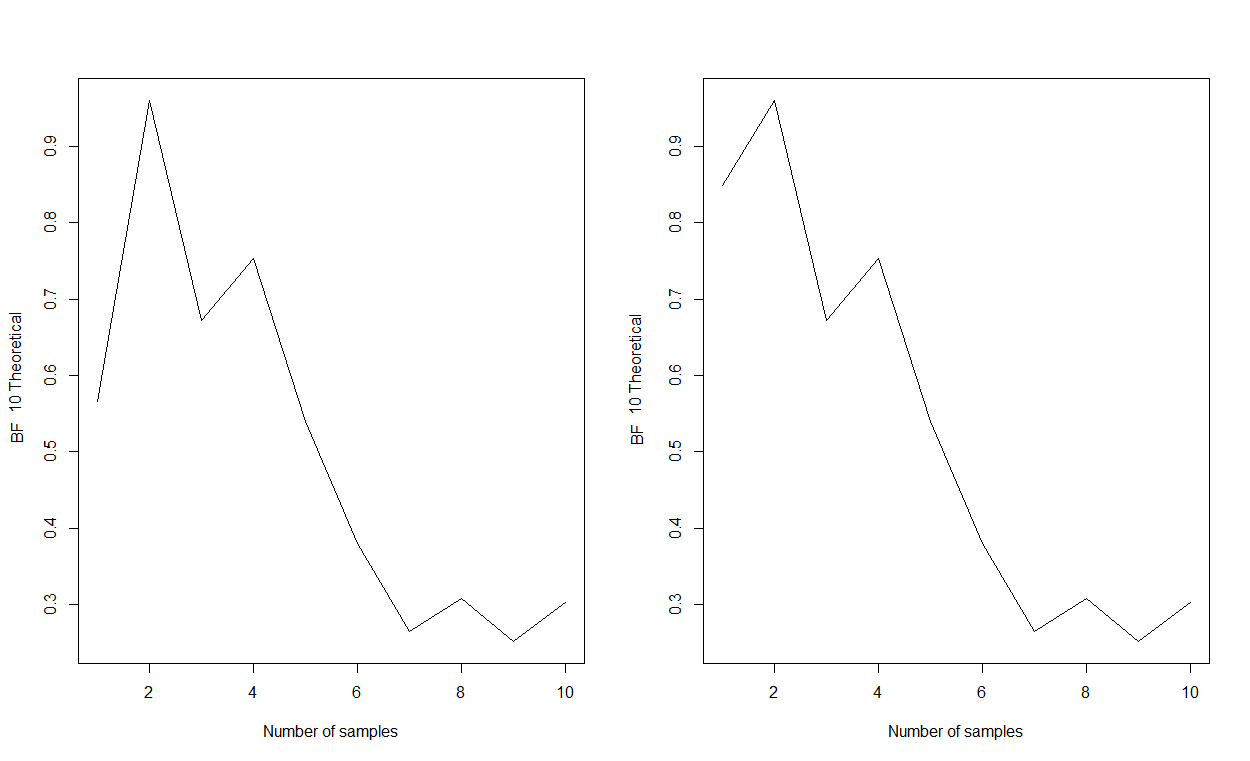}} 
	\subfigure(b){\includegraphics[width=1\textwidth]{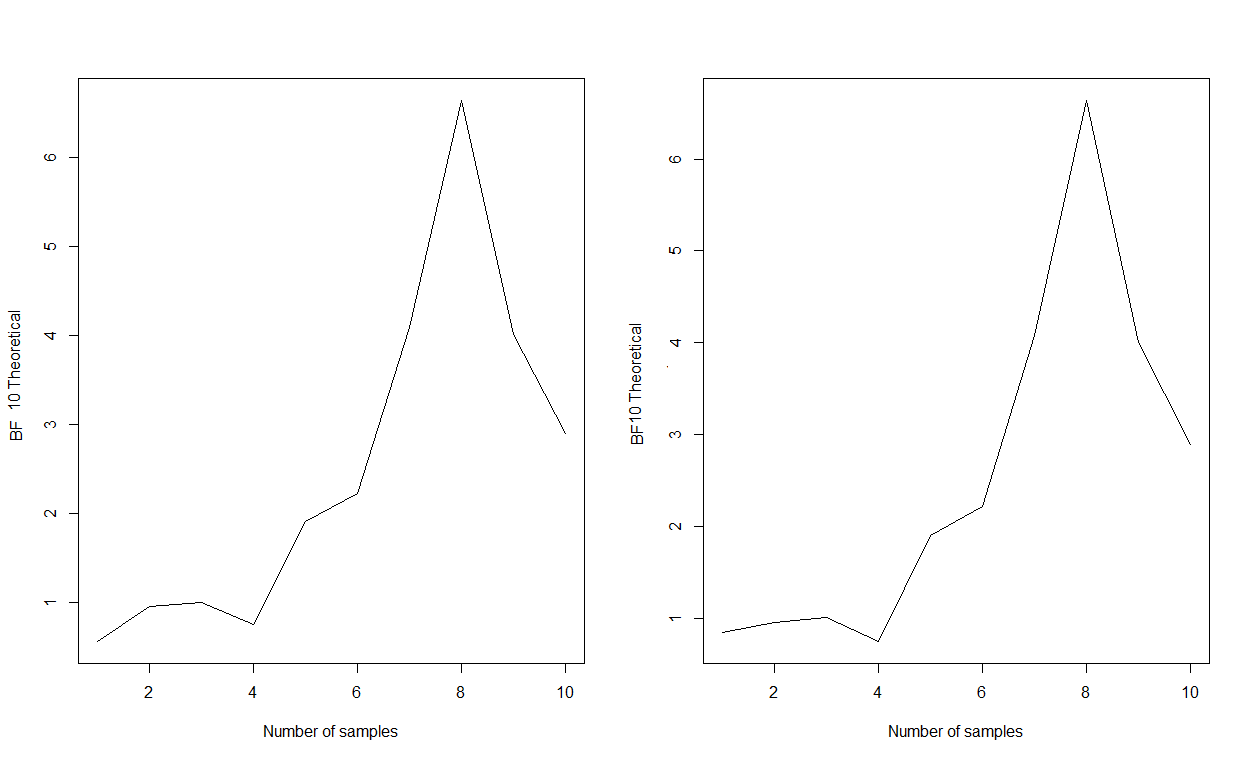}} 
	\caption{(a) Empirical and Theoretical $\overline{B}^{I}_{10}$ when $H_0$ is true (b) Empirical and Theoretical $\overline{B}^{I}_{01}$ when $H_0$ is false}
	\label{fig:separatedssbf}
\end{figure}
\newpage
Our methodological approach involves the generation of one hundred simulations, each encompassing thirty samples. These simulations will be conducted under both the null hypothesis and the alternative hypothesis. For each sample size, we aim to compute the average Bayes factors derived from all the simulations. Further, our study entails the graphical representation of these averaged Bayes factors against their respective sample sizes. Both the averages under the null and the alternative hypotheses will be depicted within the same coordinate system, facilitating a clear comparison and analysis. To automate the generation of such illustrative graphics, we utilize \textbf{R} programming language. The corresponding script to generate these charts is
\newpage
\textbf{R Program}
\begin{lstlisting}[language=R]
	Geometric_Poisson_SSBF = function(n,prob=0.5,lambda=1,simulations=100) {
		geometric_data = c(); poisson_data = c(); bfss01_geometric = c(); bfss01_poisson = c(); bf01_geo_total = c(); bf01_pois_total = c(); ibf01_geometric = c(); ibf01_poisson = c();ibf01_geo_total = c(); ibf01_pois_total = c()
		for (k in 1:simulations) {
			for (i in 1:n) {
				geometric_data = c(geometric_data,rgeom(n=1,prob=prob)); s = sum(geometric_data); y_fact = factorial(geometric_data)
				m1 = (i^(s + 0.5)*factorial(i-1)*prod(y_fact)*2); m0 = (gamma(i + s + 0.5)*sqrt(pi))
				m1_ibf = (i^(s + 0.5)*factorial(i-1)*prod(y_fact)); m0_ibf = (gamma(i + s + 0.5))      
				ibf10_correction = mean(gamma(geometric_data + 1.5))
				bfss01_geometric = c(bfss01_geometric,m0/m1); ibf01_geometric = c(ibf01_geometric,(m0_ibf/m1_ibf)*ibf10_correction)
				poisson_data = c(poisson_data,rpois(n=1,lambda=lambda))
				s = sum(poisson_data); y_fact = factorial(poisson_data)
				m1 = (i^(s + 0.5)*factorial(i-1)*prod(y_fact)*2); s = sum(log(y_fact)); m0 = (gamma(i + s + 0.5)*sqrt(pi)); m1_ibf = (i^(s + 0.5)*factorial(i-1)*prod(y_fact)); m0_ibf = (gamma(i + s + 0.5))
				ibf10_correction = mean(gamma(poisson_data + 1.5))
				bfss01_poisson = c(bfss01_poisson,m0/m1); ibf01_poisson = c(ibf01_poisson,(m0_ibf/m1_ibf)*ibf10_correction)
			}
			bf01_geo_total = cbind(bf01_geo_total,bfss01_geometric); bf01_pois_total = cbind(bf01_pois_total,bfss01_poisson); ibf01_pois_total = cbind(ibf01_pois_total,ibf01_poisson); ibf01_geo_total = cbind(ibf01_geo_total,ibf01_geometric)
		}
		bf01_geo_avg = c(); bf01_pois_avg = c(); ibf01_geo_avg = c(); ibf01_pois_avg = c(); n = nrow(bf01_geo_total)
		for (i in 1:n) {
			bf01_geo_avg = c(bf01_geo_avg,mean(bf01_geo_total[i,])); bf01_pois_avg = c(bf01_pois_avg,mean(bf01_pois_total[i,])); ibf01_pois_avg = c(ibf01_pois_avg,mean(ibf01_pois_total[i,])); ibf01_geo_avg = c(ibf01_geo_avg,mean(ibf01_geo_total[i,]));
		}
		par(mfrow=c(1,2))
		plot(1:n,bf01_pois_avg,main="BFSS01 (Black) vs IBF10 (Red) with Poisson samples ",xlab="Number of samples",ylab="BFSS01",type="l",ylim = c(min(bf01_pois_avg,ibf01_pois_avg),max(bf01_pois_avg,ibf01_pois_avg)))
		mtext("Simulations: 100 / Samples: 30 / H0: TRUE",side=3)
		lines(1:n,ibf01_pois_avg, type="l",col="red")
		plot(1:n,bf01_geo_avg,main="BFSS01 (Black) vs IBF10 (Red) with Geometric samples",xlab="Number of samples",ylab="BFSS01",type="l", ylim=c(min(bf01_geo_avg,ibf01_geo_avg),max(bf01_geo_avg,ibf01_geo_avg)))
		mtext("Simulations: 100 / Samples: 30 / H0: FALSE",side=3)
		lines(1:n,ibf01_geo_avg, type="l",col="red")
	}
	Geometric_Poisson_SSBF(n=30,prob=0.5,lambda=1,simulations=100)
\end{lstlisting}
\begin{figure}[H]
	\centering
	\includegraphics[width=1\textwidth]{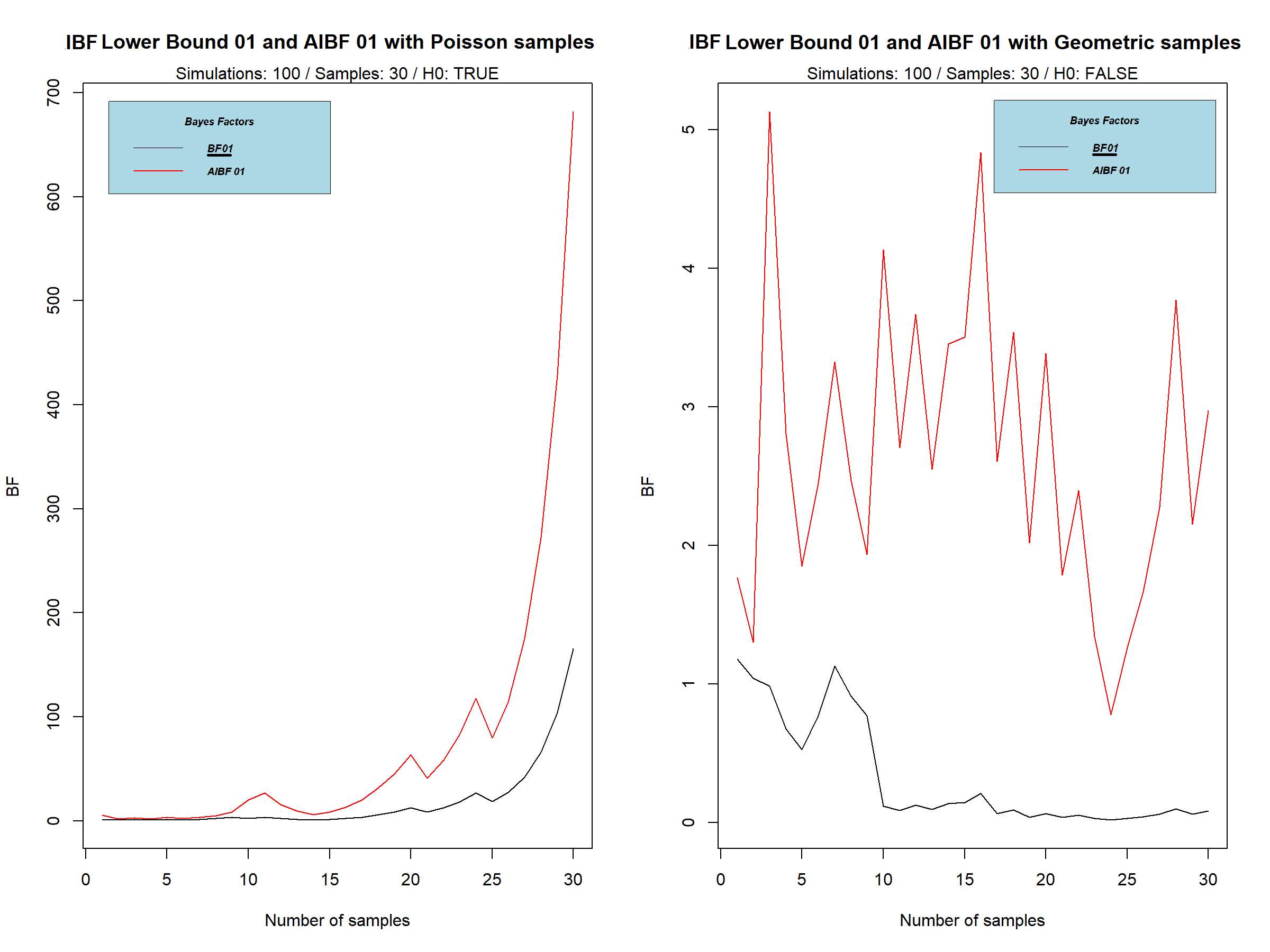}
	\caption{$AIBF_{10}$ and The $IBF_{10}$ Lower Bound Comparison for the Poisson vs Geometric Test}
	\label{fig:aibf_vs_sss_geometric_poisson}
\end{figure}
Examining the presented figure above elucidates several significant observations. Most prominently, it is discernible that $\underline{B}_{01}^{I}$ is a lower bound of the Average Intrinsic Bayes Factor (AIBF). Further, we note that the IBF lower bound is a more stable and robust form of evidence compared to the Arithmetic Intrinsic Bayes Factor (AIBF). In the scenario where data is generated under the alternative hypothesis, the AIBF presents lower values. However, these values do not furnish sufficient evidence to confidently reject the null hypothesis, $H_0$. On the other hand, the lower bound presents a more compelling case. It succeeds in providing an accurate decision - one favoring the rejection of the null hypothesis - when we deal with a sample size exceeding ten. This finding underscores the valuable role that the IBF lower bound can play in hypothesis testing, particularly when dealing with larger sample sizes.
\newpage
Moving forward, we now turn our attention to an empirical data-set that was previously studied by Cox (1962, p. 414) \cite{cox1962results}. This data-set comprises a sample of 30 observations, which were drawn from a Poisson distribution. Notably, this Poisson distribution is characterized by a mean value of 0.8. We aim to delve into a comprehensive analysis of this data-set within the purview of our ongoing investigation.
\begin{figure}[H]
	\centering
	\includegraphics[width=1\textwidth]{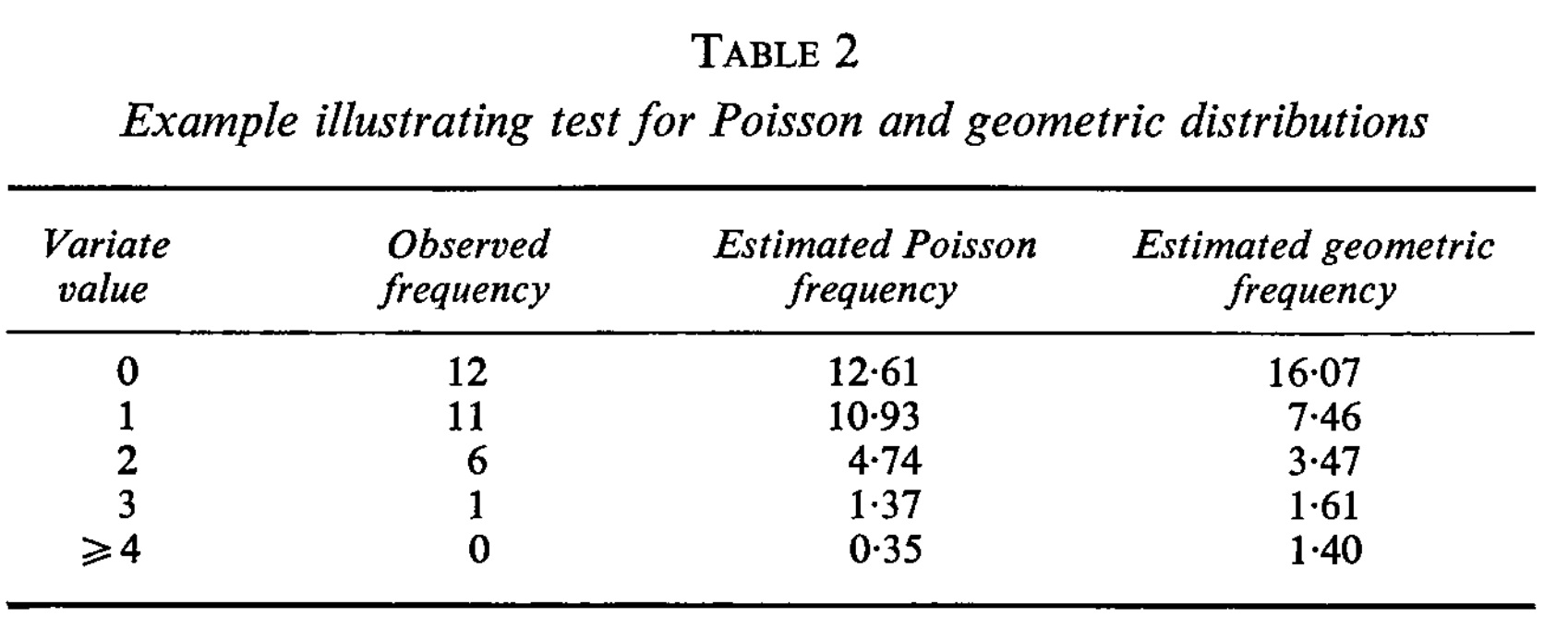}
	\caption{Cox (1962) p. 414}
	\label{fig:separatedssbf}
\end{figure}
The data utilized for this study was managed within the \textbf{R} programming environment. To ensure the robustness of our findings and to minimize the risk of any inherent bias or patterns influencing the results, we employed a randomization process on the data. This was achieved using the \emph{sample()} function, a commonly used \textbf{R} command for random sampling. Subsequent to this randomization, we proceeded to compute $\overline{B}_{10}^{I}$, specifically for the comparison between Poisson and Geometric distributions. The detailed steps followed in this computation are described next.\\ \\
\newpage
\textbf{R program}
\begin{lstlisting}[language=R]
	bfss10_comparison = function(y,c_theoretical = 2/sqrt(pi),type="empirical",chart=FALSE) { 
		bfss10_empirical = c(); 
		bfss10_theoretical = c(); 
		s = 0; 
		n = length(y)
		for (i in 1:n) { 
			s = s + y[i]
			y_fact = factorial(y[1:i])
			m1 = (i^(s + 0.5)*factorial(i-1)*prod(y_fact))
			m0 = (gamma(i + s + 0.5))
			bfss01= bfss01_mts(y[1:i],print=FALSE)
			c_empirical = bfss01$bfss_max
			bfss10_empirical = c(bfss10_empirical,c_empirical*m1/m0)
			bfss10_theoretical= c(bfss10_theoretical,c_theoretical*m1/m0)
		}
		if (chart) { 
			plot(1:n,bfss10_empirical,type="l",
			xlab="Number of samples",ylab="BFSS 10",
			main="Poisson data with parameter 0.8 (Cox 1962) Table 2")
		}
		else {
			if (type=="empirical") { return(bfss10_empirical) } 
			else if (type=="theoretical") { return(bfss10_theoretical) } 
		}
	}
	data = c(rep(0,12),rep(1,11),rep(2,6),3)
	data = sample(data)
	bfss10_comparison(y=data,type="theoretical",chart=TRUE)
\end{lstlisting}
\newpage
Using the code above we obtained the following chart:
\begin{figure}[H]
	\centering
	\includegraphics[width=.7\textwidth]{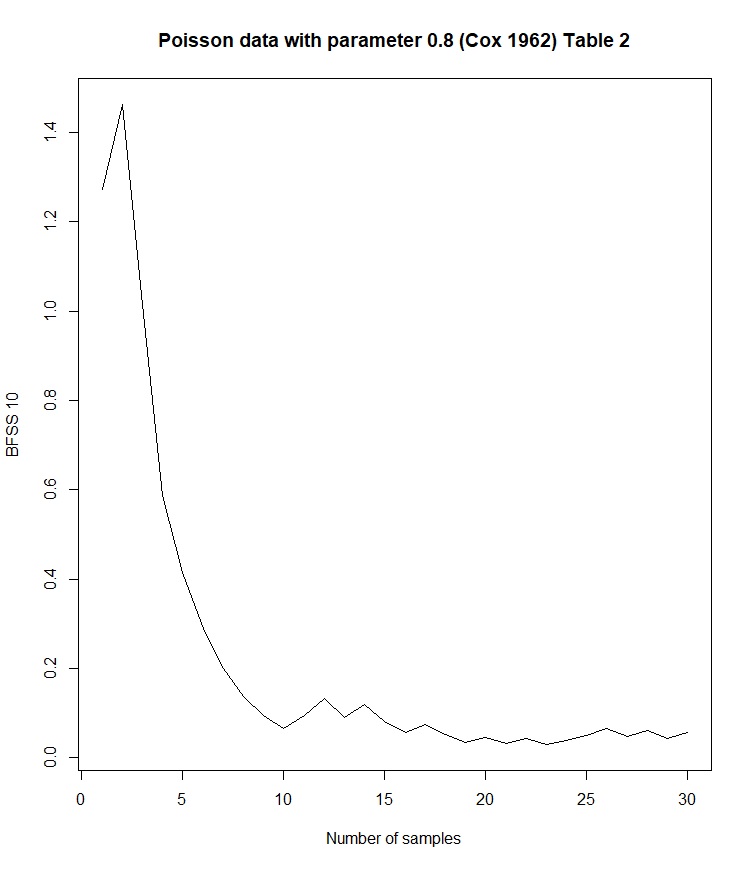}
	\caption{We used the randomized Cox data considered a growing vector with those samples and calculated the Bayes Factor as we added more data.}
	\label{fig:separatedssbf}
\end{figure}
It is evident from the chart above that the IBF bound is consistent with the results obtained by Cox (1962)\cite{cox1962results} and the IBF bound converges to the correct decision in favor of the adequate model which is the Poisson distribution. The graphical representation depicted above yields insightful observations about the performance of the IBF bound. Moreover, as sample sizes increase, the IBF bound shows convergence toward an optimal decision. It correctly leans in favor of the model that accurately reflects the underlying distribution of the data in this case, the Poisson. This empirical evidence reinforces the strength and consistency of the IBF as a tool for model selection.
\newpage
\section{The Negative Binomial vs. Poisson Separated Hypothesis Test}
In this section, we'll focus on a separated hypothesis test comparing the Negative Binomial and Poisson distributions. The Negative Binomial distribution is a discrete probability distribution that models the number of failures in a sequence of independent and identically distributed Bernoulli trials, before a specified (non-random) number of successes (denoted r) occurs. An alternative formulation is to model the number of total trials (instead of the number of failures). In fact, for a specified (non-random) number of successes (r), the number of failures (n - r) is random because the total trials (n) are random. For example, we could use the negative binomial distribution to model the number of days n (random) a certain machine works (specified by r) before it breaks down. Imagine a sequence of independent Bernoulli trials: each trial has two potential outcomes called "success" and "failure". In each trial, the probability of success is $\theta$ and of failure is $(1-\theta)$. We observe this sequence until a predefined number r of successes occurs. Then the random number of observed failures, 
Y follows the negative binomial distribution: $Y \sim NB(r,\theta)$. The pmf of Y is given by $$P(Y = y|r,\theta) = {y+r-1 \choose r-1}\theta^{r}(1-\theta)^{y} $$ 
where r is the number of successes, y is the number of failures, and $\theta$ is the probability of success on each trial. The likelihood of the Negative Binomial is given by
$$
\displaystyle\prod_{i=1}^{n} {y_i + r -1 \choose r-1}\theta^{rn}(1-\theta)^{\sum y_i}.
$$
The maximum likelihood estimator for $\theta$, the Fisher Information Matrix and the Jeffreys's prior for the Negative Binomial are:
$$ \hat{\theta} = \frac{r}{\bar{y} + r}, \hspace{5pt} I(\theta) = \frac{r}{\theta(1-\theta)^2}, \hspace{5pt} \sqrt{|I(\theta)|} \propto \frac{r^{1/2} }{\theta^{1/2}(1-\theta)} $$
The marginal distribution in this case is:
\begin{align*}
	m_1(\textbf{y}) &\propto \displaystyle\int_{0}^{1}\Bigg[\prod_{i=1}^{n}{y_i+r-1 \choose r-1}\Bigg]\theta^{rn}(1-\theta)^{\sum y_i}\frac{r^{1/2}}{\theta^{1/2}(1-\theta)}d\theta\\
	&= \displaystyle\Bigg[\prod_{i=1}^{n}{y_i+r-1 \choose r-1}\Bigg]r^{1/2}\int_{0}^{1}\theta^{rn + 1 - 1/2 - 1}(1-\theta)^{\sum y_i - 1}d\theta
\end{align*}
The change of variable $\beta = n\bar{y},\alpha =rn + 1/2$ yields: 
$$\int_{0}^{1}\theta^{\alpha - 1}(1-\theta)^{\beta - 1}d\theta = \frac{\Gamma(\alpha)\Gamma(\beta)}{\Gamma(\alpha + \beta)} $$
The marginal above becomes:
$$m_1(\textbf{y}) \propto  \displaystyle\Bigg[\prod_{i=1}^{n}{y_i+r-1 \choose r-1}\Bigg] r^{1/2}\frac{\Gamma(rn + 1/2)\Gamma(n\bar{y})}{\Gamma(n(r + \bar{y}) + 1/2)} $$
In this case, our null hypothesis is that $M_0$ is $Poisson(\lambda)$, thus 
$$ m_0(\textbf{y}) = \frac{c_0\Gamma(\sum y_i + \frac{1}{2})}{\displaystyle\prod_{i=1}^{n}y_i!}(1/n)^{(\sum y_i + \frac{1}{2})}$$
The Bayes factor can be expressed as:
$$ B_{10}(\textbf{y}) = \frac{c_1}{c_0}\displaystyle\Bigg[\prod_{i=1}^{n}{y_i+r-1 \choose r-1} y_i !\Bigg]r^{1/2}\frac{\Gamma(rn + 1/2)\Gamma(n\bar{y})}{\Gamma(n(r + \bar{y}) + 1/2)\Gamma(n\bar{y} + 1/2)} n^{(n\bar{y} + 1/2)}$$
By definition, we have that
$${y_i+r-1 \choose r-1} y_i ! = \frac{(y_i + r-1)!y_i!}{(r-1)!y_i!} = \frac{(y_i + r-1)!}{(r-1)!}$$
Hence, the Bayes factor can be expressed as:
$$\boxed{
	B_{10}(\textbf{y}) = \frac{c_1}{c_0}\displaystyle\Bigg[\prod_{i=1}^{n}\frac{(y_i + r-1)!}{(r-1)!}\Bigg]r^{1/2}\frac{\Gamma(rn + 1/2)\Gamma(n\bar{y})}{\Gamma(n(r + \bar{y}) + 1/2)\Gamma(n\bar{y} + 1/2)} n^{(n\bar{y} + 1/2)}
}
$$
Similarly, the Bayes factor for a minimal training sample $y(\ell)$ is:
$$
B_{10}(y(\ell)) = \frac{c_0}{c_1}\displaystyle\Bigg[\frac{(r-1)!}{(y(\ell) + r-1)!}\Bigg]r^{-1/2}\frac{\Gamma(r + y(\ell) + 1/2)\Gamma(y(\ell) + 1/2)}{\Gamma(r + 1/2)\Gamma(y(\ell))}
$$
The Intrinsic Bayes Factor in this case is:
$$
B_{10}(y(-\ell)|y(\ell)) = \frac{\displaystyle\big[\prod_{i=1}^{n}(y_i + r-1)!\big]n^{(n\bar{y} + 1/2)}\Gamma(r + y(\ell) + 1/2)\Gamma(y(\ell) + 1/2)\Gamma(rn + 1/2)\Gamma(n\bar{y})}{(y(\ell) + r-1)![(r-1)!]^{n-1}\Gamma(n(r + \bar{y}) + 1/2)\Gamma(n\bar{y} + 1/2)\Gamma(r + 1/2)\Gamma(y(\ell))}
$$
By using the identity $(y(\ell) + r - 1)! = \Gamma(y(\ell) + r)$, we found the Empirical SS Bayes factor upper bound for this test:
$$\overline{B}_{10}^{I^*}(\textbf{y}) = \frac{\displaystyle\big[\prod_{i=1}^{n}(y_i + r-1)!\big]n^{(n\bar{y} + 1/2)}\Gamma(rn + 1/2)\Gamma(n\bar{y})}{[(r-1)!]^{n-1}\Gamma(n(r + \bar{y}) + 1/2)\Gamma(n\bar{y} + 1/2)\Gamma(r + 1/2)}\displaystyle\sup_{\ell = 1,...,L}\frac{\Gamma(r + y(\ell) + 1/2)\Gamma(y(\ell) + 1/2)}{\Gamma(y(\ell) + r)\Gamma(y(\ell))}
$$
To obtain the supremum above, we must note that:
$${\Gamma(r + y(\ell) + 1/2)\Gamma(y(\ell) + 1/2)} \geq {\Gamma(y(\ell) + r)\Gamma(y(\ell))}, \hspace{10pt} \forall y(\ell) = 0,1,2,...
$$
Thus,
$$\displaystyle\sup_{\ell = 1,...,L}\frac{\Gamma(r + y(\ell) + 1/2)\Gamma(y(\ell) + 1/2)}{\Gamma(y(\ell) + r)\Gamma(y(\ell))} = \frac{\Gamma(r + y_{(n)} + 1/2)\Gamma(y_{(n)} + 1/2)}{\Gamma(y_{(n)} + r)\Gamma(y_{(n)})}$$
The result above, allowed us to express the Empirical IBF upper bound as:
$$\boxed{\overline{B}_{10}^{I^{*}}(\textbf{y}) = \frac{\displaystyle\big[\prod_{i=1}^{n}(y_i + r-1)!\big]n^{(n\bar{y} + 1/2)}\Gamma(rn + 1/2)\Gamma(n\bar{y})}{[(r-1)!]^{n-1}\Gamma(n(r + \bar{y}) + 1/2)\Gamma(n\bar{y} + 1/2)\Gamma(r + 1/2)}\frac{\Gamma(r + y_{(n)} + 1/2)\Gamma(y_{(n)} + 1/2)}{\Gamma(y_{(n)} + r)\Gamma(y_{(n)})}}
$$
Alternatively, we can express the Empirical IBF lower bound as:
$$\boxed{\underline{B}_{01}^{I^{*}}(\textbf{y}) = \frac{[(r-1)!]^{n-1}\Gamma(n(r + \bar{y}) + 1/2)\Gamma(n\bar{y} + 1/2)\Gamma(r + 1/2)}{\displaystyle\big[\prod_{i=1}^{n}(y_i + r-1)!\big]n^{(n\bar{y} + 1/2)}\Gamma(rn + 1/2)\Gamma(n\bar{y})}\frac{\Gamma(y_{(1)} + r)\Gamma(y_{(1)})}{\Gamma(r + y_{(1)} + 1/2)\Gamma(y_{(1)} + 1/2)}}
$$
Similarly, the Theoretical IBF Lower bound is:
$$\boxed{\underline{B}_{01}^{I}(\textbf{y}) = \frac{[(r-1)!]^{n-1}\Gamma(n(r + \bar{y}) + 1/2)\Gamma(n\bar{y} + 1/2)}{\displaystyle\big[\prod_{i=1}^{n}(y_i + r-1)!\big]n^{(n\bar{y} + 1/2)}\Gamma(rn + 1/2)\Gamma(n\bar{y})}\frac{\Gamma(r)}{\Gamma(1/2)}}
$$
\newpage
\begin{figure}
	\centering
	\subfigure{\includegraphics[width=0.90\textwidth]{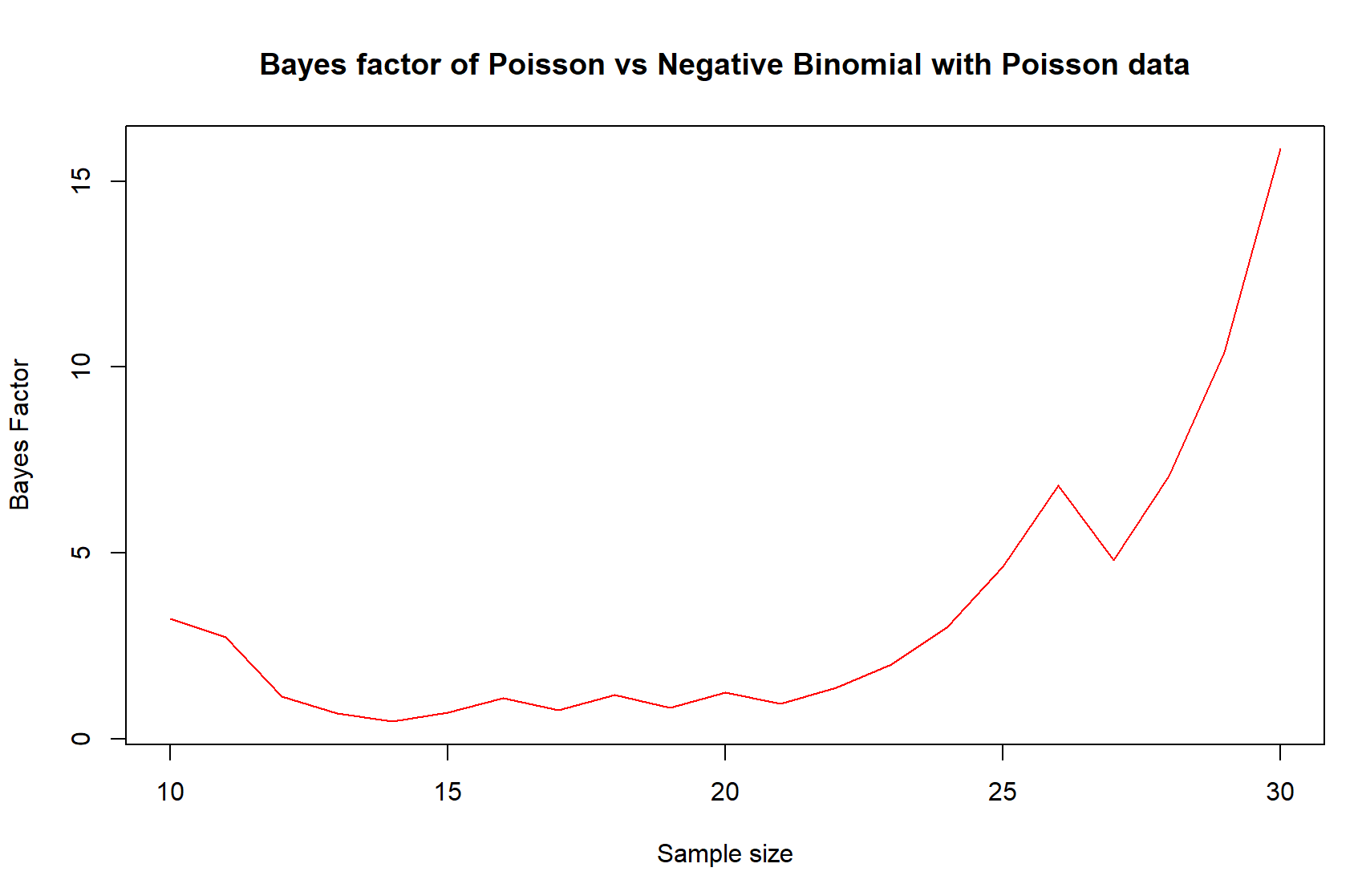}}
	\caption{Poisson vs. Negative Binomial IBF lower bound as the sample size increases with samples from $H_0$.}
	\label{fig:bfsspoissonnb}
\end{figure}

\begin{lstlisting}[language=R]
	bfss01 = function(y,m0="poisson",m1="nbinom",r=1) { 
		ybar = mean(y)
		n = length(y)
		lf1 = (n-1)*log(factorial(r-1)) + log(gamma( n*(r+ybar)+0.5 ) ) 
		+ log(gamma(n*ybar + 0.5)) + log(gamma(r+0.5)) + log(gamma(r)) 
		summ = 0
		for (i in 1:n) { summ = log(factorial(y[i] + r - 1)) + summ }
		lf2 = (n*ybar +0.5)*log(n) + log(gamma(r*n+0.5)) + log(gamma(n*ybar)) + summ
		exp(lf1)/exp(lf2)
	}
	
	n=30; y = c(); samples = c()
	for (i in 10:n) {
		samples = c(samples,rpois(n=1,lambda=1))
		y = c(y,bfss01(samples,r=1))
	}
	
	plot(10:n,y,type="l",col="red",xlab="Sample size",
	main="Bayes factor of Poisson vs Negative Binomial with Poisson data"
	,ylab="Bayes Factor")
\end{lstlisting}
\newpage
\section{Conclusions}
In conclusion, our exploration of separated hypothesis testing has provided valuable insights. Upon comprehensive analysis involving Bayes factor bounds, numerical simulations, and real-world data, our findings suggest a promising application of these bounds in general hypothesis testing. However, a crucial consideration lies in the careful selection of distributions for $M_0$ and $M_1$. It is imperative to choose these models in a manner that ensures $\inf B_{10}(\textbf{y}(\ell)) \neq 0$, or equivalently, $\sup B_{01}(\textbf{y}(\ell)) < \infty$. This selection criterion is pivotal as the effectiveness of the bound is contingent upon the appropriate identification of $M_0$ and $M_1$ and we used it in our examples, wheen we considered the Poisson as $M_0$ and the Geometric or Negative Binomial as $M_1$. Without this careful choice, the bound's utility in hypothesis testing will be affected.
\chapter{Ghosh and Samanta Alternative Approach}
This chapter explores a crucial 2001 study by J.K. Ghosh and T. Samanta \cite{berger2001objective}, presenting two methodologies that, when specific conditions are met, align with the SS approach. One such method involves setting $y_0 = \arg\sup_{y(\ell)}m_0(y(\ell))$ and subsequently resolving the equation $c B_{01}(y_0) = 1$. The Bayes factor that arises from this method is:
$$\boxed{B^{GS}_{10}(\textbf{y}) = cB_{10}(\textbf{y}) = B_{01}(y_0)B_{10}(\textbf{y}).}$$
\section{The Normal Mean Hypothesis Test Example}
\hspace{10pt}Consider the following simple hypothesis test
$$ H_0: N(0,1) \mbox{ vs. } H_1: N(\mu,1) $$
The marginal distributions for $H_0$ and $H_1$ are:
$$ m_0(\textbf{y}) = (\frac{1}{\sqrt{2\pi}})^{n} \exp\{-\frac{\sum y_i^2}{2} \} \hspace{ 10pt} m_1(\textbf{y}) = 1$$
To calculate the GS Bayes Factor, we calculated the following:
$$ y_0 = \arg\sup_{y(\ell) \in D}m_0(y(\ell)) = \arg\sup_{y(\ell) \in D}(\frac{1}{\sqrt{2\pi}})\exp\{-\frac{y(\ell)^2}{2} \} = 0$$
Additionally, we must solve this equation:
$$c B_{10}(y_0)= 1 \Longrightarrow c = B_{01}(y_0) = \frac{1}{\sqrt{2\pi}}$$
The Bayes Factor in this case is:
$$B_{10}(\textbf{y}) = (\sqrt{2\pi})^{n}\exp\{\frac{\sum y_i^2}{2}\}$$
Similarly, the GS Bayes Factor is:
$$\boxed{B^{GS}_{10}(\textbf{y}) = \frac{1}{\sqrt{2\pi}}(\sqrt{2\pi})^{n}\exp\{\frac{\sum y_i^2}{2}\} = (\sqrt{2\pi})^{(n-1)}\exp\{\frac{\sum y_i^2}{2}\}.}$$
In contrast, the IBF upper bound for this test is: 
$$\overline{B}^{I}_{10}(\textbf{y}) = B_{10}^{N}(\textbf{y})\frac{1}{\sqrt{2\pi}} $$
Therefore,
$$\overline{B}^{I}_{10}(\textbf{y}) = B^{GS}_{10}(\textbf{y}).$$
\section{The Exponential Hypothesis Test Example}
\hspace{10pt}Consider the following simple exponential hypothesis test  
$$ H_0: \lambda = \lambda_0 \mbox{ vs. } H_1: \lambda \neq \lambda_0 $$
In the exponential case, for a sample $\textbf{y}$ of size $n$, the marginal distribution is
\[
m(\textbf{x}) \propto \int_{0}^{\infty} {\lambda^{(n-1)}}\exp(-\displaystyle{\lambda\sum_{i=1}^{n}y_i})d\lambda = \displaystyle\frac{(\sum_{i = 1}^{n}y_i)^{n}}{\Gamma(n)}
\]
The marginal distributions for this hypothesis test are: 
$$ m_0(\textbf{y}) = {\lambda_0^{n}}{\exp(-\lambda_0\sum_{i=1}^{n}y_i)} \hspace{10pt} m_1(\textbf{y}) = \frac{(\sum_{i=1}^{n}y_i)^n}{\Gamma(n)} $$
To calculate we GS Bayes factor, we found that:
$$ y_0 = \arg\sup_{y(\ell) \in D}m_0(y(\ell)) = \arg\sup_{y(\ell) \in D}{\lambda_0}\exp(-\lambda_0y(\ell)) = 0$$
The next step is to solve the following equation:
$$
c B_{10}(y_0)= 1 \Longrightarrow c = B_{01}(y_0) = \lambda_0 \exp(-\lambda_0 y_0)\frac{1}{y_0} = \infty
$$
Evidently, the GS Bayes Factor does not exists for this case.
\subsection{Conclusions}
\hspace{10pt} The resemblance between the Ghost and Samanta (GS) approach and the IBF upper bound is evident. It is clear that if the marginal distribution $m_1(y(\ell))$ is independent of $y(\ell)$, then the two methods yield identical results. In other words, $B^{GS}_{10} = \overline{B}^{I}_{10}$. This equality arises due to the following reasoning:
\begin{align*}
	B^{GS}_{10} &=  B_{01}(y_0)B_{10}(\textbf{y}) \\
	&=  B_{01}(\arg\sup_{y(\ell) \in D}m_0(y(\ell)))B_{10}(\textbf{y}) \\
	&= B_{01}(\arg\sup_{y(\ell) \in D}[m_0(y(\ell))/m_1(y(\ell))])B_{10}(\textbf{y}) \\
	&= \sup_{y(\ell) \in D}B_{01}(y(\ell))B_{10}(\textbf{y}) = \overline{B}_{10}^{I}
\end{align*}
However, it is important to note that in certain cases, such as the exponential distribution, the IBF bound provides accurate results while the GS Bayes factor does not exist. Further research can be conducted to compare the suitability of each approach for different hypothesis tests in specific scenarios. Additionally, efforts can be made to identify general conditions that dictate the existence of the IBF bound and GS Bayes factors. Such investigations will contribute to a deeper understanding of the applicability and limitations of these alternative methodologies.
\chapter{Research Contributions}
\hspace{10pt}My doctoral thesis makes significant contributions to the field of Bayes factors theory. In this research, I introduce a new universal robust bound for Bayes factors that is an alternative to previous methods and is applicable to a variety of models. The bound is based on the SS Bayes Factor, which is a measure of evidence in favor of one hypothesis compared to another. This research significantly advances our understanding of Intrinsic Bayes Factors (IBF), presenting a noteworthy contribution to the field. Specifically, we have established stringent lower/upper bounds on all types of IBF - arithmetic, geometric, and median.

Our proposed bound surpasses the existing benchmark of $-ep\log(p)$ in two crucial aspects. Firstly, it shares identical asymptotic behavior as a Bayes Factor, adhering more closely to the foundational tenets of Bayesian analysis. Secondly, and more importantly, our lower bound corresponds to an actual prior distribution. This prior can be interpreted as a form of least favorable prior for the null hypothesis, embodying the spirit of conservatism intrinsic to statistical hypotheses testing.

In contrast to other IBFs, our derived variant exhibits computational efficiency. The ease of computation is particularly noteworthy given the generally arduous computational demands of other IBFs. Therefore, our lower bound provides not only a stronger mathematical formulation but also a more practical tool for researchers employing Bayesian methods. This approach bridged the gap between the Intrinsic Bayes Factors Theory with the SS Bayes Factors and compared the arithmetic Intrinsic Bayes Factor with the IBF bounds. Moreover, I developed the "Empirical IBF Bounds" methodology, resulting in a novel Empirical IBF lower bound that is even better than the theoretical bound.

To develop the Empirical Intrinsic Bayes Factor Bound, I constructed a mathematical formula and an algorithm to determine the set of all possible empirical minimal training samples and compared the empirical and theoretical Bayes Factor bounds. The thesis also contains a comparison of the lower bound with the Sellke et. al (1990) \cite{sellke2001calibration} $-ep\log(p)$ universal lower bound and the new Empirical and Theoretical SP Bayes Factors that are defined for the first time.

The thesis includes several theorems that demonstrate the effectiveness of these new methods. For example, \textbf{Theorem 3.1}  shows that the IBF lower bound gets better than the $-ep\log(p)$ universal lower bound as the sample size increases under the null hypothesis, \textbf{Theorem 3.2} states and shows that using our novel Empirical SP Prior, we can obtain our Empirical IBF Bound, and  \textbf{Theorem 3.3} states and shows that our novel Empirical SP Prior and SP Bayes factor, converges to the Theoretical SP Prior and SP Bayes Factor. This work also includes several examples and simulations, such as the Normal and Exponential distribution, ANOVA, and Linear models. In section 3.5, I generalized the technique used to compare the AIBF with the IBF Bound, and I showed that the same can be said for any other kind of Bayes factor that uses any measure of central tendency.

I also derived and calculated various Bayes factors and priors for linear models and specifically for ANOVA. I developed the necessary calculations for the Theoretical and Empirical Generalized IBF Bounds in General linear models by introducing a generalized prior to obtain both empirical and theoretical IBF bounds for two nested linear models and ANOVA. Using this generalization, we derived the ANOVA Bayes factors for different priors such as the Full Jeffreys prior and the Modified Jeffrey's prior. Furthermore, through these developments, I was able to demonstrate that the Generalized SS Bayes factor for ANOVA does not exist for the Reference Prior.

In this work, we also calculated the EP-Priors and EP-Bayes factor for the exponential distribution and the hypothesis test $\lambda = \lambda_0 \mbox { vs } \lambda \neq \lambda_0$ and compared their results with the between the SP Priors and the SP Bayes factor through numerical simulations and charts. We found that the results are similar but the prior obtained in the SP-Prior approach does not depend on the samples, as in the EP-Prior and it's much easier to calculate while their difference is small. Additionally, for the exponential case, the GS Bayes Factor was calculated in section 4.2 and compared to the SS Bayes factor and we've concluded that the GS Bayes factor does not exist while the IBF Bound both the theoretical and the empirical are well-defined and consistent. In the Geometric vs Poisson Separated models example, we calculated the AIBF and we showed the usefulness consistency, and stability of our lower bound of the AIBF.

These contributions have the potential to improve the accuracy and efficiency of statistical modeling methods, making them more accessible to researchers and practitioners in various fields. Ultimately, this research provides a valuable contribution to the field of Bayes factors theory and statistical modeling, with the potential to advance scientific knowledge and improve decision-making in a variety of domains.

\bibliographystyle{chicago}
\bibliography{referencefiles/references} 		

\end{document}